\documentclass[dvips]{article}
\usepackage{amsmath,amssymb,amsfonts,amsthm,epsfig,latexsym,graphics,enumitem}
\usepackage[usenames,dvipsnames]{color}

  \def\KK{{\mathbb K}} 
 \def\NN{{\mathbb N}}  
 \def\RR{{\mathbb R}} \def\SS{{\mathbb S}} \def\TT{{\mathbb T}}
   
 \def\ZZ{{\mathbb Z}}

\def\Si{\Sigma}
\def\La{\Lambda}
\def\De{\Delta}
\def\Om{\Omega}
\def\Ga{\Gamma}

  \def\cG{{\cal G}}  
  \def\cH{{\cal H}}  
\def\cC{{\cal C}}   \def\cO{{\cal O}} \def\cU{{\cal U}}
   \def\cP{{\cal P}} \def\cV{{\cal V}}
  \def\cK{{\cal K}}  \def\cW{{\cal W}}
\def\cF{{\cal F}}  \def\cL{{\cal L}}

\def\pitchfork{\frown\hspace{ -.625em}\mid}

\newtheorem{theorem}{{Theorem}}[section]

\newtheorem{lemma}[theorem]{{Lemma}}

\newtheorem{fact}[theorem]{{Fact}}

\newtheorem{claim}[theorem]{{Claim}}

\newtheorem{theo}[theorem]{Theorem}

\newtheorem{clai}[theorem]{Claim}
\newtheorem{lemm}[theorem]{Lemma}
\newtheorem{coro}[theorem]{Corollary}

\newtheorem{prop}[theorem]{Proposition}

\newtheorem{ques}[theorem]{Question}

\newtheorem{definition}[theorem]{{Definition}}
\newtheorem{defi}[theorem]{{Definition}}
\newtheorem{defis}[theorem]{{Definitions}}

\theoremstyle{definition}

\newtheorem{rema}[theorem]{Remark}

\newtheorem{exam}[theorem]{Example}

\newenvironment{demo}[1][Proof]{\noindent {\bf #1~: }}{\hfill$\Box$\medskip}

\title{Building  Anosov flows on $3$-manifolds}
\author{F. Beguin, C. Bonatti and B. Yu}
\date{\today}

\begin{document}

\maketitle

\hfill \emph{Dedicated to the memory of Marco Brunella\footnote{Varese Italy 1964,
Rio de Janeiro Brazil 2012}}
\begin{abstract}
We prove a result allowing to build (transitive or non-transitive) Anosov flows on $3$-manifolds by gluing together filtrating neighborhoods of hyperbolic sets. We give several applications; for example:
\begin{enumerate}
  \item we build a $3$-manifold supporting both of a transitive Anosov vector field and a non-transitive Anosov vector field;
  \item for any $n$,  we build a $3$-manifold $M$ supporting at least $n$ pairwise different Anosov vector fields;
  \item we build transitive attractors with prescribed entrance foliation; in particular,  we construct some incoherent transitive attractors;
  \item we build a transitive Anosov vector field admitting infinitely many pairwise non-isotopic transverse tori.
\end{enumerate}
\end{abstract}

\section{Introduction}
\subsection{General setting and aim of this paper}

Anosov vector fields are the paradigm of hyperbolic chaotic dynamics. They are
non-singular vector fields on compact manifolds for which the whole manifold
is a hyperbolic set: if $X$ is an Anosov vector field, the tangent bundle
to the manifold $M$  splits in a direct sum $E^s\oplus \RR X\oplus E^u$,
where  $E^s$ and $E^u$ are continuous subbundles invariant under the flow of $X$, 
and the vectors in $E^s$ and $ E^u$ are respectively uniformly contracted and uniformly
expanded by the derivative of this flow.  As all hyperbolic dynamics, Anosov vector fields 
are structurally stable: any vector field $C^1$-close to $X$ is \emph{topologically equivalent}
(or \emph{orbit equivalent}) to $X$.  Therefore, there is a hope for a topological classification 
of Anosov vector fields up to topological equivalence,
 and many works started this kind of classification in dimension $3$
(see for instance \cite{Gh1,Gh2,Ba1,Ba2,Fe1,Fe2,BaFe}).

However, we are still very far from proposing a classification, even in dimension $3$. For instance we still do not know which
manifolds support Anosov vector fields, or how many Anosov vector fields may be carried by the same manifold.

The simplest examples of Anosov vector fields on three-dimensional manifolds are the suspension of an Anosov-Thom linear 
automorphism of the torus $T^2$ and the geodesic flow of a hyperbolic Riemann surface. These two (class of) examples share some strong rigidity properties: 
\begin{itemize}
\item Plante has proved that every Anosov vector field on a surface bundle over the circle is topological equivalent to the suspension of an Anosov-Thom automorphism (\cite{Pl}), 
\item Ghys has proved that every  Anosov vector field on a circle bundle is topological equivalent to the geodesic flow of a hyperbolic surface (\cite{Gh1}),
\item Ghys again has discovered that every Anosov vector field on a three-manifold whose stable and unstable foliations are $C^\infty$  is topologically equivalent to one 
 of these canonical examples. 
 \end{itemize}
However, several pathological examples of Anosov vector fields (on closed three-manifolds) have been built.  Let us mention Franks and Williams who have built a non-transitive 
Anosov vector field (\cite{FrWi}), and Bonatti and Langevin who have built an Anosov vector field admitting a closed transverse cross section (diffeomorphic to a torus) which does not cut every orbit. 

Both Franks-William's and Bonatti Langevin's examples were obtained by gluing filtrating neighbourhoods of hyperbolic sets along their boundaries. The  aim of the present paper is to make a general theory of this type of examples. We first describe some elementary bricks (called \emph{hyperbolic plugs}). Then we prove a quite general result allowing to build Anosov vector fields (on closed three-manifolds) by gluing together such elementary bricks. We also provide a simple criterion allowing to decide whether an Anosov flow built in this manner is topologically transitive or not. Finally we illustrate our ``construction game" by several examples.

\subsection{Hyperbolic plugs}

In order to present our main results, we need to state some definitions.
We call \emph{plug} any pair $(V,X)$ where $V$ is a compact $3$-manifold
with boundary, and $X$ is a vector field on $V$,
transverse to the boundary of $V$ (in particular, $X$ is assumed to be
non-singular on $\partial V$). Given such a plug $(V,X)$, we can
decompose $\partial V$ as the disjoint union of an
\emph{entrance boundary} $\partial^{in} V$ (the part of $\partial V$
where $X$ is pointing inwards) and an \emph{exit boundary}
$\partial^{out} V$ (the part of $\partial V$ where $X$ is pointing outwards).
 The plug $(V,X)$ will be called an \emph{attracting plug} if
$\partial^{out} V=\emptyset$, and a \emph{repelling plug} if
$\partial^{in} V=\emptyset$. The plug $(V,X)$ will be called
a \emph{hyperbolic plug} if
its maximal invariant set $\La=\bigcap_{t\in \RR}X^t(V)$ is nonsingular
and hyperbolic with $1$-dimensional stable and unstable bundles.

If $(V,X)$ is a hyperbolic plug, the stable lamination $W^s(\Lambda)$
(resp. the unstable lamination $W^u(\Lambda)$) of
the maximal invariant set $\La=\bigcap_{t\in \RR}X^t(V)$
intersects transversally the entrance boundary $\partial^{in} V$
(resp. the exit boundary $\partial^{out} V$) and is disjoint
from $\partial^{out} V$ (resp. $\partial^{in} V$).
By transversality, $\cL^s_X:=W^s(\La)\cap\partial^{in}V$
and $\cL^u_X:=W^u(\La)\cap\partial^{out}V$ are one-dimensional laminations;
we call them the \emph{entrance lamination} and the \emph{exit lamination}
of $V$. We will see that the laminations $\cL^s_X$ and $\cL^u_X$
satisfy the following properties (see Proposition~\ref{p.hyperbolicplug}):
\begin{itemize}
\item[(i)] They contains finitely many compact leaf.
\item[(ii)] Every half non-compact leaf is asymptotic to a compact leaf.
\item[(iii)] Each compact leaf  may be oriented  so that its holonomy
is a contraction.  This orientation will be called the \emph{contracting orientation}.
\end{itemize}
A lamination satisfying this three properties will be called a \emph{Morse-Smale lamination},
or shortly a \emph{MS-lamination}. If the lamination is indeed a foliation, we will speak on \emph{MS-foliation}.

If $(V,X)$ is a hyperbolic attracting plug, then $\cL^s_X$ is a foliation on $\partial^{in} V$. In particular,
every connected component of $\partial V=\partial^{in} V$ is a two-torus (or a Klein bottle if we allow $V$ to be non-orientable).

Let us first state an elementary result which is nevertheless a
fundamental tool for our ``construction game":

\begin{prop}\label{p.plug}
Let $(U,X)$  and $(V,Y)$ be two hyperbolic plugs. Let
$T^{out}$ be a union  of connected components of $\partial^{out}U$ and $T^{in}$ be a union of connected components of 
$\partial^{in}V$. Assume that there exists a diffeomorphism
$\varphi\colon T^{out}\to T^{in}$ so that $\varphi_*(\cL^u_X)$ is transverse
to the foliation $\cL^s_Y$.
Let  $Z$ be the vector field induced by $X$ and $Y$ on the
manifold $W:=U\sqcup V/\varphi$.
Then $(W, Z)$ is a hyperbolic plug\footnote{This entails that there is
a differentiable structure on $W$
(compatible with the differentiable structures of $U$ and $V$
by restriction), so that $Z$ is a differentiable vector field.}.
\end{prop}

A classical example of the use of Proposition~\ref{p.plug} is the
famous Franks-William's example of a non-transitive Anosov vector field.
It corresponds to the case where $(U,X)$ is an attracting plug, $(V,Y)$
is a repelling plug,
$T^{out} = \partial^{out}U$ and $T^{in}=\partial^{in}V$. In that case, $W$
is boundaryless and $Z$ is Anosov and non-transitive.

If $(V,X)$ is a hyperbolic plug such that $V$ is embedded in a closed
three-dimensional manifold $M$ and
$X$ is the restriction of an Anosov vector field $\overline X$ on $M$,
then the stable (resp. unstable)
lamination of the maximal invariant set of $V$ is embedded in the
stable (resp. unstable) foliation of the
Anosov vector field $\overline X$. This leads to some restrictions
on the entrance and exit laminations of $V$,
and motivates the following definition.

A lamination  $\cL$ on a compact surface $S$ is a  \emph{filling MS-lamination}
if it satisfies properties (i), (ii), (iii) above and
if  every connected $C$ of $S\setminus \cL$ is
``a strip whose width tends to $0$ at both ends": more precisely,
$C$ is simply connected, and the accessible boundary of $C$ consists
of two distinct non-compact leaves $L_1,L_2$ which are asymptotic
to each other at both ends.

Any filling MS-lamination can be embedded in a $\cC^{0,1}$ foliation.
As a consequence, a closed surface carrying a filling MS-lamination
is either a torus or a Klein bottle. We will prove  that the
entrance lamination $\cL^s$ of a hyperbolic plug
is a filling MS-lamination if and only if this is also the case for the exit
lamination $\cL^u$ (see lemma~\ref{l.simple}).
Therefore we will speak of \emph{hyperbolic plugs with filling
MS-laminations}. Theorem~\ref{t.embedding} shows that
every hyperbolic plug with filling MS-laminations can be embedded
in an Anosov vector field.

\begin{exam}
Consider a transitive Anosov vector field $X$ on a closed
three-dimensional manifold $M$,
and a finite collection $\alpha_1,\dots,\alpha_m,\beta_1,\dots \beta_n$
of (pairwise different) orbits of $Y$.
Let $Y$ be the vector field obtained from $X$ by performing some
DA bifurcations  in the stable direction on the orbits
$\alpha_1,\dots,\alpha_m$, and some DA bifurcations  in the
unstable direction on the orbits $\beta_1,\dots \beta_n$. Let
$U$ be the compact manifold with boundary obtained by cutting out
from $M$ some pairwise disjoint tubular neighborhoods of
$\alpha_1,\dots,\alpha_m,\beta_1,\dots \beta_n$ whose boundaries
are transverse to $Y$. Then, $(U,Y)$ is a hyperbolic plug with
filling MS-laminations. More precisely, on each connected component
of $\partial^{in} V$ (resp. $\partial^{out} V$) the entrance
lamination $\cL^s_X$ (resp. the exit lamination $\cL^u_X$) consists
in one or two Reeb components.
\end{exam}

We finish this section by stating an addendum to
Proposition~\ref{p.plug} which allows us to build more and more
complicated hyperbolic
plugs with filling MS-laminations.  Two filling MS-laminations $\cL_1$ and
$\cL_2$ are called \emph{strongly transverse} if they are
transverse and if every connected component $C$
of $S\setminus(\cL_1\cup\cL_2)$  is a topological disc whose
boundary $\partial C$
consists in exactly four segments $a_1,a_2, b_1,b_2$ where $a_1,b_1$ lies
on  leaves of $\cL_1$ and $a_2,b_2$ lies on  leaves of $\cL_2$.

\begin{prop}\label{p.plug2}
In Proposition~\ref{p.plug}, assume furthermore that the plugs $(U,X)$
and $(V,Y)$ have filling MS-laminations and that  $\varphi_*(\cL^u_X)$ is
strongly transverse to $\cL^s_Y$.
Then the plug $(W,Z)$ has filling MS-laminations.
\end{prop}

\subsection{Building transitive Anosov flows}

In this section, we consider a hyperbolic plug with filling
MS-laminations $(U,X)$, and a diffeomorphism
$\varphi\colon \partial^{out} U\to \partial^{in} U$ so that the
laminations $\varphi_*(\cL^s_X)$ and $\cL^u_X$
are strongly transverse: we say that $\varphi$ is
\emph{a strongly transverse gluing diffeomorphism}. We denote by $Z$ the
vector field induced by $X$ on the closed manifold $U/\varphi$.
Note that there is a differentiable structure on $U/\varphi$
such that $Z$ is a differentiable vector field.

In general the vector field $Z$ is not hyperbolic. The dynamics of
$Z$ depends on the gluing diffeomorphism $\varphi$.
Two strongly transverse gluing diffeomorphisms $\varphi_0$ and
$\varphi_1$ may lead to vector fields $Z_0$ and $Z_1$ which are not
topologically equivalent, even if they are isotopic through
strongly transverse gluing diffeomorphisms. It is therefore necessary
to choose carefully the gluing diffeomorphism $\varphi$.

\begin{ques}
Given a hyperbolic plug  $(U,X)$ with filling MS-laminations and
a strongly transverse gluing diffeomorphism
$\varphi_0:\partial^{out} U\to \partial^{in} U$.
Does there exist a strongly transverse gluing diffeomorphism
$\varphi_1$, isotopic to $\varphi_0$ through strongly
transverse gluing diffeomorphisms, so that the vector field $Z_1$
induced by $X$ on  $U/\varphi_1$ is Anosov?
\end{ques}

We are not able to answer this question in general. Nevertheless we can
give a positive answer in the particular case where the
maximal invariant set $\La$ of $X$ in $U$ admits an affine Markov
partition (\emph{i.e.} a Markov partition so that the first return
map on the rectangles is affine for some coordinates). If  $\La$  does
not contain any attractor nor repeller, then a slight perturbation
of $X$ leads to a vector field $Y$, topologically equivalent to $X$,
whose maximal invariant set $\La_Y$ admits such affine Markov partition.
We will therefore allow us such perturbations.

Let $(U,Y)$ be another hyperbolic plug with filling MS-laminations,
and $\psi:\partial^{out} U\to\partial^{in} U$ be a
strongly transverse gluing diffeomorphisms. We say that
\emph{$(U,X,\varphi)$ and $(U,Y,\psi)$ are strongly isotopic}
if there is a continuous path $(U, X_t,\varphi_t)$ of hyperbolic plugs
with filling MS-laminations and strongly transverse
gluing diffeomorphisms so that $(U, X_0,\varphi_0)=(U,X,\varphi)$ and
$(U, X_1,\varphi_1)=(U,Y,\psi)$. Notice that this
implies that $(U,X)$ and $(U,Y)$ are topologically equivalent.

We will prove the following result:

\begin{theo}\label{t.transitive}
Let $(U,X)$ be a hyperbolic plug with filling MS-laminations so that the
maximal invariant set of $X$ contains neither attractors
nor repellers, and let $\varphi:\partial^{out}U\to \partial^{in} U$
be a strongly transverse gluing diffeomorphism. Then there
exist a hyperbolic plug $(U,Y)$ with filling MS-laminations and a
strongly transverse gluing diffeomorphism
$\psi:\partial^{out}U\to \partial^{in}U$ so that $(U,X,\varphi)$
and $(U,Y,\psi)$ are strongly isotopic,
and so that the vectorfield $Z$ induced by $Y$ on $U/\psi$  is Anosov.
\end{theo}

Theorem~\ref{t.transitive} allows to build an Anosov vector field $Z$
by gluing the entrance and the exit boundary of a hyperbolic plug.
We will now state a result providing a simple  criterion to decide
whether this Anosov vector field $Z$ is transitive or not.

Given a hyperbolic plug $(U,X)$ with filling MS-laminations and a
strongly transverse gluing map $\varphi:\partial^{out}U\to \partial^{in} U$,
consider the oriented graph $\cP$ defined as follows:
\begin{itemize}
\item the vertices of $\cP$ are the  basic pieces $\La_1,\dots, \La_k$
of $X$,
\item there is an edge going from $\La_i$ to $\La_j$ if $W^u_X(\La_i)$
 intersects  $W^s_X(\La_j)$, or $\varphi_*(W^u_X(\La_i)\cap \partial^{out} U)$
intersects $ W^u_X(\La_j)\cap \partial^{in} U$.
\end{itemize}
We say that $(U,X,\varphi)$ is \emph{combinatorially transitive}, if $\cP$
is \emph{strongly connected}, \emph{i.e.} if any two edges of $\cP$
can be joined by an oriented path.

\begin{prop}\label{p.transitive}
Under the hypotheses of Theorem~\ref{t.transitive}, if $(U,X,\varphi)$
is combinatorially transitive, then the Anosov vector field $Z$ is
transitive.
\end{prop}

In Theorem~\ref{t.transitive}, starting with a hyperbolic plug
without attractors nor repellers and a strongly transverse
gluing map $\varphi$, we build another vector field $Y$ orbit equivalent
to $X$ and a gluing map $\psi$ isotopic to $\varphi$
through strongly transverse gluing, so that the vector field induced by
the gluing is Anosov. It is natural to ask if
the resulting Anosov flow is independent of the choices
(of $Y$ and $\psi$), up to orbit equivalence.

\begin{ques} 
Let $(V,Y_1,\psi_1)$ and $(V,Y_2,\psi_2)$ be two hyperbolic plugs with filling MS-laminations endowed with strongly transverse gluing diffeomorphisms. 
Let  $Z_1$ and $Z_2$ be the vector fields induced by $Y_1$ and $Y_2$ on $V/\psi_1$ and $V/\psi_2$. Assume that $Z_1$ and $Z_2$ both are Anosov. 
Are $Z_1$ and $Z_2$ topologically equivalent ?
\end{ques}

\subsection{Examples}

To illustrate the power of Theorem~\ref{t.transitive}, we will use it
to build various types of examples of Anosov vector fields. We like to
think  hyperbolic plugs as elementary bricks of a
``construction game". Proposition~\ref{p.plug} and
Theorem~\ref{t.transitive} allow us to glue these elementary bricks
together in order to build more complicated hyperbolic plugs,
hyperbolic attractors, transitive or non-transitive Anosov vector fields.
We hope that the statements below will convince the reader of the
interest and of the versatility of this ``construction game".

\subsubsection{The ``blow-up, excise and glue surgery"}

As a first application of Theorem~\ref{t.transitive}, we shall prove
the following result:

\begin{theo}
\label{t.richer}
Given any transitive Anosov vector field $X$ on a closed (orientable) three-manifold
$M$, there exists a transitive Anosov vector field $Z$ on a closed
(orientable) three-manifold $N$ such that ``the dynamics of $Z$ is richer than the dynamics
of $X$". More precisely, there exists a compact set
$\Lambda\subset N$ invariant under the flow of $Z$, and a continuous onto
map $\pi:\Lambda\to M$ such that $\pi\circ X^t=Z^t\circ \pi$ for
every $t\in\RR$.
\end{theo}

The proof of Theorem~\ref{t.richer} relies on what we call the
\emph{blow-up, excise and glue surgery}. Let us briefly describe this
surgery (details will be given in section~\ref{s.surgery}). We start with a
transitive Anosov vector field $X$ on a closed three-manifold $M$.
\begin{itemize}
\item[]\textit{Step 1: blow-up.} We blow-up two periodic orbits of $X$.
More precisely, we pick two periodic orbits (with positive eigenvalues) $O,O'$ of $X$, we perform
an attracting DA (derived from Anosov) bifurcation on $O$, and a repelling
DA bifurcation on $O'$. This gives rise to a new vector field
$\overline{X}$ on $M$ which has three basic pieces: a saddle hyperbolic
set $\Lambda$, a attracting periodic orbit $O$,  and a
repelling periodic orbit~$O'$.
\item[]\textit{Step 2: excise.} Then we excise some two solid tori:
we consider the manifold with boundary $U=M\setminus (T\sqcup T')$,
where $T,T'$ are small tubular neighborhoods of $O,O'$. Under some
mild assumptions, $(U,\overline{X}_{|U})$ is a hyperbolic plug.
\item[]\textit{Step 3: Glue.} Finally, we glue the exit boundary on
the entrance boundary of $U$: we consider the manifold $N:=U/\varphi$
where $\varphi:\partial^{out} U\to\partial^{in} U$ is an appropriate gluing
 map, and the vector field $Z$ induced by $\overline X$ on
 $M$. Theorem~\ref{t.transitive} ensures that (up to perturbing $\overline{X}$ within its topological equivalence class) 
 $\varphi$ can be chosen so that $Z$ is Anosov, and Proposition~\ref{p.transitive} implies that $Z$ is  transitive. 
 The hyperbolic set
$\Lambda$ can be seen as a compact subset of $N$ which is invariant under
the flow of $Z$. Well-known facts on DA bifurcations show that there
exists a continuous map $\pi:\Lambda\to M$ inducing a semi-conjugacy
between the flow of $X$ and $Z$, as stated by Theorem~\ref{t.richer}.
\end{itemize}
Starting with a given transitive Anosov vector field on a closed
three-manifold, Theorem~\ref{t.richer} can be applied inductively,
giving birth to an infinite sequence of transitive Anosov vector fields
which are ``more and more complicated". Moreover, the ``blow-up, excise
and glue surgery" admits many variants, allowing to construct myriads
of examples of Anosov vector fields. For example, instead of starting
with a single Anosov vector field $X$, we could have started with
$n$ transitive Anosov vector fields $X_1,\dots,X_n$ in order to get a
single transitive Anosov vector field $Z$ which ``contains" the dynamics
of $X_1,\dots,X_n$. We could also have started with a non-transitive
Anosov vector field~$X$; in this case, we get an Anosov vector field
$Z$ which might or might not be transitive depending on the choice
of the periodic orbits we use for the DA bifurcations. As an application,
we will obtain the following result, which answers a question that
was asked to us by A. Katok:

\begin{theo}
\label{t.both-transitive-non-transitive}
There exists a closed orientable three-manifold supporting both a
transitive Anosov vector field and a non-transitive Ansov vector field.
\end{theo}

\subsubsection{Hyperbolic attractors}

As already mentioned above, the entrance foliation $\cL^s(U,X)$ of
an attracting hyperbolic plug $(U,X)$ is always a \emph{MS foliation}
(it has finitely many compact leaves, every half leaf is asymptotic to a
compact leaf, and every compact leaf can be oriented so that its holonomy
is a contraction). Using Theorem~\ref{t.transitive}, we shall prove
a converse statement: every MS-foliation can be realized as the
entrance foliation of a transitive attracting hyperbolic plug.
 More precisely:

\begin{theo}
\label{t.attractors}
For every MS-foliation $\cF$ on a closed orientable surface $S$, there exists
an orientable transitive attracting hyperbolic plug $(U,X)$ and a homeomorphism
$h:\partial^{in} U\to S$ such that $h_*(\cL^s(U,X))=\cF$.
 \end{theo}

Let $(U,X)$ be an attracting hyperbolic plug, and $\Lambda$ be the
maximal invariant set of $(U,X)$. Each compact leaf $\gamma$ of the
foliation $\cL^s(U,X)$ can be oriented in such a way that its holonomy
is a contraction; we call this the \emph{contracting orientation} of
$\gamma$. The attractor $\Lambda$ is said to be \emph{incoherent} if one
can find two compact leaves $\gamma_1,\gamma_2$ of $\cL^s(U,X)$ in the
same connected component of $\partial^{in} U$, such that
$\gamma_1,\gamma_2$ equipped with their contracting orientations are not
freely homotopic. The notion of incoherent hyperbolic attractors
was introduced by J. Christy in his PhD thesis (\cite{Christy-thesis}).
Christy studied the existence of Birkhoff sections for hyperbolic
attractors of vector fields on three-manifolds. He proved that
a transitive hyperbolic attractor admits a Birkhoff section if and only if
it is coherent.
He announced that he could build incoherent transitive
hyperbolic attractors. Actually, he did publish an example of
incoherent hyperbolic attractor (\cite{Ch}), but it is not clear whether
his example is transitive or not. The existence of  incoherent
transitive hyperbolic attractors is an immediate consequence
of Theorem~\ref{t.attractors}:

\begin{coro}
\label{c.incoherent-attractors}
There exists incoherent transitive hyperbolic attractors (on orientable manifolds).
\end{coro}

\subsubsection{Embedding hyperbolic plugs in Anosov flows}

Using Theorem~\ref{t.attractors}, we will be able to prove that
every hyperbolic plug with filling MS-laminations can be embedded in an
Anosov flow. More precisely:

\begin{theo}
\label{t.embedding}
Consider a hyperbolic plug with filling MS-laminations $(U_0,X_0)$.
\begin{itemize}
\item Up to changing $(U_0,X_0)$ by a topological equivalence, we can
find an Anosov vector field $X$ on a closed orientable three-manifold $M$ such
that there exists an embedding $\theta:U_0\hookrightarrow M$ with
$\theta_* X_0=X$.
\item Moreover, if the maximal invariant set of $(U_0,X_0)$
contains neither attractors nor repellers, the construction can
be done in such a way that the Anosov vector field  $X$ is transitive.
\end{itemize}
\end{theo}

In other words, the first item of Theorem~\ref{t.embedding} states that,
for every hyperbolic plug with filling MS-laminations $(U_0,X_0)$, we can
find a closed three-manifold $M$, an Anosov vector field $X$ on $M$,
and a closed submanifold with boundary $U$ of $M$, such that $X$
is transverse to $\partial U$ and such that $(U_0,X_0)$ is
topologically equivalent to $(U,X_{|U})$. The manifold $M$ and
the Anosov vector field $X$ will be constructed by gluing
appropriate attracting and repelling hyperbolic plugs on $(U_0,X_0)$.
These attracting and repelling hyperbolic plugs will be provided
by Theorem~\ref{t.attractors}.

In order to get the second item of Theorem~\ref{t.embedding}, we will
modify the Anosov vector field provided by the first
item, using the ``blow-up, excise and glue procedure" discussed above.

\subsubsection{Manifolds with several Anosov vector fields}

Our techniques also allow to construct examples of three-manifolds
supporting several Anosov flows. In~\cite{Ba4}, Barbot constructed
an infinite family of three-manifolds, each of which supports at least
two (non topologically equivalent) Anosov flows. We shall prove the
following result:

\begin{theo} \label{t.manyanosov}
 For any $n\geq 1$, there is a closed orientable three-manifolds
 $M$ supporting at least $n$ transitive Anosov vector fields
 $Z_1, \dots ,Z_n$ which are pairwise non topologically equivalent.
\end{theo}

\begin{rema}
The manifold $M$ that we will construct to prove
Theorem~\ref{t.manyanosov} has a JSJ decomposition with three pieces:
two hyperbolic pieces and one Seifert piece. These examples also positively answered
two questions asked by Barbot and Fenley in the final section of their recent paper  \cite{BaFe}\footnote{They asked the following questions. Does there exists 
examples of manifolds, with a JSJ decomposition containing more than one hyperbolic piece, supporting transitive Anosov or pseudo-Anosov flows~?  Does there exists 
examples of manifolds, with both hyperbolic and Seifert pieces in their JSJ decomposition, supporting transitive Anosov or pseudo-Anosov flows~?}.
The vector fields $Z_1,\dots,Z_n$ that we will construct are
pairwise homotopic in the space of non-singular vector fields on $M$.
\end{rema}

Theorem~\ref{t.manyanosov} admits many variants. For example, we claim
that the manifold $M$ that we will construct  also supports at least
$n$ non-transitive Anosov flows. We also claim that, for every
$n\geq 1$, there exists a graph manifold supporting at least
$n$ transitive Anosov vector fields. We will not prove those claims
to avoid increasing too much the length of the paper; we leave them
as exercises for the reader.

\subsubsection{Transitive Anosov vector fields with infinitely
many transverse tori}

Theorem~\ref{t.transitive} allows to build transitive Anosov vector fields
by gluing hyperbolic plug along their boundary. Conversely, one may
try to decompose a transitive Anosov vector field $X$ on a closed
three-manifold $M$ into hyperbolic plugs by cutting $M$ along a finite
family of pairwise disjoint tori that are transverse to $X$. Ideally,
one would like to find a canonical (maximal) finite family of
pairwise disjoint tori embedded in $M$ and transverse to $X$, so that,
by cutting $M$ along these tori, one gets a canonical (maximal)
decomposition of $(M,X)$ into hyperbolic plugs. This raises the
following question: given an Anosov vector field $X$ on a closed
three-manifold $M$, are there only finitely many  tori
(up to isotopy) embedded in $M$ and transverse to $X$~? Unfortunately,
we shall
prove that this is not the case:

\begin{theo}
\label{t.infinitely-many-tori}
There exists a transitive Anosov vector field $Z$ on a closed
orientable three-manifold $M$, such that there exist infinitely many pairwise
non-isotopic tori embedded in $M$ and transverse to $Z$.
\end{theo}

Roughly speaking, Theorem~\ref{t.infinitely-many-tori} implies that there
is no possibility of finding a ``fully canonical" maximal decomposition
of any transitive Anosov vector field into hyperbolic plugs. In a
forthcoming paper (\cite{BeBoYu}), we shall nevertheless prove that
one can find a maximal decomposition of any transitive Anosov vector
field $X$ into hyperbolic plugs, so that the maximal invariant sets of
the  plugs are canonically associated to $X$. This is what we call
the \emph{spectral-like decomposition} for transitive Anosov vector fields.

\subsubsection*{Homage} This work started in January 2012, as we were reading in a working group the beautiful paper \cite{Br} 
of Marco Brunella. It was during the same month that we learned the sad news of Marco's death. We dedicate this paper to his memory. 
\part{Proof of the gluing theorem~\ref{t.transitive}}
\section{Definitions and elementary properties}

In this paper we  consider non-singular vector fields on compact $3$-manifolds. 


\subsection{Hyperbolic plugs}

\begin{defis}
All along this paper, a \emph{plug} is a pair $(V,X)$ where $V$ is a compact three-dimensional manifold with boundary,
and $X$ is a non-singular $C^1$ vector field on $V$ transverse to the boundary $\partial V$.
Given such a plug $(V,X)$, we decompose $\partial V$ as a disjoint union
$$\partial V=\partial^{in} V\sqcup\partial^{out} V,$$
where $X$  points inward (resp. outward) $V$ along $\partial^{in} V$ (resp. $\partial^{out} V$).
We call $\partial^{in} V$ the \emph{entrance boundary} of $V$, and $\partial^{out} V$ the \emph{exit boundary} of $V$.
If $\partial^{out}V=\emptyset$, we say that $(V,X)$ is an \emph{attracting plug}. If $\partial^{in}V=\emptyset$, we say that $(V,X)$ is a \emph{repelling plug}.
\end{defis}

If $\partial V$ is non-empty, the flow of $X$ is not complete.  Every orbit of $X$ is defined for a closed time interval of $\RR$. The \emph{maximal invariant set} $\La$ of $X$ in $V$ is the set of points $x\in V$ whose forward and backward orbits are defined forever; in an equivalent way, $\La$ is the set of points whose orbit is disjoint from $\partial V$. The \emph{stable set} $W^s(\La)$ is the set of points whose forward orbit is defined forever. Equivalently, $W^s(\La)$ is the set of points whose forwards orbits accumulates on $\La$,
and is the set of points whose positive orbit is disjoint from $\partial^{out} V$.  Analogously the \emph{unstable set} $W^u(\Lambda)$ is the set of points
whose backward orbit is defined forever; this negative orbit is disjoint from $\partial^{in}V$. 

\begin{defi}
All along the paper, we call \emph{hyperbolic plug} a plug $(V,X)$ whose maximal invariant set $\La$ is hyperbolic with one-dimensional
strong stable and strong unstable bundles:  for $x\in \La$, there is a splitting
$$T_xV=E^s(x)\oplus \RR X(x)\oplus E^u(x)$$
which depends continuously on $x$, which is invariant under the derivative of the flow of $X$, and there is a Riemannian metric on $V$
so that the differential of the time one map of the flow of $X$ contracts uniformly the vectors in $E^s$ and expands uniformly the vectors in $E^u$.
\end{defi}

We cannot recall here the whole hyperbolic theory.  Let us just recall some elementary properties from the
classical theory of hyperbolic dynamical systems  that we will use in the other section:
\begin{itemize}
 \item for every $x$ the \emph{strong stable manifold} $W^{ss}(x)=\{y\in V, d(X^t(y),X^t(x)){\underset{t\to +\infty}\longrightarrow}  0\}$ is a
 $C^1$ curve through $x$ tangent at $x$ to $E^s(x)$.  The strong unstable manifold $W^{uu}(x)$ of $x$ is the strong stable manifold of $x$ for $-X$.
 \item the weak stable manifold $W^s(x)$ (resp. weak unstable manifold $W^u(x)$) of $x$ is the union of the stable manifolds (resp. unstable manifolds)
 of the points in the orbit of $x$. The weak stable and weak unstable manifold of $x\in \La$ are $C^1$ injectively immersed surface which depends continuously 
 on $x$.  The weak stable (resp. unstable) manifold of $x$ is invariant under the positive (resp.negative) flow, and by the negative
 (resp. positive)  flow for the times it is defined.
 \item $\La^+$ and $\La^-$ are a $2$-dimensional laminations $W^s(\La)$ and $W^u(\La)$, whose leaves are the weak stable and unstable (respectively)  manifolds of
 the points of $\La$. These laminations are of class $C^{0,1}$, \emph{i.e.} the leaves are $C^1$-immersed manifold tangent to a continuous plane field. Moreover, these 
 laminations are of class $C^1$, in the case where the vector field $X$ is of class $C^2$ (see e. g. \cite[Corollary 2.3.4]{Has}).
 \item The $2$-dimensional laminations  $W^s(\La)$ and $W^u(\La)$ are everywhere transverse and the intersection $W^s(\La)\cap W^u(\La)$ is
 precisely $\La$.
 \item The non-wandering set $\Om(X)$ is contained in $\La$.  It is the union of finitely many transitive hyperbolic sets called
 \emph{basic pieces of $X$}. Every point in $\La$ belongs to the intersection of the stable manifold of one basic piece with the unstable manifold of
 a basic piece.

\item there is a $C^1$-neighborhood $\cU$ of $X$ so that, for every $Y\in\cU$, the pair $(V,Y)$ is a hyperbolic plug, and
it is topologically (orbitally) equivalent to $X$: there is a homeomorphism $\varphi_Y\colon V\to V$
mapping the oriented orbits of $Y$ on the orbits of $X$.  Furthermore, for $Y$ $C^1$-close to  $X$, the homeomorphism $\varphi_Y$ can be
chosen $C^0$-close to the identity.
\end{itemize}

\subsection{Separatrices, free separatrices, boundary leaves}

We will now introduce some notions which are useful to describe more precisely the geometry of hyperbolic sets of vectors fields on three-manifolds. These notions were introduced  by the first two authors in~\cite{BeBo}. 

Let $(V,X)$ be a hyperbolic plug, and $\La$ be its maximal invariant set.  One can embed $V$ in a closed manifold $M$ and extend $X$ on $M$, so that
$\partial^{in} V$ is an exit boundary of a repelling region of $M$ and $\partial^{out} V$ is the entrance
boundary of in attracting region of $M$: now the flow of $X$ is complete and $\La$  a locally maximal hyperbolic set of $X$,
with $1$-dimensional strong stable and strong unstable bundles. Furthermore $V$ is now a filtrating set, that is
the intersection of a repelling and an attracting region of $X$. This allows us to use some notions/results that were defined/proved for complete flows on closed three-manifolds.

\begin{defis}
A \emph{stable separatrix} of a periodic orbit $O\subset\La$ is a connected component of $W^s(O)\setminus O$. A stable separatrix of a periodic orbit $O\subset\La$ is called a \emph{free separatrix} is it is disjoined from $\La$. 
\end{defis}

\begin{rema}
\label{r.stable-manifolds}
For $x\in\Lambda$, the weak stable manifold $W^s(x)$ is an injectively immersed  manifold.  
\begin{itemize}
 \item If $W^s(x)$ does not contain periodic orbit, then $W^s(x)$ is diffeomorphic to the plane $\RR^2$, and the foliation of $W^s(x)$ by the orbits of flow of $X$ is topologically equivalent to the trivial foliation of $\RR^2$ by horizontal lines.
\item If $W^s(x)$ does not contain periodic orbit $O$, then $O$ is unique, and every orbit of $X$ on $W^s(x)$ is asymptotic to $O$ in the future. If the multipliers\footnote{i.e. the eigenvalues of the derivative at $p$ of the first return map of the orbits of $X$ on a local section intersecting $O$ at a single point $p$} of $O$ are positive, then $W^s(x)$ is diffeomorphic to a cylinder, and $O$ has two stable separatrices. If the multipliers of $O$ are negative, then $W^s(x)$ is diffeomorphisc to a M\"obius band, and $O$ has only one stable separatrix. In any case, each stable separatrix of $O$ is diffeomorphic to a cylinder $S^1\times \RR$, and the foliation of this separatrix by the orbits of $X$ is topologically equivalent to the trivial foliation of  $S^1\times \RR$ by vertical lines.
\end{itemize}
\end{rema}

\begin{rema}
\label{r.free}
Let $x$ be a point in $\Lambda$.
\begin{enumerate}
\item Each orbit of $X$ on $W^s(x)$ cuts $\partial^{in} V$ in at most one point, and this point depends continuously on the orbit. Thus, the connected components of $W^s(x)\cap\partial^{in} V$ are in $1$-to-$1$ correspondence with the connected components of $W^s(x)\setminus \La$. 
\item Due to the dynamics inside each leaf $W^s(x)$; one easily show that, for each connected component $C$ of $W^s(x)\setminus \La$, one has one of the two possible situation below:
\begin{itemize}
 \item either there is a orbit $O\in\La\cap W^s(x)$ so that $C$ is a connected component of $W^s(x)\setminus O$. 
 In that case, \cite[Lemma 1.6]{BeBo} proves that $O$ is a periodic orbit, and $C$ is a free stable separatrix of $O$. In particular, $C$ is diffeomorphic to a cylinder 
 and the foliation of $C$ by the orbits of $X$ is topologically equivalent to the trivial foliation of  $S^1\times \RR$ by vertical lines. Since each orbit $X$ of $C$ cuts $\partial^{in} V$ at exactly one point which depends continuously on the orbit, one deduces than  $C\cap\partial^{in} V$ is diffeomorphic to a circle; 
  \item or there are two orbits $O_1,O_2\in \Lambda\cap W^s(x)$ so that $C$ is a connected component of
 $W^s(x)\setminus (O_1\cup O_2)$; in other words, $C$ is a strip bounded by $O_1$ and $O_1$ and the foliation of $C$ by the orbits of $X$ is topologically equivalent to the trivial foliation of  $\RR^2$ by horizontal lines. Since each orbit $X$ of $C$ cuts $\partial^{in} V$ at exactly one point which depends continuously on the orbit, one deduces than $C\cap\partial^{in} V$ is diffeomorphic to a line;
\end{itemize}
\end{enumerate}
\end{rema}

\begin{defi}
An unstable manifold $W^u(x)$ is called \emph{a unstable boundary leaf} if there is an open path $I$ cutting $W^u(x)$ transversely at a point $y$
and so that one connected component of $I\setminus \{y\}$ is disjoint from  $W^s(\La)$. 
\end{defi}

\begin{rema}
\label{r.free-2}
Lemma 1.6 of \cite{BeBo} shows that the unstable boundary leaves are precisely the unstable manifolds of the periodic orbits having a free stable separatrix. Moreover Lemma 1.6 of \cite{BeBo} shows that
there are only finitely many periodic orbits in $\La$ having a free stable separatrices. As an immediate consequence, there are only finitely many boundary leaves in $W^u(\La)$.
\end{rema}

One defines similarly free unstable separatrices and the stable boundary leaves.

\subsection{Entrance and exit laminations}

Let $(V,X)$ be a hyperbolic plug, $\La$ be its maximal invariant set, $W^s(\La)$ and $W^u(\La)$ be the $2$-dimensional
stable and unstable laminations of $\La$ respectively.

The vector field $X$ is tangent to laminations $W^s(\La)$ and $W^u(\La)$, and is transverse to $\partial V$.  This implies that
the laminations $W^s(\La)$ and $W^u(\La)$ are transverse to $\partial V$.  As a consequence, each leaf of these
$2$-dimensional laminations cuts $\partial V$ along $C^1$-curves, and the laminations
$W^s(\La)$ and $W^u(\La)$ cut $\partial V$ along $1$-dimensional laminations. Thus
$$\cL^s=W^s(\La)\cap \partial V=W^s(\La)\cap \partial^{in} V\mbox{ and }
\cL^u=W^u(\La)\cap \partial V=W^u(\La)\cap \partial^{out} V$$
are $1$ dimensional laminations.
The aim of this section is to describe elementary properties of the laminations $\cL^s$ of $\partial^{in} V$
and $\cL^u$ of $\partial^{out} V$.  More precisely:

\begin{prop}\label{p.hyperbolicplug}
The laminations $\cL^s$ and $\cL^u$ satisfy the following properties:
\begin{enumerate}
\item The laminations contains finitely many compact leaves. 
\item Every half leaf is asymptotic to a compact leaf.
\item Each compact leaf  may be oriented  so that its holonomy is a contraction.
\end{enumerate}
\end{prop}

\begin{proof}[Sketch of proof]
Item 1 is a direct consequence of Remarks~\ref{r.free} and~\ref{r.free-2}: the compact leaves of $\cL^s$ are in one-to-one correspondance with
the free stable separatrices of periodic orbits, and there are only finitely many such separatrices.

Consider a leaf $\gamma$ of $\cL^s$. According to Remark~\ref{r.free}, it corresponds to a connected component of some
$W^s(x)\setminus \La$ which is a strip $B$ bounded by
two orbits $O_1$ and $O_2$ in $\La\cap W^s(x)$.
Consider the unstable manifold $W^u(O_i)$.  These unstable leaves are
boundary leaves of $W^u(\La)$.  According to Remark~\ref{r.free-2}, it follows that $O_1$ and $O_2$ belong to the
unstable manifolds of periodic orbits having a free separatrix.  Therefore  the $\lambda$-lemma implies that $B$ accumulates these
free separatrices, and $\gamma$ accumulates on the corresponding compact leaves of $\cL^s$.
Lemma 1.8 of \cite{BeBo} makes this argument rigorously for proving item 2.

Let us now explain item 3.  Let $\gamma_0$ be a compact leaf of $\cL^s$. Then $\gamma_0$ is the intersection of $\partial^{in} V$ with a free stable separatrix of 
a periodic orbit $O_0\subset \La$.
The leaves of $\cL^s$, in the neighborhood of $\gamma_0$, are the transverse intersection of the weak stable manifolds of orbits which belong to
$W^u(O_0)$. Notice that $W^u(O_0)$ and $\partial^{in}V$ are both transverse to the lamination $W^s(\La)$.
Furthermore, $\gamma_0$ and $O_0$ are
contained in the same leaf  $W^s(O_0)$ which is either a cylinder or a Mobius band.
\begin{itemize}
\item If $W^s(O_0)$ is a cylinder then,
$\gamma_0$ and $O_0$ are homotopic in the leaf $W^s(O_0)$.  Therefore the holonomy of $\cL^s$ along $\gamma_0$ is
conjugated the holonomy of the lamination induced by
$W^s(\La)$ on $W^u(O_0)$. This holonomy is a contraction if one endows the orbit $O_0$ with the orientation induced by the vector field $-X$.

\item If $W^s(O_0)$ is a Mobius band then,
$\gamma_0$ is homotopic (in $W^s(O_0)$ to  $2.O_0$ .  Therefore the holonomy of $\cL^s$ along $\gamma_0$ is conjugated to
the square of the holonomy of the lamination induced by
$W^s(\La)$ on $W^u(O_0)$: once again, this holonomy is a contraction if one endows the orbit $O_0$ with the orientation induced by the vector field $-X$.
\end{itemize}
\end{proof}

\begin{defi}
\label{d.MS}
A lamination (resp. a foliation) on a closed surface is called $\cL$ a  \emph{MS-lamination}\footnote{``MS" stands for ``Morse-Smale"} (resp. a \emph{MS-foliation}) if it satisfies the following properties:
\begin{enumerate}
\item it has only finitely many compact leaves,
\item every half leaf is asymptotic to a compact leaf,
\item each compact leaf  may be oriented  so that its holonomy is a contraction.
\end{enumerate}
\end{defi}

Proposition~\ref{p.hyperbolicplug} states that the entrance/exit laminations of a hyperbolic plug are MS-laminations.

\begin{rema}
If $(V,X)$ is a hyperbolic attracting plug, then $\cL^s$ is a foliation on $\partial^{in} V$ (and $\partial^{out} V$ is empty).
In particular, $\partial V$ consists of some tori  (and possibly some Klein bottles
if $V$ is not orientable).  
\end{rema}

\subsection{Connected component of the complement of the laminations}

Let us start with a very general observation.
 
 \begin{definition}
 \label{d.strip-exceptionnal}
  Let $S$ be a closed surface and $\cL$ be a $1$-dimensional lamination with finitely many compact leaves, 
 and $C$ be a connected components of $S\setminus \cL$. We call $C$ a \emph{strip} if it satisfies the two following properties:
 \begin{itemize}
  \item $C$ is homeomorphic to $\RR^2$,
  \item the accessible boundary\footnote{that is, the points in the boundary which are  an extremal point of a segment whose interior is contained in $C$.} of $C$ consists in exactly two non-compact leaves of $\cL$ which are asymptotic two each other at both ends,
 \end{itemize}
 Otheriwse we say that $C$ is an \emph{exceptionnal component} of $S\setminus\cL$. 
\end{definition}

 \begin{lemm}\label{l.exceptional}
 Let $S$ be a closed surface and $\cL$ be a $1$-dimensional lamination with finitely many compact leaves. Then there are only finitely many exceptionnal components in $S\setminus\cL$.
 \end{lemm}

 \begin{proof}
 There is  a smooth Morse Smale vector field  $Z$ transverse to the lamination (one easily build a
 continuous vector field transverse to the lamination; transversality is an open property; hence one may perturb this vector
 field by to turn it into a smooth Morse-Smale vector field). Now, at most finitely many connected component or $S\setminus \cL$ may contain
 a singular point of $Z$ or a whole periodic orbit of $Z$.

 If $C$ is a component which does not contain any singular point  and any periodic orbit, then the dynamics of $Z$ restricted to $C$ does not contain
 any non-wandering point. Furthermore, by transversality of $Z$ with $\cL$ (hence of the accessible boundary or $C$)
 any orbit reaching a small neighborhood of $\partial C$ goes out of $C$ in a finite time.
 As a direct consequence, there is $T>0$ so that no orbit of  $Z_{|C}$ is defined for a time interval of lenght larger than $T$. This implies 
 that either $C$ is an annulus bounded by two compact leaves of $\cL$, or $C$ is a strip (as definition~\ref{d.strip-exceptionnal}).
 \end{proof}

 \subsection{The crossing map}
 \label{ss.crossing-map}

We consider a hyperbolic plug $(V,X)$. As usual, we denote by $\La$ the maximal invariant set of $(V,X)$, by $\cL^s\subset\partial^{in} V$ the entrance lamination of $(V,X)$, and 
 by $\cL^u\subset\partial^{out} V$ the exit lamination of $(V,X)$. 
 
 \begin{defi}
 The positive orbit of any point $x\in \partial^{in}V\setminus \cL^s$ reaches $\partial^{out}V$ in a finite time at a point $\Ga(x)\in\partial^{out} V\setminus\cL^u$. 
 The map $x\mapsto\Ga(x)$  will be called \emph{the crossing map} of the plug $(V,X)$. Using the fact that the orbits of $X$ are transverse to $\partial V$, one easily sees that $\Ga$ defines a diffeomorphism from 
 $\partial^{in}V\setminus \cL^s$ to $\partial^{in}V\setminus \cL^s$. 
 \end{defi}

 \begin{lemm} \label{l.exceptional-inout} 
 A connected component $C$ of $\partial^{in}V\setminus \cL^s$ is a strip if and only if $\Ga(C)$ is a strip in $\partial^{out}V\setminus \cL^u$.
 \end{lemm}

 \begin{proof}
 Using the fact that $\Ga$ is a diffeomorphism, up to reversing the flow of $X$, one is reduced to prove the following assertion:
 if $C$ is a strip then $\Ga(C)$ is a strip.

Assume that $C$ is a strip. First notice that $\Ga(C)$ is homeomorphic to $C$, hence to $\RR^2$.
 It remains to check that the accessible boundary consists in $2$ leaves.
 For that purpose, one considers a smooth vector field $Z$ on $\partial^{in}V$ transverse to $\cL^s$ and without singular point in $C$.
 Then the orbits of $Z$ restricted to the union of $C$ with its accessible boundary induces a trivial foliation
 by segments. Consider now the image of this foliation by $\Ga$.  This is a foliation of $\Ga(C)$.
 However, as $Z$ is transverse to the  boundary of $C$ which is contained in the stable manifold of the hyperbolic set $\La$,
 the $\lambda$-lemma (or a cone field argument) implies that $\Ga_*(Z)$ tends uniformly to the tangent direction to the
 unstable lamination $\cL^u$ when $\Ga(x)$ tends to the boundary of $\Ga(C)$.

 Thus one gets a foliation on the closure of $\Ga(C)$; this implies that $\Ga(C)$ is a strip, concluding.
 \end{proof}

 A very similar proof allows us to prove:
 
 \begin{lemm}
 A connected component $C$ of $\partial^{in} V\setminus \cL^s$ is an annulus bounded by two compact leaves (resp. a Mobius band bounded by one compact leaf) of $\cL^s$  if and only if  $\Ga(C)$ is an annulus bounded by two compact leaves  (resp. a Mobius band bounded by one compact leaf) of $\cL^u$.
 \end{lemm}

\subsection{Filling MS-lamination, and pre-foliation}

If a hyperbolic plug can be embedded in an Anosov flow, the stable and unstable manifolds of its maximal invariant set is a sub-lamination of
the stable and unstable foliations, respectively. This leads to restrictions on its stable and unstable laminations.

\begin{defi}
We say that a lamination $\cL$ of a closed surface $S$ is a \emph{pre-foliation} if it can be completed as
a foliation of $S$. Notice that this implies that every connected component of $S$ is either
the torus $\TT^2$ or the Klein bottle $\KK$.
\end{defi}

As a direct consequence of Proposition~\ref{p.hyperbolicplug}, Lemmas~\ref{l.exceptional} and~\ref{l.exceptional-inout}, one obtains:

\begin{lemm}
\label{l.pre-foliation}
Let $(V,X)$ be an hyperbolic plug.  The lamination $\cL^s$ is a pre-foliation if and only if every exceptional
component of $\partial^{in} V\setminus \cL^s$ is either an annulus or a Mobius band bounded by compact leaves of $\cL^s$.
Furthermore,
$$\cL^s \mbox{  is a pre-foliation if and only if } \cL^u \mbox{  is a pre-foliation }.$$
\end{lemm}


The components of $\partial^{in} V\setminus \cL^s$ (resp. $\partial^{out} V\setminus \cL^u$) which are annuli or Mobius band will sometimes lead to specific difficulties.
For this reason we introduce a more restrictive notion:

\begin{defi}\label{s.lamination} 
A lamination $\cL$ on a closed surface $S$ is called a \emph{filling MS-lamination} if:
\begin{itemize}
 \item this is a MS-lamination (see definition~\ref{d.MS}), 
 \item $S\setminus \cL$ has no exceptional component (in other words, every connected component of $S\setminus \cL$ is a strip, see definition~\ref{d.strip-exceptionnal} and figure~\ref{f.filling-MS}).
\end{itemize}
\end{defi}

\begin{figure}[htp]
\begin{center}
  \includegraphics[totalheight=5cm]{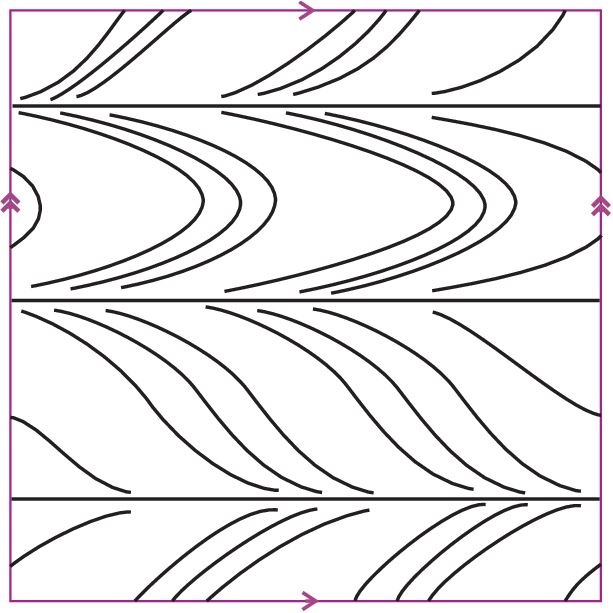}\\
  \caption{A filling MS Lamination}\label{f.filling-MS}
  \end{center}
\end{figure}

\begin{lemm}
\label{l.embed-in-foliation}
Let $\cL$ be a filling MS-lamination of a closed surface $S$.  Then $\cL$ is  a prefoliation. In particular, every component of $S$ is either a torus $\TT^2$ or a Klein bottle $\KK$.
Furthermore, if $\cF$ is a foliation containing $\cL$ as a sub-lamination, then $\cF$ is a MS-foliation.
Finally, two foliations containing $\cL$ as a sublamination are topologically conjugated.
\end{lemm}

\begin{proof}[Idea for the proof] 
One just need to foliate each connected component of the complement of $\cL$. Such a component is  strip
bounded by two asymptotic leaves.
There is a unique way to foliate such a strip up to topological conjugacy.
\end{proof}

Once again, as a consequence  of  Proposition~\ref{p.hyperbolicplug}  and
Lemmas~\ref{l.exceptional} and~\ref{l.exceptional-inout}, one easily shows:

\begin{lemm}\label{l.simple}
Let $(V,X)$ be a hyperbolic plug. The stable lamination $\cL^s$ is a filling MS-lamination
if and only if $\partial^{in}V\setminus \cL^s$ has no exceptional component.
The entrance lamination $\cL^s$ is a filling MS-lamination if and only if $\cL^u$ is filling MS-lamination.
\end{lemm}

If $\cL^s$ and $\cL^u$ are filling MS-laminations, we will speak say that $(V,X)$ is a \emph{hyperbolic plugs with filling MS-laminations}.

\subsection{Hyperbolic plugs with pre-foliations and invariant foliation}

Let $(V,X)$ be a hyperbolic plug so that the entrance lamination $\cL^s_X$ and the exit lamination $\cL^u_X$ are pre-foliations.
The following lemma shows that every foliation on $\partial^{out}V$ transverse to $\cL^u$ extends to an $X$-invariant foliation
of $V$ containing $W^s(\La)$ as a sublamination.

\begin{lemm}\label{l.invariantfoliations} 
Let $F^s$ be a foliation on $\partial^{out}V$ which is transverse to $\cL^u$. Then $F^s$ extends on $V$ in an invariant  foliation $\cF^s$ with two-dimensional leaves containing $W^s(\La)$ as a sublamination.  In particular,
$\cF^s\cap \partial^{in}V$ is a foliation which extends $\cL^s$.
\end{lemm}

\begin{proof} 
First notice that $V\setminus W^s(\La)$ is the (backwards) $X$-orbit of the set $\partial^{out}V$. The $X$-orbits of the leaves of $F^s$ are the leaves of a foliation $\cF^s_0$ of $V\setminus W^s(\La)$. The foliation $\cF^s_0$ is tangent to $X$ and therefore transverse to $\partial^{in}V$. It induces a $1$-foliation on $\partial^{in}V\setminus \cL^s$.

Thus, it is enough to check that the leaves of $\cF^s_0$ tend to the leaves of $W^s(\La)$. Notice that a point in $V\setminus W^s(\La)$
in a very small neighborhood of $W^s(\La)$ has its positive orbit which meets $\partial^{out} V$ at a point very close to $\cL^u$.  Thus, for
proving that $\cF^s_0$ extends by continuity by $W^s(\La)$, it is enough to consider the negative orbits through small segments of leaves of
$F^s$ centered at points of $\cL^u$.

Thus we fix $\varepsilon>0$ and we consider the family of segments of leaves of $F^s$ centered at the points   of $\cL^u$.  It is a
$C^1$-continuous family of segments parameterized by a compact set. We consider the negative orbits by the flow of  $X$ of these segments.
We need to prove that these orbits tends to the stable leaves of $\La$ as the time tends to $-\infty$. That is the classical $\lambda$-lemma.
\end{proof}

\begin{lemm} 
Every $X$-invariant $\cC^{0,1}$ foliation $\cF$ containing $W^s(\La)$ as a sublamination induces on
$\partial^{out}V$ a one-dimensional foliation $\cL$ transverse to $\cL^u$.
\end{lemm}

\begin{proof} 
As $\cF$ is $X$-invariant, it is transverse to $\partial^{out}V$. Thus it induces a one-dimensional foliation
$\cL$ on $\partial^{out}V$. Furthermore, as $\cF$ is $\cC^{0,1}$ and contains $W^s(\La)$ which is transverse to $W^u(\La)$ along $\La$,
one gets that $\cF$ is transverse to $W^u(\La)$ in a neighborhood $\cO$ of $\La$.
Now every point $x$ of $\cL^u$ is on the orbit of $X$ of a point in $\cO$. The $X$-invariance of both $W^s(\La)$ and
$\cF$ implies that $\cF$ and $W^u(\La)$
are transverse at $x$. One deduces that $\cL$ is transverse to $\cL^u$.
\end{proof}

Thus, the $2$-dimensional $X$-invariant foliations on $V$ containing $W^s(\La)$ as a sublamination are in
one to one correspondance with the $1$-dimensional
foliation on $\partial^{out}V$ transverse to $\cL^u$.

This shows in particular that $V$ admits many invariant foliations.
These foliations will help us to recover the hyperbolic structure when we will glue the exit with the entrance boundaries. For that we
need to control the expansion/contraction properties of the crossing  along the directions tangent to these foliations.
This is the aim of next lemma:

\begin{lemm}\label{l.stablefoliation} Let $\cF^s$ be an invariant $\cC^{0,1}$-foliation on $V$ containing $W^s(\La)$, and let $F^s_{in}$ denote the
intersection  $\cF^s\cap \partial^{in}V$.  Then the derivative $\Ga_*$ of the crossing map contracts arbitrarily
uniformly the vectors tangent to $F^s_{in}$ in small neighborhoods of $\cL^s$. More precisely, for every $\varepsilon>0$ there is $\delta>0$ so that,
given any $x\in\partial^{in}V\setminus \cL^s$ with $d(x,\cL^s)<\delta$, given any vector $u\in T_x\partial^{in}V$
tangent to the leaf of $F^s_{in}$ one has:
$$\|\Ga_*(u)\|<\varepsilon\|u\|.$$
\end{lemm}
\begin{proof}

The maximal invariant set $\La$ admits arbitrarily small filtrating neighborhoods. Recall that $\La$ is hyperbolic so that the area on the
center stable space $E^{cs}(x) =E^{ss}(x)\oplus \RR X(x)$ is uniformly contracted along the orbit of $x\in\La$.  Let $E^{cs}(x)$, $x\in V$
denotes the tangent plane to $\cF^s$. One denotes by  $$J^s_t(x)=\left|Det \left((X_t)_*|_{_{E^{cs}(x)}}\right)\right|$$
the determinant of the restriction to
$E^{cs}(x)$ of the derivative of the flow of $X$ at the time $t$.

One deduces that
there is a filtrating neighborhood $\cU$ of $\La$ and $0<\lambda< 1$ so that
$$x\in\cU  \mbox{ and }t\geq 1\Rightarrow J^s_t(x)<\lambda^t<1.$$

The proof consists now in noting that, for $x\in\partial^{in}V$, close to $\cL^s$,  the orbit segment joining $x$ to $\Ga(x)$
consists in a segment contained in $\cU$,  whose length tends to infinity
as $x$ tends to $\cL^s$, and two bounded segments: one joining $x\in\partial^{in}V$ to $\cU$,
and one joining $\cU$ to $\Ga(x)\in\partial^{out}V$.

The first and the last segments have a bounded effect on $J^s_t$, as their lengths are uniformly bounded. One deduces that
$$J^s_{\tau(x)}(x)\to 0 \mbox{ as } x\to \La^s,$$
where $\tau(x)$ is the crossing time of $x$ (that is $\Ga(x)=X_{\tau(x)}(x))$.

Let $u$ be the unit vector tangent to $\cL^s$ at $x\in L^s$.
The difficulty is that the vector $\Ga_*(u)$ is not the image $(X_{\tau(x)})_*(u)$;
the vector $\Ga_*(u)$ is the projection
along $X(y)$ on $T_y \partial^{out}V$ of $(X_{\tau(x)})_*(u)$,   where $y=\Ga(x)=X_{\tau(x)}(x)$.
In order to simplify the calculation let us choose a metric on $V$ so that $X$ is orthogonal to $\partial V$ and $\|X\|=1$.
With these notations, one gets
$$\left\|\Ga_*(u)\right\|= J^s_{\tau(x)}(x).$$
Thus $\frac{\left\|\Ga_*(u)\right\|}{\|u\|}$ tends to $0$ as $x\to \cL^s$. This completes  the proof.
\end{proof}

\subsection{Strongly transverse lamination}

Consider an Anosov flow $X$ on a closed $3$-manifold $M$, and assume that two plugs $(V_1,X_1)$ and $(V_2,X_2)$ are embedded in $(M,X)$ so that
a component $S$ of $\partial^{out}V_1$ is also a component of $\partial^{in} V_2$.   Then the laminations $\cL_1^u$ and $\cL_2^s$ are  not only
transverse, they extend on $S$ as two transverse foliations (the unstable and stable foliation of $X$):
not every two transverse filling MS-laminations may extend as two transverse foliations.
This difficulty leads to the following definition:

\begin{defi} Let $S$ be a compact surface and let $\cL_1$ and $\cL_2$ be two laminations on $S$.  We say that $\cL_1$ and $\cL_2$ are
\emph{strongly transverse} if
\begin{itemize}
 \item $\cL_1$ and $\cL_2$ are transverse at every point of $\cL_1\cap \cL_2$
 \item every connected component of $S\setminus \cL_1\cap\cL_2$  is a disc whose closure is the image of a square $[0,1]^2$ by an immersion which is a diffeomorphism
 on $(0,1)^2$, and so that $[0,1]\times\{0,1\}$ is mapped in leaves of $\cL_1$ and $\{0,1\}\times [0,1]$ is mapped in leaves of $\cL_2$.
\end{itemize}
\end{defi}

One can  easily show the following lemma.
\begin{lemm}\label{l.stronglytransverse}
 If $\cL_1$ and $\cL_2$ are strongly transverse, they extend in transverse foliations.  In particular, $\cL_1$ and $\cL_2$ are pre-foliations.
\end{lemm}

\subsection{Gluing vector fields}

Let $V_1$ and $V_2$ be manifolds with boundary, $S_1$ and $S_2$ be  unions of boundary components of $\partial V_1$ and $\partial V_2$,
and $X_1$  and $X_2$ be vector fields on $V_1$ and $V_2$, transverse to $S_1$ and $S_2$ so that $X_1$
goes out $V_1$ through $S_1$ and in $V_2$ through $S_2$.
Let $\varphi\colon S_1\to S_2$ be a diffeomorphism. Let $V$ be the quotient $(V_1\amalg V_2)/x\simeq \varphi(x)$.

\begin{lemm}\label{l.glue} 
With the notation above there is a differential structure on $V$ so that $V$ is a smooth manifold and
there is a $C^1$-vector field on $V$ so that the restriction of $X$ to $V_i$ is $X_i$, $i=1,2$.
\end{lemm}

\begin{demo} 
Just notice that the flow box theorem implies that there is a tubular neighbourhood $U_i$ of $S_i$ in $V_i$, and some coordinates $(x,t):U_1\to S_i\times (-\varepsilon, 0]$ (resp. $(x,t):U_2\to S_2\times [0,\varepsilon)$) so that, in these coordinates $X_i$ is the trivial vector field  $\frac{\partial}{\partial t}$.
\end{demo}

\begin{defi}
\label{d.strongly-isotopic}
Let $X_0$ and $X_1$ be two vector fields on the same compact $3$-manifold with boundary $V$, so that $(V,X_0)$ and $(V,X_1)$ both are hyperbolic plugs with filling MS-laminations. Let $\varphi_0:\partial^{out}_{X_0} V\to\partial^{in}_{X_0} V$ and  $\varphi_1:\partial^{out}_{X_1} V\to\partial^{in}_{X_1} V$ be strongly transverse gluing maps respectively for $(V,X_0)$ and $(V,X_1)$. We say that $(V,X_0,\varphi_0)$ and $(V,X_1,\varphi_1)$ are \emph{strongly isotopic} if there exists a continuous path $(U, X_t,\varphi_t)_{t\in [0,1]}$ of hyperbolic plugs with filling MS-laminations and strongly transverse gluing maps. 
\end{defi}

\begin{rema}
 If $(V,X_0,\varphi_0)$ and $(V,X_1,\varphi_1)$ are strongly isotopic, the structural stability of hyperbolic plugs imples that $(V,X_0)$ and $(V,X_1)$ are topologically equivalent.
\end{rema}

\section{Gluing hyperbolic plugs without cycles}
\label{s.no-cycle}

In this section, we consider two hyperbolic plugs $(U,X)$ and $(V,Y)$. We also consider a union $T^{out}$ of connected components of $\partial^{out} U$, a union $T^{in}$ of connected components of $\partial^{in} V$, and a gluing map $\varphi\colon T^{out}\to T^{in}$ such that the lamination $\varphi_*(\cL^u_X)$ is transverse to the lamination $\cL^s_Y$. Then we consider the manifold with boundary $W:=(X\sqcup V)/\varphi$, and the vector field $Z$ induced by $X$ and $Y$ on $W$.

The aim of the section is to prove that $(W,Z)$ is a hyperbolic plug (Proposition~\ref{p.plug}) and that $(W,Z)$ has filling MS-laminations provided that this is the case for $(U,X)$ and $(V,Y)$ and provided that $\varphi$ is a strongly transevers egluing map (Proposition~\ref{p.plug2}).

\subsection{Hyperbolicity of the new plug: proof of Proposition~\ref{p.plug}}

\begin{proof}[Proof of Proposition~\ref{p.plug}]
The vector field $Z$ is transverse to the boundary of $W$; hence $(W,Z)$ is a plug. So we only need to check that the maximal invariant set of $(W,Z)$ is a hyperbolic set.
Let $\La_X$, $\La_Y$, $\La_Z$ be the maximal invariant sets of $(U,X)$, $(V,Y)$, $(W,Z)$ respectively. Then $\La_Z$ is the union of $\La_X$, $\La_Y$ 
and the $Z$-orbit of the set $\varphi_*(\cL^u_X)\cap \cL^s_Y$. A classical consequence of the hyperbolic theory asserts that the orbit of $\varphi_*(\cL^u_X)\cap \cL^s_Y$ inherit of a hyperbolic structure, so that the maximal invariant set on the vector field $Z$ on $U\sqcup_{\varphi} V$ is hyperbolic: for $x\in \varphi_*(\cL^u_X)\cap \cL^s_Y$, the stable (resp. unstable) bundle at $x$ is the direct sum of the line $\RR.Y(x)$ (resp. $\RR.(\varphi_*X)(x)$) and the line tangent to $\cL^s_Y$ (resp. $\varphi_*(\cL^u_X)$) at $x$.
\end{proof}

The following simple observation (whose proof is left to the reader) will be used many times in remainder of the paper:

\begin{prop}\label{p.gluedlaminations}
The exit boundary of the plug $(W,Z)$ is
$$\partial^{out}W=(\partial^{out}U\setminus T^u)\cup \partial^{out}V.$$
Furthermore, the lamination $\cL^u_W$ coincides with:
 $$\cL^u_X\mbox{ on }\partial^{out}U\setminus T^u\mbox{, and with }
 \cL^u_Y\sqcup (\Ga_Y)_*(\varphi_*(\cL^u_X)\setminus \cL^s_Y) \mbox{ on }\partial^{out}V,$$
 where $\Ga_Y$ is the crossing map of the plug $(V,Y)$.
\end{prop}

\subsection{Filling MS-laminations: proof of Proposition~\ref{p.plug2}}

In this subsection, we assume that the hyperbolic plugs $(U,X)$ and $(V,Y)$ have filling MS-laminations, and that $\varphi$ is a strongly transverse gluing map. 

\begin{proof}[Proof of Proposition~\ref{p.plug2}]
According to Lemma~\ref{l.simple} it is enough to prove that $\cL^u_Z$ is a filling MS-lamination. For that we consider a connected component $C$ of $\partial^{out}W\setminus \cL^u_Z$. 

If $C\subset \partial^{out}U\setminus T^{out}$, then $C$ is a connected component of $\partial^{out}U\setminus \cL^u_X$ and therefore is a strip bounded by two asymptotic leaves, ending the proof in this case.

According to Proposition~\ref{p.gluedlaminations}, we can now assume that $C\subset \partial^{out}V$.
Furthermore, according to Proposition~\ref{p.gluedlaminations},
$C$ is disjoint from $\cL^u_Y$, so that one can consider $\Ga_Y^{-1}(C)$:
\begin{itemize}
 \item either $\Ga_Y^{-1}(C)$ is contained in $\partial^{in} V\setminus T^{in}$.  In that case $\Ga_Y^{-1}(C)$ is
 a connected component of $\partial^{in}V\setminus \cL^s_Y$ so that $C$ itself is a connected component of
 $\partial^{out}V\setminus \cL^u_Y$.  Thus is a strip bounded by two asymptotic leaves, ending the proof in this case.
 \item or $\Ga_Y^{-1}(C)$ is contained in $T^{in}$.
\end{itemize}
Thus we are led to assume that $\Ga_Y^{-1}(C)$ is contained in $T^{in}$.

\begin{lemm}
Assume that $\Ga_Y^{-1}(C)$ is contained in $T^{in}\subset \partial^{in}V$.  Then $\Ga_Y^{-1}(C)$
is a connected component of $T^{in}\setminus (\cL^s_Y\cup \varphi(\cL^u_X))$.
\end{lemm}

\begin{proof}
The set $\Ga_Y^{-1}(C)$ is disjoint from $\cL^s(Y)$ because the range of $\Ga_Y^{-1}$ is $\partial^{in}V\setminus \cL^s_Y$.  
Proposition~\ref{p.gluedlaminations} implies that it is also disjoint from $\varphi(\cL^u_X)$, as $C$ is disjoint from
$\Ga(\varphi(\cL^u_X))\subset \cL^u_W$.  Therefore $\Ga_Y^{-1}(C)$ is contained in a connected component
$C_1$ of
$T^{in}\setminus (\cL^s_Y\cup \varphi(\cL^u_X))$. Now $\Ga_Y(C_1)$ is disjoint from $\cL^u_Z$ (Proposition~\ref{p.gluedlaminations})
one deduces the other inclusion: $\Ga_Y(C_1)\subset C$.
\end{proof}

Let $C_1:=\Ga_Y^{-1}(C)$.  By definition of strongly transverse lamination, $C_1$ is an immersed
square $[0,1]^2$ bounded by two segments $\gamma_1$, $\gamma_2$ (images of the horizontal segments $[0,1]\times \{0,1\}$) in
 $\cL^s_Y$ and two segments $\sigma_1$, $\sigma_2$ in $\varphi_*(\cL^u_X)$ (images of the vertical segments
 $\{0,1\}\times [0,1] $). Notice that $L_i=\Ga_Y(\sigma_i)$, $i=1,2$, is a leaf of $\cL^u_W$ contained in the accessible boundary of $C$.
 We will see that $L_1$ and $L_2$ are asymptotic on both sides.  More precisely:
 
 \begin{lemm}
 \label{l.strip} 
 With the notation above, that there is a foliation on $C\cup L_1\cup L_2$ whose leaves are segments with one end point in $L_1$ and the other end point on
 $L_2$.  Furthermore, the length of the leaves tends to $0$ when one of their endpoints tends to one end $L_i$.
 \end{lemm}

\begin{proof}
According to Lemma~\ref{l.invariantfoliations} applied to $(V,Y)$, there is an $Y$-invariant foliation $\cF^s$ on $V$  containing $W^s(\La_Y)$ as a sub-lamination
and so that the foliation $\cF^{s}_{out}$ induced by $\cF^s$ on $\partial^{out} V$ is transverse to $\cL^u_Y$. We denote by $\cF^{s}_{in}$ the foliation induced by 
$\cF^s$ on $\partial^{in} V$. The lamination $\cL^s_Y$ is a sublamination of $\cF^{s}_{in}$.  In particular, the sides $\gamma_1$ and $\gamma_2$ are leaf
segments of $\cF^{s}_{in}$. Unfortunately, the foliation $\cF^{s}_{in}$ may fail to be transverse to the segments $\sigma_1,\sigma_2$. However,
since the transversality is an open property, there is a neighborhood of $\gamma_1\cup\gamma_2$ in $C_1$ which is foliated by segments of leaves of $\cF^{s}_{in}$ 
with one endpoint on $\sigma_1$ and the other endpoint on $\sigma_2$. As a consequence, there is a foliation $\cG_1$ of $C_1$, which coincides with $\cF^{s,in}\cap C_1$ 
on a neighborhood of $\gamma_1\cup\gamma_2$, so that the leaves of $\cG_1$ are segments with one endpoint on $\sigma_1$  and the other endpoint on $\sigma_2$.
The announced foliation is $\Ga_Y(\cG_1)$.  Lemma~\ref{l.stablefoliation} implies that the length of the leaves of this foliation tends to $0$ when
one of their endpoints tends to one of the ends of $\Ga_Y(\sigma_i)=L_i$.
\end{proof}

Lemma~\ref{l.strip} implies that $C$ is a strip whose accessible boundary consists in two leaves of $\cL^u_Z$ with are asymptotic to each other at both ends. 
This concludes the proof of Proposition~\ref{p.plug2}.
\end{proof}

 \begin{figure}[ht]
\begin{center}
\includegraphics[totalheight=5cm]{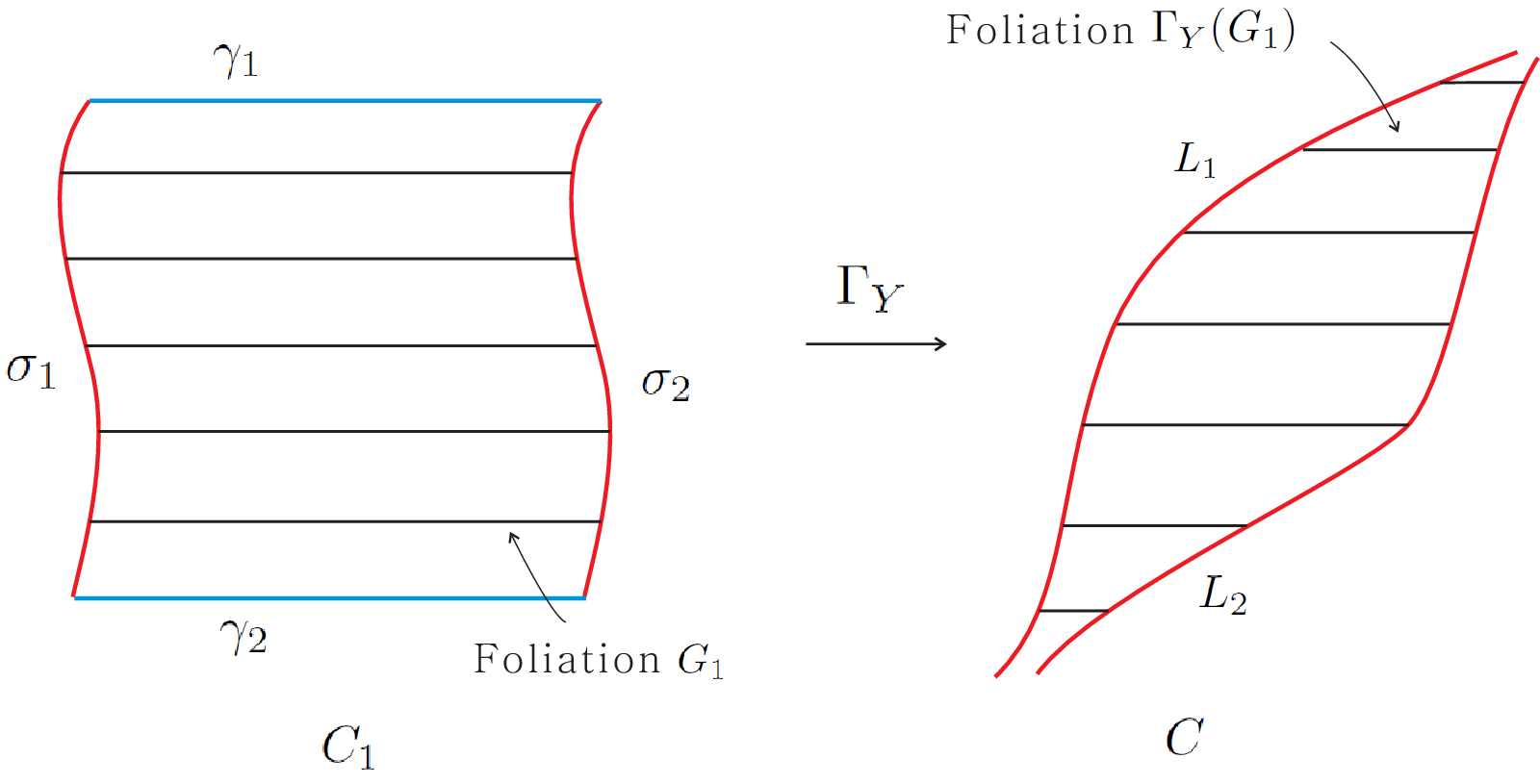}
\caption{\label{f.foliation-crossingmap}The crossing map $\Ga_Y: (C_1, G_1)\rightarrow (C, \Ga_Y (G_1))$.}
\end{center}
\end{figure}

\section{Normal form}

Given a hyperbolic plug $(V,X)$ with filling MS-laminations, and   a strongly transverse gluing map $\varphi_0\colon \partial^{out}V\to \partial^{in}V$, the main purpose of the present section is to perturb $(V,X)$ (within its topological equivalence class) and the map $\varphi_0$ (with its isotopy class of  strongly transverse gluing map) in order to get some foliations on $\partial^{in} V$ and $\partial^{out} V$ satisfying some nice properties. Our perturbation uses Markov partition by disjoint rectangles. Such Markov partition exists if and only if the hyperbolic set does not contain any attractor nor repeller. For this reason, we will need to assume that the maximal invariant set of $(V,X)$ does not contain attractors and repellors. 

\begin{defi}
A hyperbolic plug $(V,X)$ is called a \emph{saddle hyperbolic plug}, if the maximal invariant set of $(V,X)$ does not contain attractors, nor reppellors.
\end{defi}

More precisely, we will prove the following proposition: 

\begin{prop}\label{p.preparation} 
Let $(V,X)$ be a saddle hyperbolic plug with filling MS-laminations, endowed with a strongly transverse gluing map $\varphi_0\colon \partial^{out}V\to \partial^{in}V$.
Then there exists a vector field $Y$ on $V$ arbitrarily $C^1$-close to $X$ and a map $\varphi_1:\partial^{out}_X V\to\partial^{in}_X V$ with the following properties.
 \begin{itemize}
  \item $(V,Y)$ is a hyperbolic plug, $\varphi_1$ is a strongly transverse gluing map for $(V,Y)$, and $(V,Y,\varphi_1)$ is strongly isotopic to $(V,X,\varphi_0)$ (definition~\ref{d.strongly-isotopic}). We denote by $\Lambda_Y$ the maximal invariant set of $(V,Y)$.
  \item There exist  smooth $Y$-invariant foliations $\cG^s$ and $\cG^u$ on $V$, so that $W^s(\La_Y)$ is a sublamination of $\cG^s$ and $W^u(\La_Y)$ is a sublamination of $\cG^u$. We denote by $\cG^s_{in}$ $\cG^s_{out}$, $\cG^u_{in}$, and $\cG^u_{out}$ the intersection of $\cG^s$ and $\cG^u$ with $\partial^{in}V$ and $\partial^{out}V$.
  \item For each compact leaf of $\cG^u_{out}$ and of $\cG^s_{in}$ the holonomy is conjugated to a homothety.
  \item The image of $\cG^u_{out}$ by $\varphi_1$ is transverse to $\cG^s_{in}$. 
 \end{itemize}
\end{prop}

Appart from proving this proposition, we will also establish some bound of the rate of contraction/expansion of the crossing map (see Lemma~\ref{l.stronghyperbolicity}), and build specific invariant neighborhoods of the maximal invariant set of $(V,X)$, called adapted neighborhoods.

\subsection{Linear model for the vector field $X$}

If $\Lambda$ is a locally maximal hyperbolic set without attractors and repellers for a vector field in dimension $3$, then $\La$ admits a transverse 
cross section $\Si$. Moreover, the first return map on $\Sigma$ admits a Markov partition by disjoint rectangles contained in $\Si$, so that $\Si\cap \La$ is the maximal invariant set of the union of the rectangles (see for instance \cite{BeBo}).
One can refine such a Markov partition by considering intersection of the rectangles with their (positive or negative) images under
the first return map. Iterating the process the diameters of the rectangles can be
made arbitrarily small.  When the diameters  are small enough the restrictions of the first
return map to the rectangles are almost affine; by a classical argument one can perform
a $C^1$-small perturbation of the vector field so that the first return map becomes affine on each rectangles, preserving the vertical and horizontal foliations. 
This proves to the following lemma:

\begin{lemm}
\label{l.affine} 
Let $(V,X)$ be a saddle hyperbolic plug and $\La_X$ its maximal invariant set.
There is an arbitrarily small $C^1$ perturbation  $Y$ of $X$, topologically equivalent to $X$,  admitting an \emph{affine Markov partition},
that is:
\begin{itemize}
 \item a Markov partition consisting in smooth disjoint rectangles,
 \item the boundary of each rectangle is disjoint from $\La_Y$,
 \item every orbit of $\La_Y$ meets the union of the interior of the rectangles,
 \item there are coordinates on the rectangles so that the first return map preserves the horizontal and vertical foliations defined by these coordinates and is affine on each rectangle.
\end{itemize}.
\end{lemm}

Lemma~\ref{l.affine}  motivates the following definition:

\begin{defi}
An \emph{affine plug} is a saddle hyperbolic plug admitting an affine Markov partition by disjoint rectangles.
\end{defi}

\begin{rema}
\label{r.holonomy-affine}
For an affine plug, the holonomy of each compact leaf of the entrance (resp. exit) lamination $\cL^s$ (resp. $\cL^u$) is an affine contraction, \emph{i.e.} an homothety.
\end{rema}

As $Y$ can be chosen arbitrarily $C^1$-close to $X$, the laminations $(\varphi_{0})_*(\cL^u_Y)$ and $\cL^s_Y$ are still strongly transverse. Thus $(V,Y,\varphi_0)$ is
still a saddle hyperbolic plug endowed with a strongly transverse gluing diffeomorphism. 

By pushing along the flow the vertical and horizontal foliations of the rectangle of an affine Markov partition, one gets a pair of invariant foliations which extend the stable and unstable laminations of the maximal invariant set. More precisely:

\begin{lemm}\label{l.linearfoliation} 
Let $(V,Y)$ be an affine plug with maximal invariant set $\La$.
There is an invariant neighborhood $\cU_0$ of $\Lambda$ endowed with two smooth invariant 2-dimensional foliations $\cF^s$ and $\cF^u$ so that
\begin{itemize}
 \item $\cF^s$ and $\cF^u$  are transverse to each other,
 \item $\cF^s$ and $\cF^u$  are both transverse to $\partial V=\partial^{in} V\cup \partial^{out} V$,
 \item the leaves of the laminations $W^s(\La)$ and $W^u(\La)$ are leaves of $\cF^s$ and $\cF^u$, respectively.
\end{itemize}
Let $\cU^{in}_0$ and $\cU^{out}_0$ denote the intersections of $\cU$ with $\partial^{in} V$ and $\partial^{out} V$ respectively. By transversality, $\cF^s$ and $\cF^u$ induce :
\begin{itemize}
 \item two smooth transverse $1$-dimensional foliations $\cW^s_{in},\cW^u_{in}$ on $\cU^{in}_0$, so that $\cL^s$ is a sublamination of $\cW^s_{in}$,
 \item two smooth transverse $1$-dimensional foliations $\cW^s_{out},\cW^u_{out}$ on $\cU^{out}_0$, so that $\cL^u$ is a sublamination of $\cW^u_{out}$.
 \end{itemize}
Furthermore, one can choose the foliations $\cF^s$ and $\cF^u$ so that the only compact leaves of $\cW^s_{in}$ and $\cW^u_{out}$ are those of the laminations $\cL^s$ and $\cL^u$, and their holonomies are homotheties.
\end{lemm}

\begin{proof} 
The neighborhood $\cU_0$ is just the union of the orbits of $Y$ intersecting the rectangles of an affine Markov partition of $\La$. The  foliation $\cF^s$ and $\cF^u$ are obtained by saturating by the flow of $Y$ the vertical and horizontal foliations of the rectangles.
\end{proof}

\subsection{Local linearization of the gluing map}
\label{ss.lineargluing}

The aim of this section is to prove the following result, which provides a kind of ``normal form" for the gluing map of an affine plug
in a neighborhood of the intersection of the laminations of the boundary.

\begin{prop}\label{p.lineargluing} 
Let $(V,X)$ be  an affine plug with maximal invariant set $\La$, so that the entrance lamination $\cL^s$ and the exit lamination $\cL^u$ of $(V,X)$ are prefoliations. 
Let $\varphi_0\colon \partial^{out}V\to \partial^{in} V$ be a diffeomorphism  so that $\varphi_{0,*}(\cL^u)$ and $\cL^s$ are strongly transverse.
Let $\cU_0$ be an invariant neighborhood of $\La$ endowed with two smooth foliations $\cF^s$ and $\cF^u$
as given by Lemma~\ref{l.linearfoliation}. Let $\cU^{in}_0$ and $\cU^{out}_0$ be the intersections of $\cU$ with $\partial^{in} V$ and $\partial^{out} V$. Observe that these are neighbourhoods of the laminations $\cL^s$ and $\cL^u$ in $\partial^{in} V$ and $\partial^{out} V$ respectively. Let $\cW^s_{in}$, $\cW^u_{in}$, $\cW^s_{out}$ and $\cW^u_{out}$  be the foliations induced by on $\cF^s$ and $\cF^u$ on  $\cU^{in}_0$ and $\cU^{out}_0$.

There exists a diffeomorphism $\varphi$ and an invariant neighborhood  $\cU\subset \cU_0$ of $\La$ so that:
\begin{itemize}
\item $\varphi_*(\cL^u)$ and $\cL^s$ are strongly transverse;
\item the foliations $\varphi_*(\cW^s_{out})$ and $\varphi_*(\cW^u_{out})$ coincide with $\cW^s_{in}$ and $\cW^u_{in}$ on $\varphi(\cU^{out})\cap \cU^{in}$.
\item $\varphi$ is isotopic to $\varphi_0$ among strongly transverse gluing maps.
\end{itemize}
\end{prop}

The proof of the proposition uses the following technical lemma:

\begin{lemm}\label{l.transversefoliation} 
Let $D$ be a compact disc of dimension $2$ endowed with $3$ smooth foliations $\cF$,$\cG$, and $\cH$. Assume that $\cF$ is transverse to
both $\cG$ and $\cH$. Let $\cK$ and $\cL$ be sublaminations of $\cF$ and $\cH$ respectively.  One assume that $\cK$ and $\cL$ have empty interior, and
that $\cK\cap\cL$ is disjoint from the boundary $\partial D$.

Then there is a smooth isotopy $(\psi_t)_{t\in [0,1]}$ of diffeomorphisms of $D$ with the following properties:
\begin{itemize}
\item $\psi_0=id$
 \item $\psi_t$ coincides with the identity map in a neighborhood of $\partial D$ for every $t$,
 \item $\psi_t$ preserves each leaf of $\cF$ for every $t$,
 \item $\psi_1(\cH)$ coincides with $\cG$ in a neighborhood of $\cK\cap\psi_1(\cL)$.
\end{itemize}
\end{lemm}

\begin{proof}
First observe that any diffeomorphism coinciding with the identity close to $\partial D$ and preserving every leaf of $\cF$ is
 isotopic to the identity inside the set of diffeomorphisms preserving each leaf of $\cF$: to prove this fact, one just need just consider 
 a barycentric isotopy along the leaves. Therefore, we only need to build the diffeomorphism $\psi_1$.

 Using the fact that  $\cK$ and $\cL$  have empty interior, the intersection of every leaf  of $\cK$ with $\cL$ is totally discontinuous. As $\cK$ has empty interior, one  deduces that  one cover $\cK\cap \cal L$ by finitely many pairwise disjoint rectangles, so that the vertical segments of these rectangles are segments of leaves of $\cF$, the horizontal segments of these rectangles are segments of leaves of $\cG$, and the boundaries of these rectangles is disjoint from $\cK \cap \cL$. Fix such a rectangle $R$.

 Let $\sigma$ be a connected component of $\cK\cap R$ (note that $\sigma$ is a vertical segment of $R$). One can find an arbitrarily thin vertical subrectangle $R_\sigma$ of $R$ containing $\sigma$, so that the vertical sides of $R_\sigma$ are disjoint from $\cK$.  Any connected component of the intersection of a leaf of $\cL$ with $R_\sigma$ is disjoint from the horizontal boundary of $R_\sigma$, and ``crosses $R_\sigma$ horizontally" intersecting every vertical segment of $R_\sigma$. If $R_\sigma$ is thin enough, then the same is true for a connected component of the intersection of a leaf of $\cH$ with $R_\sigma$ in a neighbourhood of $\cL\cap R$ (because $\cH$ is tranverse to $\sigma$ and the horizontal boundary of $R_\sigma$ is disjoint from $\cL$). Therefore one can find a diffeomorphism $\psi_\sigma$ supported in the interior of $R_\sigma$, preserving the vertical segments of $R$ (i.e. the leaves of $\cF$), and so that any connected component of $R_\sigma$ with a leaf $\cH$ is mapped on an horizontal segment of $R_\sigma$ (i.e. in a leaf pof $\cG$) in a neighborhood of $L\cap R_\sigma$. Now one can cover $\cL\cap R$ by finitely many disjoint such rectangles $R_{\sigma_i}$ and the announced diffeomorphism $\psi_1$ is the product of the  $\psi_{\sigma_i}$'s.
\end{proof}

We are now ready for proving the proposition.

\begin{proof}[Proof of Proposition~\ref{p.lineargluing}]
We consider the foliations $\cW^s_{in}$, $\cW^u_{in}$, $\varphi_0(\cW^s_{out})$, and $\varphi_0(\cW^u_{out})$ defined by Lemma~\ref{l.linearfoliation} on some neighborhoods
$\cU_0^{in}\subset \partial^{in} V$ and $\cU_0^{out}\subset \partial^{out} V$ of $\cL^s$ and $\cL^u$ respectively.

Since $(V,X)$ is a saddle hyperbolic plug, the  laminations $\cL^s$ and $\cL^u$ have empty interior. By assumption, the laminations $\varphi_0(\cL^u)$ and $\cL^s$ are (strongly) transverse to each other. One deduces that the intersection  $\varphi_0(\cL^u)\cap \cL^s$ is totally discontinuous compact subset of the surface $\partial^{in}V$.
Therefore $\varphi_0(\cL^u)\cap \cL^s$  can be covered by the interior of a finite union of disjoint arbitrarily small compact discs $D_i$'s.

As the laminations $\varphi_0(\cL^u)$ and  $\cL^s$ are (strongly) transverse, there is a neighborhood $\cO$ of $\varphi_0(\cL^u)\cap \cL^s$ on which the foliations
$\varphi_0(\cW^u_{out})$ and $\cW^s_{in}$ are transverse. By shrinking $\cO$ if necessary, one may assume that $\cO\subset \varphi_0(\cU_0^{out})\cap \cU_0^{in}$.
We choose the discs $D_i$'s small enough so that they are contained in $\cO$.

According to Lemma~\ref{l.transversefoliation} for each of the disc $D_i$ there is a diffeomorphism $\varphi_i$ supported in the interior of $D_i$,
preserving each leaf of $\cW^s_{in}$ (and isotopic to the identity through diffeomorphisms preserving the leaves of $\cW^s_{in}$),
and so that $\psi_i(\varphi_0(\cW^u_{out}))$ coincides with $\cW^u_{in}$ in the neighborhood of $\psi_i(\varphi_0(\cL^u))\cap D_i)\cap \cL^s$.
Let $\psi^{in}$ be the diffeomorphism of $\partial^{in}V$ coinciding with $\psi_i$ on each $D_i$ and with the identity out of the $D_i$.
Let $\varphi_1:=\psi^{in}\circ \varphi_0$. Note that $\varphi_1(\cW^u_{out})$ coincides with $\cW^u_{in}$ in the neighborhood of $\varphi_1(\cL^u)\cap \cL^s$.
So we got half of the conclusion.  Now, we need now to push $\varphi_0(\cW^s_{out})$ on $\cW^s_{in}$ without destroying what has been done.

 For that purpose, we consider the foliations $\cW^s_{out}$, $\cW^u_{out}$, $\varphi_1^{-1}(\cW^s_{in})$ and $\varphi_1^{-1}(\cW^u_{in})$.
Notice that  $\cW^u_{out}$ and  $\varphi_1^{-1}(\cW^u_{in})$ coincide in a neighborhood $\cO_1$ of $\cL^u\cap\varphi_1^{-1}(\cL^s)$, where the foliations
$\cW^s_{out}$ and  $\varphi_1^{-1}(\cW^s_{in})$ are transverse. We cover $\cL^u\cap\varphi_1^{-1}(\cL^s)$ by a family of disjoint discs $\De_j$ contained in $\cO_1$,
and with boundary disjoint from $\cL^u\cap\varphi_1^{-1}(\cL^s)$.
We apply Lemma~\ref{l.transversefoliation} with $\cF=\cW^u_{out} =\varphi_1^{-1}(\cW^u_{in})$, $\cG=\cW^s_{out}$ and
$\cH=\varphi_1^{-1}(\cW^s_{in})$.  This provides a diffeomorphism denoted $(\psi^{out})^{-1}$,
supported in the union of the interiors of the discs $\De_j$ keeping invariant each leaf of $\varphi_1^{-1}(\cW^u_{in})$ and sending
$\varphi_1^{-1}(\cW^s_{in})$ on $\cW^s_{out}$, in a small neighborhood of $(\psi^{out})^{-1}(\varphi_1^{-1})(\cL^s)\cap \cL^u$.
Notice that $(\psi^{out})^{-1}(\varphi_1^{-1}(\cW^s_{in}))=\varphi_1^{-1}(\cW^s_{in})=\cW^u_{out}$.
The desired diffeomorphism is $\varphi:=\varphi_1\circ \psi^{out}=\psi^{in}\circ \varphi_0\circ\psi^{out}$. 
\end{proof}

\subsection{Adapted neighborhoods}

In the reminder of the section, we consider an affine saddle plug $(V,X)$, a strongly transverse gluing map $\varphi:\partial^{out} V\to\partial^{in} V$, an invariant neighborhood $\cU$ of the with maximal invariant set $\La$ of $(V,X)$, and some transverse foliations $\cF^s$ and $\cF^u$ on $\cU$ containing respectively $W^s(\La)$ and $W^u(\La)$ as sublaminations. We set $\cU^{in}:=\cU\cap \partial^{in}V$ and $\cU^{out}:=\cU\cap \partial^{out}V$. We denote by $\cW^s_{in}$, $\cW^u_{in}$, $\cW^s_{out}$ and $\cW^u_{out}$ the foliations induced by on $\cF^s$ and $\cF^u$ on  $\cU^{in}$ and $\cU^{out}$. We assume that these objects satisfy the conclusion of Proposition~\ref{p.lineargluing}, i.e. we assume that the foliations $\varphi_*(\cW^s_{out})$ and $\varphi_*(\cW^u_{out})$ coincide with $\cW^s_{in}$ and $\cW^u_{in}$ on $\varphi(\cU^{out})\cap \cU^{in}$.
Recall that $\cU^{in}:=\cU\cap \partial^{in}V$ is a neighborhood of $\cL^s$. Hence $\cU^{in}$ contains all but finitely many elements of the connected components of $\partial^{in} V\setminus\cL^s$. 

\begin{defis}
A \emph{$in$-square} is a disc  $C\subset \partial^{in} V$  with the following properties
\begin{itemize}
\item the  boundary $\partial C$ is contained in $\cU^{in}$,
\item $\partial C$ consists in exactly four segments: two segments of leaves of $\cW^s_{in}$ and two segments of leaves of $\cW^u_{in}$,
\item there is a diffeomorphism from $C$ to $[0,1]^2$ so that, on a neighbourhood of the boundary of $C$, this diffeomorphisms maps $\cW^s_{in}$ and $\cW^u_{in}$ on the horizontal and vertical foliations of $[0,1]^2$.
\end{itemize}
A compact neighborhood $\cU^{in}_1\subset \cU^{in}$ of $\cL^s$ will be called an \emph{adapted neighborhood} of $\cL^s$ if, for any connected component $C$ of the complement of $\cL^s$, the complement $C\setminus \cU^{in}_1$ is either empty or is the interior of a  $in$-square. 
\end{defis}

As $\cL^s$ is a filling MS-lamination one easily proves:

\begin{lemm}
\label{l.adapted} 
The lamination $\cL^s$ admits a basis of adapted neighborhoods: every neighborhood of $\cL^s$ in $\partial^{in} V$ contains an adapted neighborhood.
\end{lemm}

One defines analogously adapted neighborhoods of $\cL^u$, and proves that $\cL^u$ admits a basis of adapted neighborhoods.

\subsection{The crossing map}

Recall that, for every point $x$ in $\partial^{in}V\setminus \cL^s$, the positive orbit of $x$ exits $V$ at a point $\Ga(x)$ of $\partial^{out} V\setminus \cL^u$. The map
 $\Gamma\colon \partial^{in}V\setminus \cL^s\to \partial^{out}V\setminus \cL^u$ is a diffeomorphism called the \emph{crossing map} of $(V,X)$ (see subsection~\ref{ss.crossing-map}). The invariance (under the flow of $X$) of the neighbourhood $\cU$ and the foliations $\cF^s$ and $\cF^u$ imply that 
 $$\Gamma(\cU^{in}\setminus \cL^s)=\cU^{out}\setminus \cL^u\quad\quad \Gamma_*(W^s_{in})=W^s_{out}\quad\quad\Gamma_*(W^u_{in})=W^u_{out}.$$
One easily deduces the next two lemmas:

\begin{lemm} 
\label{l.adapted-neighbourhood-invariant}
If $\cV^{in}\subset\cU^{in}$ is an adapted neighborhood of $\cL^s$, then $\Gamma(\cV^{in})\cup \cL^u$ is an adapted neighborhood of $\cL^u$. 
\end{lemm}

\begin{defi}
An invariant neighborhood $\cV$ of $\La$ will be called \emph{an adapted neighborhood of $\La$} if both $\cV^{in}=\cV\cap\partial^{in} V$ and $\cV^{out}=\cV\cap \partial^{out}V$ are adapted neighborhoods of $\cL^s$ and $\cL^u$ respectively.
\end{defi}

\begin{lemm}
\label{l.filling} 
Let  $\cG^s_{in}$ and $\cG^u_{in}$ be smooth transverse foliations on $\partial^{in} V$, which coincide respectively with $W^s_{in}$ and $W^u_{in}$ on an adapted neighborhood of $\cL^s$. Then $\cG^s_{in}$ and $\cG^u_{in}$ extend to smooth transverse $X$-invariant foliations $\cG^s$ and $\cG^u$ in $V$, which coincide with  $\cF^s$ and $\cF^u$ respectively on a neighborhood of $\La$. As a consequence, $\Ga_*(\cG^s_{in})$ and $\Ga_*(\cG^u_{in})$ extend on $\partial^{out} V$ to smooth transverse foliations $\cG^s_{out}$ and $\cG^u_{out}$ which coincide with $W^s_{out}$ and $W^u_{out}$ respectively on an adapted neighborhood of $\cL^u$. 
\end{lemm}

The following lemma shows that the crossing map is `` as  strongly hyperbolic as we want''  in small neighborhoods of $\cL^s$:

\begin{lemm}\label{l.stronghyperbolicity} 
Given any $\lambda>1$, there is an adapted neighborhood $\cU^{in}_\lambda$ of $\cL^s$
so that, for $x\in\cU^{in}_\lambda\setminus\cL^s$ the crossing map $\Gamma$ expands vectors tangent to the leaves of $\cW^u_{in}$ by
larger than $\lambda$, and contracts vectors tangent to the leaves of $\cW^s_{in}$ by a factor smaller than $\lambda^{-1}$.
\end{lemm}

\begin{proof}
It is a direct consequence of Lemma~\ref{l.stablefoliation}. Let us just recall the idea.
For  every point $x\in \partial^{in} U\setminus \cL^s$  close enough to  $\cL^s$,
the positive orbit goes in finite time in a small neighborhood of the hyperbolic set $\La_X$, then
spends an arbitrarily large interval of time close to $\La_X$,
and then reaches $\partial^{out} U$ in a finite time.  Therefore it is enough to choose the adapted
neighborhood $\cU^{in}_\lambda$ small enough for getting the desired strength of hyperbolicity for the
crossing map.
\end{proof}

\subsection{Modifying the gluing map to get some tranversality}

Let $\cG^s$ and $\cG^u$ be a choice of smooth transverse $X$-invariant foliations given by Lemma~\ref{l.filling}. We denote by $\cG^s_{in},\cG^u_{in},\cG^s_{out},\cG^u_{out}$  the one-dimensional foliations induced by $\cG^s$ and $\cG^u$ on $\partial^{in} V$ and $\partial^{out} V$ respectively. According to Proposition~\ref{p.lineargluing}, we can (and we do) assume that there exists an invariant neighbourhood $\cU$ of $\La$ so that $\varphi_*(\cG^s_{out})$ and $\varphi_*(\cG^u_{out})$ coincide with $\cG^s_{in}$ and $\cG^u_{in}$ on $\varphi(\cU^{out})\cap \cU^{in}$. Up to shrinking $\cU$, we may (and we do) assume that $\cU$ is an adapted neighbourhood.

\begin{lemm}\label{l.transverse}  
There is a map $\varphi_1:\partial^{out} V\to\partial^{in} V$, isotopic to $\varphi$ and coinciding with $\varphi$ on an adapted neighborhood of the exit lamination $\cL^u$, so that $(\varphi_1)_*(\cG^u_{out})$ is tranvsverse to $ \cG^s_{in}$.
\end{lemm}

\begin{rema}
The map $\varphi_1$ is a strongly transverse gluing map since it coincides with $\varphi$ on an adapted neighborhood of  $\cL^u$. Moreover, $\varphi_1$ is isotopic to $\varphi$ inside the set of strongly transverse gluing maps.
\end{rema}

We start the proof of Lemma~\ref{l.transverse} by a very general lemma:

\begin{lemm}\label{l.square} 
Let $\cF$ and $\cG$ be smooth foliations of the square $C=[0,1]^2$ so that
$[0,1]\times \{0,1\}$ consists in leaves of $\cF$ and $\{0,1\}\times [0,1]$ consists in leaves of $\cG$, and
$\cF$ and $\cG$ are transverse in  a neighborhood of the boundary $\partial C$.  Then there is a smooth
diffeomorphisms $\psi$ of $C$ equal to the identity map in a neighborhood of $\partial C$ so that
$\psi(\cG)$ is transverse to $\cF$.
\end{lemm}

\begin{proof}
The hypothesis imply that each of the foliations $\cF$ and $\cG$ are smoothly conjugated to trivial foliations, so that
we may assume that
$\cF$ is the horizontal foliation $\{[0,1]\times\{t\}\}_{t\in[0,1]}$.   Now, there is a foliation $\cH$ on $C$, transverse to the leaves of
$\cF$ and coinciding with $\cG$ in a neighborhood of $\partial C$ and having the same holonomy from $\{0\}\times [0,1]\to \{1\}\times [0,1]$ as $\cG$.

The foliations $\cG$ and $\cH$ are smoothly conjugated by a diffeomorphism $\psi$
which coincide with the identity close to $\partial C$, concluding.
 \end{proof}

\begin{proof}[Proof of Lemma~\ref{l.transverse}] 
The proof consist in $3$ steps:

\begin{claim} 
There is a diffeomorphism $\psi^{in}\colon \partial^{in} V\to \partial^{in} V$ which coincides with the identity map on a
neighborhood of $\cL^s$, and so that $(\psi^{in}\circ\varphi)_*(\cL^u)$ is transverse to $\cG^s_{in}$.
\end{claim}

\begin{proof}
We have assumed that there exists an adapted neighbourhood $\cU$ of $\Lambda$ so that $\varphi_*(\cG^s_{out})$ and $\varphi_*(\cG^u_{out})$ coincide with $\cG^s_{in}$ and $\cG^u_{in}$ on $\varphi(\cU^{out})\cap \cU^{in}$. Hence, $\varphi(\cL^u)$ is already transverse to $\cG^s_{in}$ on $\cU^{in}$. So it is  enough to consider a connected component $R$ of $\partial^{in} V\setminus\cU^{in}$. Since $\cU^{in}$ is an adapted neighbourhood of $\cL^s$, $R$ is an $in$-rectangle~: the restrictions to $R$ of $\cG^s_{in}$ and $\cG^u_{in}$ are the trivial horizontal and vertical foliations of $R$. Moreover $\varphi_*(\cL^u)\cap R$ is a lamination coinciding with $\cG^u_{in}$ in a neighbourhood of $\partial R$,  and each leaf of $\varphi_*(\cL^u)\cap R$ is a segment joining the bottom horizontal segment to the top horizontal segment of $\partial R$. A similar proof to the one of Lemma~\ref{l.transverse} proves the existence of a diffeomorphism $\psi^{in}_R$  equal to the identity in a neighborhood of $\partial R$ and  so that $(\psi^{in}_R\circ\varphi)_*(\cL^u)\cap R)$
is transverse in $R$ to $\cG^s_{in}$.  The announced diffeomorphism $\psi^{in}$ is the product of the diffeomorphisms $\psi^{in}_{R_1},\dots,\psi^{in}_{R_n}$ associated to the connected components $R_1,\dots,R_n$ of $\partial^{in} V\setminus\cU^{in}$. 
\end{proof}

Let $\varphi_1:=\psi^{in}\circ \varphi$. Notice that $\varphi_1$ is isotopic to $\varphi$ through strongly transverse gluing diffeomorphisms. Moreover, there is an adapted neighbourhood $\cV\subset\cU$ of $\La$ so that $(\varphi_1)*(\cG^s_{out})$ and $(\varphi_1)_*(\cG^u_{out})$ coincide with $\cG^s_{in}$ and $\cG^u_{in}$ on  $\varphi_1(\cV^{out})\cap \cV^{in}$.

\begin{claim} 
There is $\psi^{out}\colon \partial^{out} V\to \partial^{out} V$ which is the identity map in a  neighborhood of $\cL^u$, and so that $(\varphi_1\circ \psi^{out})^{-1}_*(\cL^s)$ is transverse to $\cG^u_{out}$.
\end{claim}

\begin{proof}
The proof is identical to the one of the first claim, reversing the flow of $X$.
\end{proof}

Now we set $\varphi_2:= \varphi_1\circ\psi^{out}$. Then  $\varphi_2$ is isotopic to $\varphi$ through strongly transverse diffeomorphisms, $(\varphi_2)_*(\cL^u)=(\varphi_1)_*(\cL^u)$ is transverse to $\cG^s_{in}$,  and $(\varphi_2)_*(\cG^u_{out})$ is transverse to $\cL^s_{in}$.  
So $\cG^s_{in}$ and $\varphi_2(\cG^u_{out})$ may fail to be transverse only in the interior of connected component
of $\partial^{in} V\setminus \cL^s\cup (\varphi_2)_*(\cL^u)$.  As $\cL^s$ and $(\psi_2)_*(\cL^u)$ are strongly transverse,
the closure of each component of $\partial^{in}V\setminus( \cL^s\cup(\psi_2)_*(\cL^u))$ is a square having
two sides on leaves of $\cL^s$ and two sides on leaves of $(\psi_2)_*(\cL^u))$.
One concludes by applying  Lemma~\ref{l.square} in each of
these squares with $\cF= \cG^s_{in}$ and $\cG=(\varphi_2)_*(\cG^u_{out})$,
which completes the proof.
\end{proof}

\begin{proof}[Proof of Proposition~\ref{p.preparation}]
It suffices to put together Lemma~\ref{l.affine}, Lemma~\ref{l.linearfoliation}, Proposition~\ref{p.lineargluing} and Lemma~\ref{l.transverse}.
\end{proof}

\begin{defi}
The \emph{return map} $\Theta:\partial^{in} V\to\partial^{out} V$ associated to $(V,X,\varphi_1)$ is obtained by composing the crossing map $\Gamma$ and the gluing map $\varphi_1$~:
$$\Theta:=\varphi_1\circ\Gamma.$$
\end{defi}

Note that, if $\varphi_1$ satisfies the conclusion of Lemma~\ref{l.transverse}, then the foliation $\Theta_*(\cG^u_{in})$ and $\cG^s_{in}$ are tranvserse:
$$\Theta_*(\cG^u_{in})\pitchfork\cG^s_{in}.$$

\section{Perturbation of the return map and proof of Theorem~\ref{t.transitive}}

The aim of this section is to prove of Theorem~\ref{t.transitive} and Proposition~\ref{p.transitive}. 

All over the section, we consider a saddle hyperbolic plug $(V,X)$ with filling MS-laminations, and a strongly transverse gluing map $\varphi:\partial^{out} V\to\partial^{in} V$. We denote by $\Lambda$ the maximal invariant set of $(V,X)$. Since $(V,X)$ is a saddle hyperbolic plug, $\Lambda$ does not contain attractors nor repellors. We denote by $\Gamma:\partial^{in} V\setminus\cL^s\to \partial^{out}V\setminus \cL^u$ the crossing map of $(V,X)$, and by $\Theta:=\varphi\circ\Gamma$ the return map of $X$ on $\partial^{in} V$. 

According to Proposition~\ref{p.preparation}, we may (and we do) assume that $V$ is endowed with a pair of two-dimensional smooth $X$-invariant foliations $\cG^s$ and $\cG^u$, transverse to $\partial V$ and transverse to each other, containing respectively $W^s(\Lambda)$ and $W^u(\Lambda)$ as sublaminations. We denote by $\cG^s_{in},\cG^u_{in},\cG^s_{out},\cG^u_{out}$ the one-dimensional foliations induced by  $\cG^s$ and $\cG^u$ on $\partial^{in} V$ and $\partial^{out} V$ respectively. Recall that the entrance laminaition $\cL^s$ and the exit lamination $\cL^u$ are sublaminations of the $\cG^s_{in}$ and $\cG^u_{out}$ respectively. Again by Proposition~\ref{p.preparation}, we may (and we do) assume that the holonomy of each compact leaf  of  $\cG^s_{in}$ and $\cG^u_{out}$ is conjugated to a homothety, and that $\varphi_{*}(\cG^u_{out})$ is transverse to $\cG^s_{in}$.

\subsection{Reduction of Theorem~\ref{t.transitive} to a perturbation of the return map $\Theta$}

Theorem~\ref{t.transitive} states that there exists a strongly tranverse gluing diffeomorphism $\psi: \partial^{out} V\to\partial^{in} V$, which is isotopic to $\varphi$ through strongly tranvserse gluing diffeomorphisms, and such that the vector field induced by $X$ on the closed manifold $V/\psi$ is Anosov. As stated by the following lemma, proving that $X_\psi$ is Anosov amounts to proving that its first return map $\Theta_\psi:=\psi\circ\Gamma$ is hyperbolic: 

\begin{lemm}
\label{l.hyperboilicity-vector-field-vs-return-map}
Consider a strongly transverse gluing map $\psi:\partial^{out}V\to\partial^{in} V$. Denote by $Z_\psi$ the vector field induced by $X$ on the closed manifold $V/\psi$, and by  $\Theta_\psi:=\psi\circ\Gamma$ the return map of $Z_\psi$ on $\partial^{in} V$. Assume that there exists two continuous cone fields $\cC^s_{in}$ and $\cC^u_{in}$ on $\partial^{in} V$ so that:
\begin{itemize}
\item the cone fields $\cC^u_{in}$ and $\cC^s_{in}$ are invariant under $d\Theta_\psi$ and $d\Theta_\psi^{-1}$ respectively, and the vector in $\cC^u_{in}$ and $\cC^s_{in}$ are uniformly expanded by  $d\Theta_\psi$ and $d\Theta_\psi^{-1}$ respectively (for some riemannian metric);
\item $\cC^u_{in}$ contains the direction tangent to $\cG^u_{in}$ and the direction tangent to $\psi_*(\cG^u_{out})$, but it does contain neither the direction tangent to $\cG^s_{in}$ nor the  the direction tangent to $\psi_*(\cG^s_{out})$; 
\item $\cC^s_{in}$ contains the direction tangent to $\cG^s_{in}$ and the direction tangent to $\psi_*(\cG^s_{out})$, but it does contain neither the direction tangent to $\cG^u_{in}$ nor the  the direction tangent to $\psi_*(\cG^u_{out})$.
\end{itemize}
Then the vector field induced $Z_\psi$ is Anosov. 
\end{lemm}

\begin{proof}
By assumption, the maximal invariant set $\Lambda$ of $(V,X)$ does not contain attractors, nor repellors. Therefore, $\Lambda$ is transversally totally discontinuous, and we may consider a local section $\Sigma$ of $\Lambda$. By such, we mean that $\Sigma$ is a collection of closed topological discs, $\Sigma$ is tranvserse to $X$ (or equivalently, $Z_\psi$), the boundary of $\Sigma$ is disjoint from $\Lambda$, and the interior of $\Sigma$ intersects every orbit of $X$ in $\Lambda$. We denote by $f$ the first return map of the orbit of $X$ on $\Sigma$. We denote by $\cG^s_\Sigma$ and $\cG^u_\Sigma$ the 1-dimensional foliations induced by $\cG^s$ and $\cG^u$ on $\Sigma$. Note that $\Sigma$ can be chosen so that it is contained in an arbitrarily small neighbourhood of $\Lambda$. 

Every orbit of $Z_\psi$ either is contained in $\Lambda$, or intersects $\partial^{in} V$. Therefore, the interior of $\Sigma\cup\partial^{in} V$ intersects every orbit of $Z_\psi$. We denote by $f_\psi$ the first return map of the vector field $Z_\psi$ on $\Sigma\cup\partial^{in} V$. By classical elementary arguments, proving that the vector field $Z_\psi$ is hyperbolic (\emph{i.e.} Anosov) amounts to proving that $f_\psi$ is hyperbolic. In order to prove that $f_\psi$ is indeed hyperbolic, we will construct some cone fields $\cC^s$ and $\cC^u$ on $\mathrm{int}(\Sigma)\cup\partial^{in} V$, prove that $\cC^u$ and $\cC^s$ are invariant under $df_\psi$ and $df_\psi^{-1}$ respectively, and that the vectors in $\cC^u$ and $\cC^s$ and uniformly expanded by $df_\psi$ and $df_\psi^{-1}$ respectively.

By assumption, we already have some cone fields $\cC^u_{in}$ and $\cC^u_{in}$ on $\partial^{in} V$. Moreover, $\Lambda$ is a hyperbolic set for $X$; hence there exists some cone fields  $\cC^u_{\Sigma}$ and $\cC^s_\Sigma$ on $\mathrm{int}(\Sigma)$ which are invariant under $df$ and $df^{-1}$ respectively, and so that the vectors in $\cC^u_{\Sigma}$ and $\cC^s_\Sigma$ and uniformly expanded by $df$ and $df^{-1}$ respectively. We may assume that these cone fields  $\cC^u_{\Sigma}$ and $\cC^s_\Sigma$ respectively contain the directions tangent to the foliations induced by $\cG^u$ and $\cG^s$ on $\Sigma$. We consider the cone fields $\cC^u$ and $\cC^s$ on $\mathrm{int}(\Sigma)\cup\partial^{in} V$, which coincide with  $\cC^u_{in}$ and $\cC^u_{in}$ on $\partial^{in} V$ and coincide with  $\cC^u_{\Sigma}$ and $\cC^s_\Sigma$ on $\mathrm{int}(\Sigma)$. In order to check that these cone fields satisfy the desired properties, we will decompose the first return map $f_\psi$ into four parts. Namely, we consider the restrictions of $f_\psi$ to $\Sigma\cap f_\psi^{-1}(\Sigma)$, $\partial^{in} V\cap f_\psi^{-1}(\partial^{in} V)$, $\partial^{in} V\cap f_\psi^{-1}(\Sigma)$ and $\partial^{in} V\cap f_\psi^{-1}(\Sigma)$. We denote these restrictions by  $f_{\psi,1}$, $f_{\psi,2}$, $f_{\psi,3}$ and $f_{\psi,4}$ respectively. 

\begin{itemize}
\item The map $f_{\psi,1}:\Sigma\to\Sigma$ is nothing but the first return map $f$ of the orbits of $X$ on $\Sigma$ (because a segment of orbit of $Z_\psi$ which does not cross $\partial^{in} V$ is a segment of orbit of $\Sigma$). Hence, $\cC^u_{\Sigma}$ and $\cC^s_{\Sigma}$ are invariant respectively under $df_{\psi,1}$ and $df_{\psi,1}^{-1}$, and so that the vectors in $\cC^u_{\Sigma}$ and $\cC^s_{\Sigma}$ and uniformly expanded respectively by $df_{\psi,1}$ and $df_{\psi,1}^{-1}$.
\item The map $f_{\psi,2}:\partial^{in} V\to\partial^{in} V$ is a restriction of the return map $\Theta_\psi$ (namely, the restriction to the set of points $x$ such that the forward $Z_\psi$-orbit of $x$ intersects $\partial^{in} V$ before intersecting $\Sigma$). Hence our assumption ensures that $\cC^u_{in}$ and $\cC^s_{in}$ are invariant under $df_{\psi,2}$ and $df_{\psi,2}^{-1}$ respectively, and so that the vectors in $\cC^u_{in}$ and $\cC^s_{in}$ and uniformly expanded  by $df_{\psi,2}$ and $df_{\psi,2}^{-1}$ respectively.
\item Consider a neighbourhood $U$ of $\Lambda$. Suppose that $\Sigma$ is contained in a neighbourhood $V$ of $\Lambda$ so that $V$ is much smaller than $U$. Then, a segment of orbit of $X$ starting at some point of $\partial^{in} V$ and ending at some point of $\Sigma$ will have to spend a very long time in $U$ before hitting $\Sigma$. Hence (using the same arguments as in the proof of Lemma~\ref{l.stronghyperbolicity}), the vectors tangent to $\cG^u_{in}$ (resp. $\cG^s_{in}$) will be expanded (resp. contracted) by a very large factor along this segment of orbit. According to our assumptions, the cone fields $\cC^u_{in}$ contains the direction tangent to $\cG^u_{in}$ but does not contain the direction tangent to $\cG^s_{in}$. Therefore, provided that the section $\Sigma$ is contained in a small enough neighbourhood of $\Lambda$, the derivative of $f_{\psi,3}:\partial^{in} V\to\Sigma$ will map the cone field $\cC^u_{in}$ to an arbitrarily thin cone field around $\cG^u_{\Sigma}$ (in particular, the image of $\cC^u_{in}$ will be contained in $\cC^u_{\Sigma}$) and will expand uniformly the vectors in $\cC^u_{in}$. Similarly, the derivative of $f_{\psi,3}^{-1}$ will map $\cC^s_{\Sigma}$ inside $\cC^s_{in}$, and will expand uniformly the vectors in $\cC^s_\Sigma$.
\item Similar arguments show that, provided that the section $\Sigma$ is contained in a small enough neighbourhood of $\Lambda$, the cone fields satisfy the desired properties with respect to the map $f_{\psi,4}:\Sigma\to\partial^{in} V$ (here we use the fact that $\cC^u_{in}$ and $\cC^s_{in}$ respectively contain the directions tangent to $\psi(\cG^u_{out})$ and $\psi(\cG^s_{out})$, but do not contain the directions tangent to $\psi(\cG^s_{out})$ and $\psi(\cG^u_{out})$.
\end{itemize}
The four points above show that, if $\Sigma$ is contained in a small enough neighbourhood of $\Lambda$, then the cone fields $\cC^u$ and $\cC^s$ are invariant  under $df_{\psi}$ and $df_{\psi}^{-1}$ respectively, and so that the vectors in $\cC^u$ and $\cC^s$ and uniformly expanded by $df_{\psi}$ and $df_{\psi}^{-1}$ respectively. In other words, the first return map $f$ is hyperbolic provided that $\Sigma$ is small enough. Hence the vector field $Z_\psi$ is Anosov.
\end{proof}

The gluing map $\psi:\partial^{out} V\to\partial^{in} V$ (whose existence is claimed by theorem~\ref{t.transitive}) will be obtained as a composition $\psi=\psi^{in}\circ\varphi\circ\psi^{out}$, where $\varphi$ is the original gluing map, $\psi^{in}$ is a self-diffeomorphism of the entrance boundary $\partial^{in} V$ and $\psi^{out}$ is a self-diffeomorphism of the exit boundary $\partial^{out} V$. The diffeomorphisms $\psi^{in}$ and $\psi^{out}$ will be provided by the following proposition:

\begin{prop}
\label{p.returnmap} 
Given any $\lambda>1$ and $\varepsilon>0$, there exists a diffeomorphism $\psi^{in}\colon \partial^{in} V\to \partial^{in} V$  with the following properties:
\begin{itemize}
\item $\psi^{in}$ coincides with the identity map on a neighborhood of the lamination $\cL^s$
\item $\psi^{in}$ preserves each leaf of the foliation $\cG^u_{in}$
\item the foliation $(\psi^{in})^{-1}_*(\cG^s_{in})$  is $\varepsilon$-$C^1$-close to the foliation $\cG^s_{in}$
\item the derivative of $\Ga\circ\psi^{in}$ expands vectors tangent to $\cG^u_{in}$ by a factor larger than $\lambda$: for any vector $u$ tangent to a leaf of $\cG^u_{in}$, one has $\|(\Ga\circ\psi^{in})_*(u)\|>\lambda\|u\|$.
\end{itemize}
Analogously, there exists a diffeomorphism $\psi^{out}\colon \partial^{out}V\to \partial^{out} V$ so that
\begin{itemize}
\item $\psi^{out}$ coincide with the identity map on a neighborhood of the lamination $\cL^u$
\item $\psi^{out}$ preserves each leaf of the foliation $\cG^s_{out}$
\item the foliation $\psi^{out}_*(\cG^u_{out})$  is $\varepsilon$-$C^1$-close to the foliation $\cG^u_{out}$
\item the derivative of $(\Ga\circ\psi^{in})^{-1}$ expands vectors tangent to $\cG^s_{out}$ by a factor larger than $\lambda$: for any unit vector $u$ tangent to a leaf of $\cG^s_{out}$, one has $\|(\Ga\circ\psi^{out})^{-1}_*(u)\|>\lambda.\|u\|$.
\end{itemize}
\end{prop}

\begin{proof}[Proof of  Theorem~\ref{t.transitive} assuming Proposition~\ref{p.returnmap}]
Given $\lambda>1$ and $\varepsilon>0$, we consider the diffeomorphisms $\psi^{in}_{\lambda,\varepsilon}$ and $\psi^{out}_{\lambda,\varepsilon}$ associated to $\lambda,\varepsilon$ by Proposition~\ref{p.returnmap}. Then we consider the gluing map
$$\psi_{\lambda,\varepsilon}:=\psi^{in}_{\lambda,\varepsilon}\circ\varphi\circ\psi^{out}_{\lambda,\varepsilon}$$
the vector field $Z_{\lambda,\varepsilon}$ induced by $X$ on the closed manifold $V/\psi_{\lambda,\varepsilon}$. Obverse that, since $\psi^{in}$ and $\psi^{out}$ coincide with the identity on neighbourhoods of $\cL^s$ and $\cL^u$ respectively, $\psi^{in}_{\lambda,\varepsilon}$ is a strongly tranverse gluing map which is isotopic to $\varphi$ inside the set of strongly transverse gluing maps. We want to prove that the vector field $Z_{\lambda,\varepsilon}$ is Anosov, provided that $\varepsilon$ is small enough and $\lambda$ is large enough. So we are left to proving that the return
$$\Theta_{\lambda,\varepsilon}:=\psi_{\lambda,\varepsilon}\circ\Gamma$$
satisfies the hypotheses of Lemma~\ref{l.hyperboilicity-vector-field-vs-return-map}, provided that $\varepsilon$ is small enough and $\lambda$ is large enough.
For technical reasons, it is convenient to introduce  the map 
$$\widehat\Theta_{\lambda,\varepsilon}:=(\psi^{in})^{-1}\circ\Theta_{\lambda,\varepsilon}\circ\psi^{in}=\varphi\circ\psi^{out}\circ\Gamma\circ\psi^{in},$$
and the foliations 
$$\cG^s_{in,\lambda,\varepsilon}:=(\psi_{\lambda,\varepsilon}^{in})^{-1}_*(\cG^s_{in})\quad\mbox{ and }\quad\cG^u_{out,\lambda,\varepsilon}=(\psi_{\lambda,\varepsilon}^{out})_*(\cG^u_{out}).$$
Note that
$$(\psi_{\lambda,\varepsilon}^{in})^{-1}_*(\cG^u_{in})= \cG^u_{in} \quad\mbox{and}\quad(\psi_{\lambda,\varepsilon}^{out})_*(\cG^s_{out}) = \cG^s_{out}$$
since $\psi^{in}_{\lambda,\varepsilon}$ preserves the foliation $\cG^u_{in}$ and $\psi^{out}_{\lambda,\varepsilon}$ preserves the foliation $\cG^u_{out}$ by assumption.

The diffeomorphism $\psi_{\lambda,\varepsilon}^{out}$ preserves each leaf of $\cG^s_{out}$. We denote by $h_{out,\lambda,\varepsilon}^s:\partial^{out} V\to\partial^{out} V$ the holonomy of the foliation $\cG^s_{out}$ between $x$ and $\psi_{\lambda,\varepsilon}^{out}(x)$. 

\begin{clai}
\label{c.holonomy-bounded}
The length of the segment of leaf of $\cG^s_{out}$ joining a point $x$ to $\psi_{\lambda,\varepsilon}^{out}(x)$ is bounded by a constant is independent of
$\varepsilon$, $\lambda$ and $x$. As a consequence, the action of the holonomy $h^s_{out,\lambda,\varepsilon}$ on vectors tangent to $\cG^u_{out}$ is uniformly
 bounded independently of $\varepsilon$ and $\lambda$.
\end{clai}

\begin{proof} 
Recall that $\cL^u$ is a filling MS-lamination, $\cL^u$ is a sub-lamination of the foliation $\cG^u_{out}$, and $\psi^{out}_{\lambda,\varepsilon}$ is the identity 
map in a neighborhood $\cL^u$. Therefore, if $x$ is not a fixed point of $\psi^{out}_{\lambda,\varepsilon}$, then the points $x$ and $\psi^{out}_{\lambda,\varepsilon}(x)$ 
belong to the same segment of leaf of  $\cG^s_{out}\setminus \cL^u$.  As $\cL^u$ is a filling MS-lamination, these segments have a uniformly bounded length, proving the first assertion. The second assertion is a direct consequence from the first one and the fact that $\cG^s_{out}$ is a smooth foliation.
\end{proof}

\begin{clai} 
The  perturbed foliations $\cG^s_{in,\lambda,\varepsilon}$ and $\cG^u_{out,\lambda,\varepsilon}$ tends to the non-perturbed foliations $\cG^s_{in}$ and 
$\cG^u_{out}$ when $\varepsilon\to 0$ for the $C^1$-topology. As a consequence, for $\varepsilon$ small enough, the foliation $(\widehat\Theta_{\lambda,\varepsilon})_*\varepsilon( \cG^u_{in})=\varphi_*(\cG^u_{out,\lambda,\varepsilon})$ is uniformly (in $\lambda,\varepsilon$) transverse to  $\cG^s_{in,\lambda,\varepsilon}$.
\end{clai}

\begin{proof}
The first assertion is follows immediately from the definition of the foliations $\cG^s_{in,\lambda,\varepsilon}$ and $\cG^u_{out,\lambda,\varepsilon}$, and from the properties of the   maps $\psi^{in}$ and $\psi^{out}$. The second assertion is an direct consequence of the first one.
\end{proof}

\begin{claim}
\label{l.uniformexpansion}
There exist some constants $C>0$ and $\varepsilon_0>0$ such that, for  any $\lambda>1$ and $0<\varepsilon<\varepsilon_0$, the return map $\widehat\Theta_{\lambda,\varepsilon}$ expands uniformly the vectors tangent to $\cG^u_{in}$ by a factor larger than $C\lambda$, and its inverse $(\widehat\Theta_{\lambda,\varepsilon})^{-1}$ expands uniformly the vectors tangent to $\cG^s_{in,\lambda,\varepsilon}$ by a factor larger than $C\lambda$. 
\end{claim}

\begin{proof} 
Let $u$ be a vector tangent to $\cG^u_{in}=\cG^u_{in,\lambda,\varepsilon}$. One writes
$$\widehat\Theta_{\lambda,\varepsilon, *}(u)= \varphi_*\circ (\psi^{out}_{\lambda,\varepsilon})_*\circ \Ga_*\circ(\psi_{\lambda,\varepsilon}^{in})_*(u).$$
See figure~\ref{f.end-proof-main-theorem}. By definition of the map $\psi_{\lambda,\varepsilon,*}^{in}$, one has
$$\|\Ga_*\circ(\psi_{\lambda,\varepsilon}^{in})_*(u)\|>\lambda.\|u\|.$$
Furthermore, $\Ga_*\circ(\psi_{\lambda,\varepsilon}^{in})_*(u)$ is tangent to $\cG^{u}_{out}$. Thus we just need to see that the action of $\psi_{\lambda,\varepsilon}^{out}$ on the vectors tangent to $\cG^{u}_{out}$ is bounded, independently of $\lambda$ and $\varepsilon$ (for $\varepsilon$ smaller than some $\varepsilon_0$).

Recall that $h_{out,\lambda,\varepsilon}^s$  is the holonomy of the foliation $\cG^s_{out}$ between $x$ and $\psi_{\lambda,\varepsilon}^{out}(x)$. The foliation $\cG^u_{out}$ is transverse to $\cG^s_{out}$, and $\psi_{\lambda,\varepsilon}^{out}(\cG^u_{out})$ is $\varepsilon$-$C^1$-close to $\cG^u_{out}$. In particular, there exists $\varepsilon_0$ so that, for $0<\varepsilon<\varepsilon_0$, the foliation $(\psi_{\lambda,\varepsilon}^{out})_*(\cG^u_{out})$ is uniformly (in $\lambda$ and $\varepsilon$) transverse to the foliation $\cG^s_{out}$. Therefore, for $v$ tangent to $\cG^u_{out}$, the ratio between $\|(h^s_{out,\lambda,\varepsilon})_*(v)\|$ and $\|(\psi_{\lambda,\varepsilon}^{out})_*(v)\|$ is  bounded  independantly of $\lambda$, $\varepsilon$ and $v$. Using the Claim~\ref{c.holonomy-bounded}, we conclude that $\widehat\Theta_{\lambda,\varepsilon}$ expands uniformly the vectors tangent to $\cG^u_{in}$ by a factor larger than $C\lambda$. See figure~\ref{f.end-proof-main-theorem}. The arguments are similar for the action of the map $(\widehat\Theta_{\lambda,\varepsilon})^{-1}$ on the vectors tangent to $\cG^s_{in,\lambda,\varepsilon}$. 
\end{proof}

\begin{figure}[htp]
\begin{center}
  \includegraphics[totalheight=6cm]{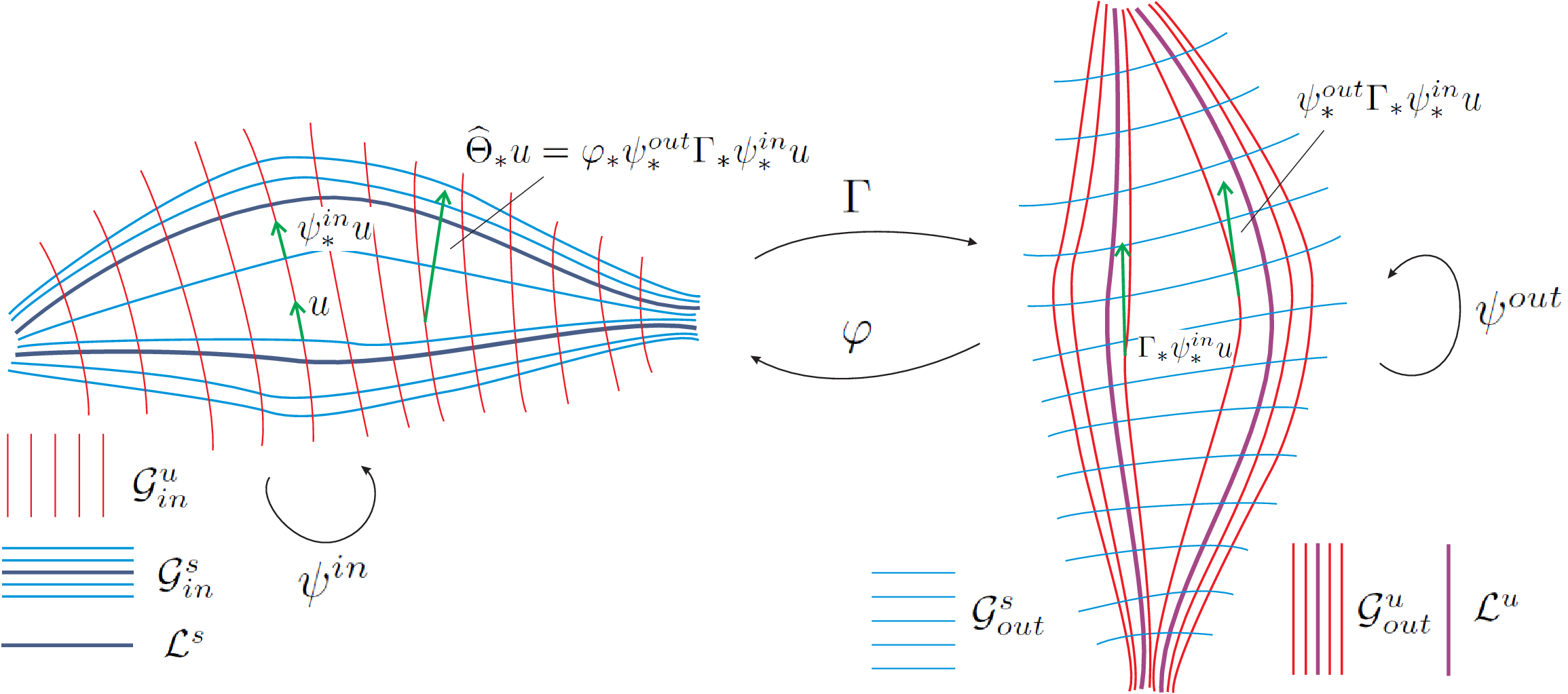}\\
\caption{\label{f.end-proof-main-theorem}The action of $\widehat\Theta_{\lambda,\varepsilon}$ on a vector $u$ tangent  to $\cG^u_{in}$.}
  \end{center}
\end{figure}

During the proof of claim~\ref{l.uniformexpansion}, we have chosen $\varepsilon_0$ so that the foliations $\varphi_*(\cG^u_{out,\lambda,\varepsilon})$ are uniformly transverse to $\cG^s_{in,\lambda,\varepsilon}$ for $\varepsilon\in [0,\varepsilon_0]$ and $\lambda>1$. Moreover, the foliation $\varphi_*(\cG^s_{out})$ is transverse to the foliation $\cG^u_{in}$. This allows us to choose a continuous conefield $\widehat\cC^u_{in}$ on $\partial^{in}V$ so that $\widehat\cC^u_{in}$ contains the direction tangent to the foliation $\cG^u_{in}$ and the direction tangent to of the foliation $\varphi_*(\cG^u_{out,\lambda,\varepsilon})$ for $\varepsilon\in [0,\varepsilon_0]$ and  $\lambda>1$, and so that $\widehat\cC^u_{in}$ does not contain neither the direction tangent to the foliation $\cG^s_{in,\lambda,\varepsilon}$ for any $\varepsilon\in [0,\varepsilon_0]$ and  $\lambda>1$ nor the direction tangent to the foliation $\varphi_*(\cG^s_{out})$.  As a direct consequence of Lemma~\ref{l.uniformexpansion} one gets:

\begin{fact} 
\label{f.cone-1}
For every $\varepsilon\in [0,\varepsilon_0]$, when $\lambda\to\infty$, the conefield  $\widehat\cC^u_{in}$ is mapped by $ \Ga_{\lambda,\varepsilon}$  in an arbitrarily small cone
field around $\cG^u_{out,\lambda,\varepsilon}$ and the vectors in $\widehat\cC^u_{in}$ is expanded by an arbitrarily large factor. As a consequence, there exists $\lambda_0$ such that for every $\varepsilon\in [0,\varepsilon_0]$ and $\lambda\geq \lambda_0$, the conefield $\widehat\cC^u_{in}$ is strictly invariant by $\widehat\Theta_{\lambda,\varepsilon}$ and the vectors in that conefield are uniformly expanded by $\widehat\Theta_{\lambda,\varepsilon}$.
\end{fact}

From now on, we fix $\varepsilon\in [0,\varepsilon_0]$ and $\lambda\geq\lambda_0$. We consider the cone field $\cC^u_{in}:=\psi^{in}_{\lambda,\varepsilon}(\widehat\cC^u_{in})$. This cone fields contains the direction tangent to the foliations 
$$(\psi^{in}_{\lambda,\varepsilon})^{-1}_*(\cG^u_{in})=\cG^u_{in}\quad \mbox{ and }\quad (\psi^{in}_{\lambda,\varepsilon})^{-1}_*\varphi_*(\cG^u_{out,\lambda,\varepsilon})=(\psi_{\lambda,\varepsilon})_*(\cG^{u}_{out}),$$ 
and it does not contain the direction tangent to the foliations 
$$(\psi^{in}_{\lambda,\varepsilon})^{-1}_*(\cG^s_{in,\lambda,\varepsilon})=\cG^s_{in} \quad\mbox{ nor }\quad(\psi^{in}_{\lambda,\varepsilon})^{-1}_*\circ\varphi_*(\cG^s_{out})=(\psi_{\lambda,\varepsilon})_*(\cG^{s}_{out}).$$ 
Moreover, fact~\ref{f.cone-1} implies that the conefield $\cC^u_{in}$ is strictly invariant by $\Theta_{\lambda,\varepsilon}$ and that the vectors in $\cC^u_{in}$ are uniformly expanded by $\Theta_{\lambda,\varepsilon}$ (for the norm associated to the pullback under $\psi^{in}_{\lambda,\varepsilon}$ of the initial riemannian metric). In other words, the cone field $\cC^u_{in}$ satisfies all the hypotheses of Lemma~\ref{l.hyperboilicity-vector-field-vs-return-map} (for the return map $\Theta_{\lambda,\varepsilon}$). The construction of a cone field  $\cC^s_{in}$ is completely similar. So we can apply Lemma~\ref{l.hyperboilicity-vector-field-vs-return-map}, which shows that the vector field $Z_{\lambda,\varepsilon}$ is Anosov.
\end{proof}

\begin{rema}
\label{r.partial-gluing}
There is a version of Theorem~\ref{t.transitive} where one does not glue the whole exit boundary of a hyperbolic plug on the whole entrance boundary. More precisely, let $(V,X)$ be a saddle hyperbolic plug with filling MS-laminations $(V,X)$. Let $T^{in}$ and $T^{out}$ be unions of connected components of $\partial^{in} V$ and $\partial^{out} V$ respectively. Let  $\varphi:T^{out}\to T^{in}$ be a map so that $\varphi_*(\cL^u\cap T^{out})$ is strongly transverse to $\cL^s\cap T^{in}$. Exactly the same arguments as above allow to prove that there is a vector field $Y$ on $V$ which is $C^1$-close to $X$ and map $\psi:T^{out}\to T^{in}$, so that:
\begin{itemize}
\item $(V,X,\varphi)$ and $(V,Y,\psi)$ are strongly isotopic, 
\item if $Z_\psi$ is the vector field induced by $X$ on $V/\psi$, then $(V/\psi,Z_\psi)$ is a hyperbolic plug. 
\end{itemize}
\end{rema}

\subsection{Perturbation of the return map $\Theta$: proof of Proposition~\ref{p.returnmap}}

This section is devoted to the proof of Proposition~\ref{p.returnmap}. We will only deal with the diffeomorphism $\psi^{in}:\partial^{in} V\to\partial^{in} V$. The construction of the diffeomorphism$\psi^{out}$ is analogous, up to reversing flow).  

Let us briefly present the construction of $\psi^{in}$.  Fix $\varepsilon>0$ and $\lambda>1$. According to  Lemma~\ref{l.stronghyperbolicity}, the crossing map $\Gamma$ expands the vectors tangent to $\cG^u_{in}$ by a factor at least $\lambda$ on some neighborhood of $\cL^s$. The image by $\Ga$ of such neighborhood is a neighborhood of $\cL^u$. Such a neighborhood contains an adapted neighborhood, whose complement consists in finitely many $in$-square (see Lemma~\ref{l.adapted}). Our proof will consist in building the diffeomorphism $\psi^{in}$ in one of these $in$-squares and to extend it on the whole $\partial^{in} V$ by gluing it with the identity map by a bump function, using the fact that the expansion of vectors tangent to $\cG^u_{in}$ is arbitrarily large out of these $in$-squares. 

As $\cL^s$ is a filling MS-lamination, every connected component $B$ of $\partial^{in} V\setminus \cL^s$ is a strip whose accessible boundary consists in two non-compact leaves  of $\cL^s$, which are asymptotic to each other at both ends. Each end of $B$ spirals around a compact leaf of $\cL^s$, with contracting linear holonomy (see figure~\ref{f.strip}). Our construction will be divided in two steps.
\begin{itemize}
 \item We will first build a diffeomorphism $\psi_h$ of $B$,  defined as the  product of a diffeomorphism $h$ of a segment $I^u$ of $\cG^u_{in}$ leaf, by the identity map in the direction of the leaves of $\cG^s_{in}$. The diffeomorphisms $\psi_h$ will have all the announced property, except that it will not coincide with the identity close to the the boundary of $S$, so that it cannot be extended on the whole $\partial^{in} V$.
\item Then, we will ``slow down" the diffeomorphism $\psi_h$ close to the ends of $B$, in order to be able to extend continuously $\psi_h$ on $\partial^{in}V$ (so that the extension of $\psi_h$ will coincide with the identity on the complement of $B$).
\end{itemize}
The main difficulty is to manage to ``slow down" $\psi_h$ without distroying the hyperbolicity. A key ingredient to do that will be the uniform control of distorsion of the holonomies of $\cG^s_{in}$ (this the reason why we need the holonomy of $\cG^s_{in}$ along a compact leaf to be conjugated to a homothety).

\begin{figure}[htp]
\begin{center}
  \includegraphics[totalheight=5cm]{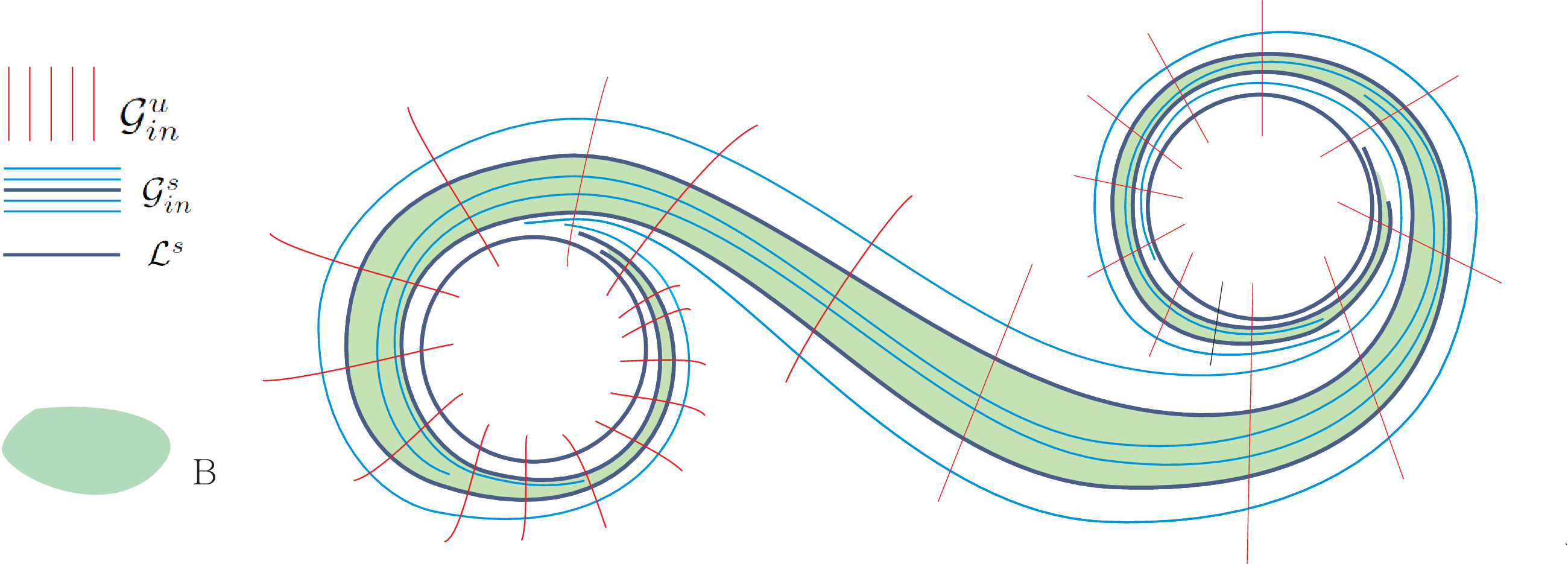}\\
\caption{\label{f.strip}A connected component of $\partial^{in} V\setminus\cL^s$.}
  \end{center}
\end{figure}

\subsubsection{Distorsion control of the holonomies}

\begin{lemm}
\label{l.distorsion} 
Let $\cF$ be a MS-foliation of a compact surface $S$, so that the holonomy of each compact leaf is conjugated to a homothety.
Let $\cL$ be a filling MS-sub-lamination of $\cF$. Let $\cG$ be a smooth foliation transverse to $\cF$.

Then there is $C>1$ with the following property. Let $I,J$ be two segments of $\cG$-leaves, whose interiors are contained in a
connected component $B$ of $S\setminus \cL$, and whose end points are on the boundary of $B$.
Let $H_{I,J}$ be the holonomy of the foliation $\cF$ from $J$ to $I$.  Then for every $x,y\in J$ one has:
$$C^{-1}<\frac{DH_{I,J}(x)}{DH_{I,J}(y)}<C$$
\end{lemm}

An important point is that the constant $C$  depends neither on the connected component $B$ of $S\setminus \cL$ nor on the segments $I,J$.

\begin{proof} 
First notice that the existence of such an announced constant $C$ does not depend on the metric on the surface $S$ (only the value of $C$ will depend on the metric).  Therefore we may choose a metric on $S$ so that the holonomy of every compact leaf of $\cF$ is a homothety. More precisely, denote by $\gamma_1,\dots,\gamma_p$ the compact leaves of $\cF$. We choose a metric on $S$ such that, for every $i\in\{1,\dots,p\}$, there is a tubular neighborhood $T_i$ of the compact leaf $\gamma_i$ so that the fibers of the tubular neighborhood are segments of leaves of $\cG$, and so that the holonomy map from any fiber to any other fiber is a homothety.

Since $\cL$ is a filling MS-lamination, every half-leaf of $\cF$ spirals around some compact leaf. It follows that the length of a segment of leaf of $\cF$ which is disjoint from $T_1\cup\dots\cup T_p$ is uniformly bounded. As a consequence, there exists a constant $\ell$ with the following property. Given a connected component $B$ of $S\setminus \cL$, and two segments $I,J$ of $\cG$-leaves as in the statement of Lemma~\ref{l.distorsion}, the holonomy map $H_{I,J}$ can be decomposed  as
$$H_{I,J}= H_{I,I_1}\circ H_{I_1,J_1}\circ H_{J_1,J}.$$
 where $I_1,J_1$ are segments of $\cG$-leaves in  $\gamma$ of $S\setminus \cL$, and with endpoints on the boundary of $\gamma$ so that $H_{I,I_1}$ and $H_{J_1,J}$  are homotheties,  and the holonomy $H_{I_1,J_1}$ is along $\cF$-leaf segments of length bounded by $\ell$. 
 
 On the one hand, the distorsion of $H_{I,J}$ coincides with the distorsion of $H_{I_1,J_1}$ (since $H_{I,I_1}$ and $H_{J_1,J}$  are homotheties). On the other hand, the distorsion of $H_{I_1,J_1}$ is uniformly bounded (because the holonomies of the foliation $\cF$ along segments of leaves of length bounded by $\ell$ have uniformly bounded derivative). Hence, the distorsion of $H_{I,J}$ is uniformly bounded.
\end{proof}

\subsubsection{Building $\psi^{in}$ on a large square}
\label{ss.psi-h}

\begin{defi}
Let $I$ be a compact segment of $\cG^u_{in}$-leaf contained in a connected component $B$ of $\partial^{in} V\setminus \cL^s$. Given $h$ a diffeomorphism of $I$, so that the end points of $I$ are flat fixed points for $h$. We denote by  $\psi_h$  the unique diffeomorphism of $B$ so that
\begin{itemize}
 \item $\psi_h$ is the identity  out of the $\cG^s_{in}$-saturation of $I$
 \item $\psi_h$ preserves (globally) the foliation $\cG^s_{in}$
 \item $\psi_h$ preserves each leaf segment of $\cG^u_{in}$,
 \item the restriction of $\psi_h$ to $I$ is  $h$.
\end{itemize}
\end{defi}

The aim of this subsection is to prove the following result:

\begin{prop}
\label{p.largesquare} 
Given any $\lambda>1$ and  any component $B$ of $\partial^{in}V\setminus \cL^s$ there is a segment $I$ of $\cG^u_{in}$-leaf contained in $B$,  and a diffeomorphism $h:I\to I$ o that the end points of $I$ are flat fixed points for $h$, and so that, for any vector $u$ tangent to $\cG^u_{in}$ at some point $x\in B$, one has :
$$\|(\Ga\circ\psi_h)_*(u)\|>\lambda.\|u\|.$$
\end{prop}

Porposition~\ref{p.largesquare} announces a control of the expansion on unit vectors tangent to $\cG^u_{in}$ at any point $x$ of a connected component $B$ of $\partial^{in}V\setminus \cL^s$.  We start by getting such a control along one lsegment of $\cG^u_{in}$-leaf crossing $B$:

\begin{lemm}
\label{l.derivativeinterval}
Let $B$ be a component of $\partial^{in}V\setminus \cL^s$, and $\sigma$ be  leaf of the restriction of  $\cG^u_{in}$ to $B$. Fix any constant $A>1$. Then, there is a diffeomorphism $h:\sigma\to\sigma$, equal to the identity map out of some compact part of $\sigma$,  so that for every vector $u$ tangent to $\cG^u_{in}$ at some point $x\in\sigma$ one has
$$\|(\Ga\circ\psi_h)_*(u)\|>A.\|u\|.$$
\end{lemm}

\begin{proof}  
Observe that $\sigma$ is a interval of bounded length, and $\Ga(\sigma)$ is an entire leaf of $\cG^u_{out}$ hence isometric to $\RR$. Furthermore, according to Lemma~\ref{l.stronghyperbolicity}, the rate of expansion of the crossing map $\Gamma$ for vector tangent to $\sigma$ tends to infinity close to the ends of $\sigma$. Therefore Lemma~\ref{l.derivativeinterval} is a direct consequence of the following general lemma (whose proof is left ot the reader). 

\begin{lemm} 
Let $\phi\colon ]0,1[\to\RR$ be a diffeomorphism whose derivative tends to $+\infty$ when $t$ tends to $0$ or $1$. For any $A>1$, there is a diffeomorphism $\tilde\phi$ which
coincides  with $\phi$ in a neighborhood of $0$ and of $1$ and whose derivative is everywhere larger than $A$.
\end{lemm}

This completes the proof of Lemma~\ref{l.derivativeinterval}
\end{proof}

The following lemma allows to compare the rate of expansion of the map $\Ga\circ\psi_h$ for vector tangent to $\cG^u_{in}$ at different points of $\partial^{in}V\setminus \cL^s$.

\begin{lemm}\label{l.holonomies} 
Let $B$ be a connected component  of $\partial^{in}V\setminus \cL^s$ and $\sigma$ be leaf of the restriction of  $\cG^u_{in}$ to $B$. There exists a constant $\alpha_\sigma>0$ with the following property. For every diffeomorphism $h\colon\sigma\to \sigma$ supported in a compact segment $I\subset \sigma$, and every vectors $u,v$ tangent to $\cG^u_{in}$ at some points $x,y\in B$, such that $x,y$ belong to the same leaf of $\cG^s_{in}$ and $y\in\sigma$, one has
$$\frac{\|(\Ga\circ\psi_h)_*(u)\|}{\|u\|}>\alpha_\sigma \frac{\|(\Ga\circ\psi_h)_*(v)\|}{\|v\|}.$$
\end{lemm}

\begin{proof}
Denote $\sigma_x$ the leaf of the restriction of  $\cG^u_{in}$ to $B$ containing $x$. Observe that $\Ga(B)$ is a connected component of $\partial^{out} V\setminus\cL^u$, and $\Si:=\Ga(\sigma)$ and $\Si_x:=\Ga(\sigma_x)$ are two leaves of $\cG^u_{out}$ contained in $\Ga(B)$. We denote by $H_{\sigma_x\to\sigma}\colon \sigma_x\to\sigma$  the holonomy of the foliation $\cG^s_{in}$ from $\sigma_x$ to $\sigma$. We denote by $H_{\Si\to\Si_x} \colon\Si\to \Si_x$ the holonomy of the foliation $\cG^s_{out}$ from $\Sigma$ to $\Sigma_x$. 

By construction, the restriction of $\psi_h$ to $\sigma_x$ is  conjugated to $h$ by $H_{\sigma_x\to\sigma}$.
One  deduces that the restriction of $\Ga\circ\psi_h$  to $\sigma_x$ can be written as:
\begin{equation}
\label{e.crossing-and-holonomies}
(\Ga\circ\psi_h)|_{\sigma_x}=H_{\Si\to \Si_x}\circ(\Ga\circ\psi_h)|_{\sigma}\circ H_{\sigma_x\to \sigma.}
\end{equation}
The following lemma gives a uniform bound for the derivative of $H_{\Si\to\Si_x}$:

\begin{lemm} 
\label{l.bound-holonomy-1}
There exists a constant $\beta>1$, such that the holonomy of the foliation $\cG^s_{out}$ between two leaves of $\cG^u_{out}$ in the same connected component $\partial^{out}V\setminus \cL^u$ has a derivative which is bounded by $\beta$.
\end{lemm}

\begin{proof}
Just notice that $\cG^s_{out}$ is a smooth foliation, and that the segment of leaves of $\cG^u_{out}$ contained in $\partial^{out}V\setminus \cL^u$ have uniformly bounded length.
\end{proof}

The following Lemma~\ref{l.holonomies} gives a uniform lower bound for the derivative of $H_{\Si\to\Si_x}$:

\begin{lemm} 
\label{l.bound-holonomy-2}
There is a constant $\beta_\sigma>0$ so that, for every $x\in B$ and any vector $u$ tangent to $\sigma_x$, one has
$$\|(H_{\sigma_x\to\sigma})_*(u)\|>\beta_\sigma.\|u\|.$$
\end{lemm}

\begin{proof} 
The component $B$ is a strip  whose ends converge to compact leaves whose holonomies are conjugated to homotheties.   Lemma~\ref{l.distorsion} asserts that the holonomy of the foliation $\cG^s_{in}$ between two leaves of the restriction to $\gamma$ of $\cG^u_{in}$ have uniformly bounded distorsion $C$. As a consequence, for every $x\in B$, the derivative on the holonomy $H_{\sigma_x,\sigma}$ is  larger than $ \frac{\ell(\sigma)}{C\cdot \ell(\sigma_x)}$ where $\ell$ is the length. One concludes by noticing that the length $\ell(\sigma_x)$ is uniformly bounded, so that
$$\inf_{x\in\gamma}\frac{\ell(\sigma)}{C\cdot \ell(\sigma_x)}>0.$$
\end{proof}

Putting together equality~\eqref{e.crossing-and-holonomies}, Lemma~\ref{l.bound-holonomy-1} and Lemma~\ref{l.bound-holonomy-2}, one easily sees that the constant $\alpha_\sigma:=\beta_\sigma\cdot \beta^{-1}$ satisfies the properties announced in Lemma~\ref{l.holonomies}. This completes the proof of Lemma~\ref{l.holonomies}.
\end{proof}

\begin{proof}[Proof of Proposition~\ref{p.largesquare}]
One just needs to combine Lemma~\ref{l.derivativeinterval} and Lemma~\ref{l.holonomies}, with a constant $A$ larger than $\alpha_\sigma.\lambda$.
\end{proof}

\subsubsection{Estimates for the derivative of $\psi_h$}

\begin{coro}
\label{c.distorsion} 
Let $B$ be a connected component of $\partial^{in}V\setminus \cL^s$ and $I$ be a segment of $\cG^u_{in}$-leaf contained in $B$, and $h$ be a diffeomorphism of $I$ so that the end points of $I$ are flat fixed points for $h$. We consider the diffeomorphism $\psi_h$ of $B$ associated to $h$ (see subsection~\ref{ss.psi-h}).  Let $u$ be a vector tangent to $\cG^u_{in}$ at some point $y\in B$. Then:
$$C^{-1}\inf_{x\in I} |Dh(x)|.\|u\|\leq\|D\psi_h(u)\|\leq C\sup_{x\in I} |Dh(x)|.\|u\|.$$
where $C$ is the bound on the distorsion of the holonomy of foliation $\cG^s_{in}$ given by Lemma~\ref{l.distorsion}. 
\end{coro}

\begin{proof} 
Let $\sigma_y$ be the leaf through $y$ of the restriction of $\cG^u_{in}$ to $B$. Notice that the restriction of $\psi_h$ to $\sigma_y$ is the conjugated of $h$ by the holonomy of $\cG^s_{in}$. According to Lemma~\ref{l.distorsion}, the distorsion of this holonomy is bounded by $C$. This yields the desired estimates.
\end{proof}

\subsubsection{``Slowing down" the diffeomorphisms $\psi_h$ close to the ends of the strip}

Proposition~\ref{p.largesquare} built diffeomorphism $\psi_h$ of a connected component $B$ of $\partial^{in}V\setminus \cL^s$. Recall that $B$ is a strip bounded by two non-compact leaves  of $\cL^s$ which are asymptotic to each other at both ends.  Each end of the strip $B$ spirals around a compact leaf of $\cL^s$. The diffeomorphism $\psi_h$ coincides with the identity outside of the $\cG^s_{in}$-saturation of some compact interval $I\subset B$. Nevertheless, $\psi_h$ does not tend to the identity close to the ends of $B$. This is the reason why we need to ``slow down" $\psi_h$ close to the ends of $B$.

\medskip

We consider  a compact leaf $c$ of $\cL^s$ (or equivalently of $\cG^s_{in}$) contained in the closure of $B$ (\emph{i.e.} there is one end of $B$ spiraling around $c$).
We orient $c$ so that its holonomy is a linear contraction. Recall that $\cG^u_{in}$  is a smooth foliation transverse to $\cG^s_{in}$. So, one can choose a smooth  tubular neighbourhood $O$ of $c$ so that:
\begin{itemize}
 \item the boundary $\partial O$ is transverse to $\cG^s_{in}$,
 \item the fibers of $O$ are segments of leaves of $\cG^u$.
\end{itemize}

We choose a parametrization of $c$ by $S^1=\RR/\ZZ$, so that the universal cover of $O$ can be identify with $\RR\times [-1,1]$ where the lifts of the leaves of $\cG^u_{in}$ are the segments $\{t\}\times [-1,1]$. For every $\theta\in S^1$, and every $t\in\RR$, we will denote by $H_{\theta,t}$ the holonomy of the foliation $\cG^s_{in}$
from the fiber $\sigma^u_\theta= \{\theta\}\times [-1,1]$ to the fiber $\sigma^u_{\theta+t}=\{\theta+t\}\times [-1,1]$. More precisely, we choose a lift  $\bar \theta$ of  $\theta$, and we considers the holonomy of the lifted foliation $\bar \cG^s_{in}$ from the fiber $\{\bar \theta\}\times [-1,1]$ to the fiber $\{\bar\theta+t\}\times [-1,1]$;  the projection of this holonomy does not depend on the lift $\bar \theta$. Notice that, for every $t>0$ and every $\theta$, the holonomy $H_{\theta,t}$ is defined on the whole fiber, and is a contraction.

\begin{lemm}
\label{l.ends}  
Let $C$ be the constant given by Lemma~\ref{l.distorsion}. Let  $I$ be a segment of $\cG^u_{in}$-leaf ocontained in $\sigma^u_\theta\cap B$. We denote by $I_t$ the image of $I$ by the holonomy $H_{\theta,t}$.  Let $h$ be a diffeomorphism of $I$ so that the end points of $I$ are flat fixed points of $h$. For every $\varepsilon >0$, there is a diffeomorphism $\psi^+$ of $\gamma$, with the following properties:
 \begin{itemize}
 \item $\psi^+$ preserves each leaf $\sigma^u$ of $\cG^u_{in}$;
 \item $\psi^+$ is equal to $h$ on $I$;
 \item $\psi^+$ is the identity out of the $\cG^s_{in}$-saturation of  $I$;
 \item $\psi^+$ coincides with $\psi_h$ outside $O$, and also coincides with $\psi_h=H_{0,-t}hH_{0,-t}^{-1}$ on $I_{-t}$ for every $t>0$;
 \item for $t>0$ large enough, $\psi_+$ is the identity map on $I_t$;
 \item $\psi^+(\cG^s_{in})$ is $\varepsilon$-$C^1$-close to $\cG^s_{in}$;
 \item the action of $\psi^+$ on vectors tangent to $\cG^u_{in}$ is controlled by the derivative of $h$; more precisely, for every vector $u$ tangent to $\cG^u_{in}$, 
 \begin{equation}
 \label{e.control}
 C^{-1} \inf_{x\in I} |Dh(x)|.\|u\|  \leq \|(\psi^+)_*(u)\|\leq C\cdot \sup_{x\in I} |Dh(x)|.\|u\|;
 \end{equation}
\end{itemize}
\end{lemm}

\begin{proof}
We consider  the  isotopy $(h_t)_{t\in [0,1]}$ joining $h$ to the identity by convex sum, \emph{i.e.}  $h_t(x)-x=t (h(x)-x)$. Hence $Dh_t(x)-1= t\cdot (Dh(x)-1)$. We consider a   smooth decreasing  map $\tau\colon \RR\to [0,1[$ such that $\tau(t)=1$ for $t<0$ and $\tau(t)=0$ for $t$ large enough.  We defined $\psi^+$ as follows:
\begin{itemize}
\item $\psi^+=\psi_h$ outside $O$,;
\item $\psi^+=\mathrm{Id}$ outside the $\cG^s_{in}$-saturation of $I$;
\item $\psi^+=H_{\theta,t} h_{\tau(t)} H_{\theta,t}^{-1}$ on $I_t$.
\end{itemize} 
One easily checks that this definition is coherent. The diffeomorphism $\psi^+$ satisfies trivially all the desired properties, except for the two last ones (the control of distance between the foliations $\psi^+(\cG^s_{in})$ and $\cG^s_{in}$, and the control of the action of the derivative of $\psi^+$ on the vectors tangent to the $\cG^u_{in}$-leaves).

To obtain proximity between the foliations $\psi^+(\cG^s_{in})$ and $\cG^s_{in}$, one just needs to notice that:
\begin{itemize}
\item the $C^1$-distance between $\psi^+(\cG^s_{in})$ and $\cG^s_{in}$ tends to $0$ when $\sup_{\RR} | D\tau(t)|$ tends to $0$;
\item we can choose  the function $\tau$ so that $\sup_{\RR} | D\tau(t)|$ is arbitrarily small.
\end{itemize}

So we are left to check the last property. For this purpose, we consider a vector $u$ tangent to a $\cG^u_{in}$-leaf at some point $y\in B$. Assume that the point $y$ belongs to $O$. Hence, we have 
$$\|(\psi^+)_*(u)\|=\|(H_{theta,t}\circ h_{\tau(t)}\circ H_{\theta,t}^{-1})_*(u)\|= \frac{\|DH_{\theta,t}(z_2)\|}{\|DH_{\theta,t}(z_1)\|}.\|1+\tau(t)\cdot (Dh(z_1)-1)\|.\|u\|$$
where $z_1=H_{\theta,t}^{-1}(y)$ and $z_2=h(z_1)$. Using Lemma~\ref{l.distorsion} and the fact that $|\tau(t)|$ is less than $1$, this yields to the desired inequality~\eqref{e.control}. If the point $y$ is not in $O$, we get the inequality by similar (but easier) arguments: indeed $\psi^+=\psi_h$ outside $O$, and the restriction of $\psi_h$ to a $\cG^u_{in}$ is conjugated to $h$ by the holonomy of the foliation $\cG^s_{in}$.
\end{proof}

\subsubsection{End of the proof of Proposition~\ref{p.returnmap}}

\begin{proof}[Proof of Proposition~\ref{p.returnmap}] 
Fix $\lambda>1$ and $\varepsilon>0$. According to Lemma~\ref{l.stronghyperbolicity}, there is an adapted neighbourhood of $\cU^{in}_\lambda$ of $\cL^s$ in $\partial^{in} V$, so that:
\begin{equation}
\label{e.expansion-1}
\mbox{$\|(\Gamma)_*(u)\|\geq \lambda.\|u\|$ for every vector $u$ tangent to $\cG^u_{in}$ at some point $x\in \partial^{in} V\setminus \cU^{in}_\lambda$.}
\end{equation}
The set $\partial^{in} V\setminus\cU^{in}_\lambda$ is contained in finitely many connected components $B_1,\dots,B_m$ of $\partial^{in}V\setminus \cL^s$. Recall that each $B_i$ is a strip bounded by two non-compact leaves of $\cL^s$ which are spiraling (at both ends) around some compact leaves of $\cL^s$.

For $i=1,\dots, m$, we consider a homeomorphism $h_i$ associated to $\lambda$ and $B_i$ by Proposition~\ref{p.largesquare}. The map $\Gamma\circ\psi_{h_i}$ expands vectors tangent to $\cG^u_{in}$ by a factor larger than $\lambda$:
\begin{equation}
\label{e.expansion-2}
\|(\Gamma\circ\psi_{h_i})_*(u)\|\geq \lambda.\|u\|\mbox{ for every $u$ tangent to $\cG^u_{in}$ at some point $x\in B_i$.} 
\end{equation}
 The only trouble is that $\psi_{h_i}$ cannot be extended as a diffeomorphisms on the closure of $B_i$. To overcome this problem, we will modify $\psi_{h_i}$ on the ends of the strip $B_i$ using Lemma~\ref{l.ends}.

Let $m$ be a lower bound for the derivatives of all the $h_i$'s. According to Corollary~\ref{c.distorsion}, there exists a constant $C$ such that, for every $i$,
\begin{equation}
\label{e.expansion-3}
 \|(\psi_{h_i})_*(u)\|\geq C^{-1}m.\|u\|\mbox{ for every $u$ tangent to $\cG^u_{in}$ at some point $x\in B_i$.}
\end{equation}

Now, we use Lemma~\ref{l.adapted} and again Lemma~\ref{l.stronghyperbolicity} to get an adapted neighborhood $\cU^{in}\subset \cU^{in}_{\lambda}$ of $\cL^s$ so that:
\begin{equation}
\label{e.expansion-4}
\mbox{$\|(\Gamma)_*(u)\|\geq (\lambda\cdot C^2m^{-1}).\|u\|$ for every vector $u$ tangent to $\cG^u_{in}$ at some point $x\in \partial^{in} V\setminus \cU^{in}$.}
\end{equation}
By definition of adapted neighborhood, the complement of $\cU^{in}$ consists in finitely many $in$-rectangles, and each connected component of $\partial^{in}V\setminus \cL^s$ contains at most one of these $in$-rectangles. We denote by $R_i$ the $in$-rectangle contained in the strip $B_i$. The set $\partial^{in} V\setminus\cU^{in}_\lambda$ is contained in the interior of the union of the $R_i$'s. Up to fattening the $in$-rectangles $R_i$ (that is, up to shrinking the adapted neighbourhood $\cU^{in}$) one may assume that the $\cG^u_{in}$-sides of the $in$-rectangle $R_i$ are contained in tubular neighborhoods of the compact leaves fo $\cL^s$ in the closure of $B_i$.

 Applying Lemma~\ref{l.ends} to both the $\cG^u_{in}$-sides of the rectangle $R_i$, one gets a diffeomorphism $\psi_i$ of the strip $B_i$ so that
 \begin{itemize}
  \item $\psi_i$  preserves every leaf of the restriction of $\cG^u_{in}$ to $B_i$;
  \item  the restriction of $\psi_i$  to the $in$-rectangle $R_i$ is $\psi_{h_i}$;
  \item $\psi_i$ coincides with the identity map out of a compact subset of $B_i$;
  \item $\psi_i$ expands vectors tangent to $\cG^u_{in}$ by a factor larger than $C^{-2}m$:
  \begin{equation}
\label{e.expansion-5}
 \|(\psi_i)_*(u)\|\geq C^{-2}m.\|u\|\mbox{ for every $u$ tangent to $\cG^u_{in}$ at some point $x\in B_i$.}
\end{equation}
  \item  the $C^1$-distance between the foliations $(\psi_i)_*(\cG^s_{in})$ and $\cG^s_{in}$ is smaller than $\varepsilon$;
 \end{itemize}

 We consider the diffeomorphism $\psi^{in}$ of $\partial^{in} V$ which coincides with $\psi_i$ on the strip $B_i$ and conincides with the identity map out of the union of the $B_i$'s. Let us check that $\psi^{in}$ satisfies all the announced properties: it is the identity map on a neighborhood of $\cL^s$, preserves every leaf of $\cG^u_{in}$, and $(\psi^{in})_*(\cG^s_{in})$ is $\varepsilon$-$C^1$-close to $\cG^s_{in}$.  It remains to control the action of derivative of $\Ga\circ \psi^{in}$ on vectors tangent to $\cG^u_{in}$.
 \begin{itemize}
 \item On $\partial^{in} V\setminus \bigcup_i B_i$, the diffeomorphism $\psi^{in}$ coincides with the identity. Therefore, \eqref{e.expansion-1} and the inclusion of $\partial^{in} V\setminus \bigcup_i B_i$ in $\cU_{\lambda}$ implies that $\Ga\circ \psi^{in}$ expands vectors tangent to $\cG^u_{in}$ by a factor larger than $\lambda$.
 \item On $B_i\setminus\cU^{in}=R_i$, the diffeomorphism $\psi^{in}$ coincides with $\psi_{h_i}$. Therefore, \eqref{e.expansion-2}  implies that $\Ga\circ \psi^{in}$ expands vectors tangent to $\cG^u_{in}$ by a factor larger than $\lambda$.
 \item On $B_i\cap\cU^{in}$, the diffeomorphism $\psi^{in}$ coincides with $\psi_i$. Therefore \eqref{e.expansion-4} and \eqref{e.expansion-5} imply that $\Ga\circ\psi^{in}$ expands vectors tangent to $\cG^u_{in}
 $ by a factor larger than~$\lambda$. 
 \end{itemize}
 This conpletes the proof of Proposition~\ref{p.returnmap} (and therefore also of Theorem~\ref{t.transitive}).
\end{proof}

\subsection{Transitivity}

The aim of this subsection is to prove Proposition~\ref{p.transitive}.

\begin{lemm} \label{l.transitivity}
Every orbit of the Anosov flow given by Theorem~\ref{t.transitive}  which is not contained in $V$ has
its stable  and unstable manifold cutting $\cL^u$ and $\cL^s$, respectively.
\end{lemm}
\begin{proof} This orbit cuts $\partial^{out} V$ so that its stable manifold contains a leaf of the foliation
$\cG^s_{out}$, whose all leaves cut $\cL^u$.
\end{proof}

\begin{proof}[Proof of Proposition~\ref{p.transitive}]
Lemma~\ref{l.transitivity} implies that every orbit $\gamma$ of the resulting Ansosov flow has its stable (resp. unstable)
manifold cutting transversely the unstable (resp. stable) manifold  of a basic piece of the maximal
invariant  set $\La$ in $V$.  The combinatorial transitivity means that, after gluing, all the basic pieces of $\La$
are related by a cycle, and hence belong to the same basic piece of the Anosov flow.  Now the stable and unstable manifolds of
$\gamma$ cut the unstable  and stable manifold of this basic piece of the Anosov flow, so that $\gamma$ belongs to this
basic piece.  One deduces that the whole manifold is a unique basic piece, which means that the flow is transitive.
\end{proof} 
\part{Applications of the gluing theorem~\ref{t.transitive}}
\section{MS-foliations and filling MS-laminations}
\label{s.simple-foliations}

The purpose of this section is to investigate the
 geometry of filling MS-laminations on closed orientable
surfaces. We will define the \emph{combinatorial types} of a 
filling MS-lamination: these are simple
   combinatorial objects encoding the orientations of
    the compact leaves of the lamination. Then we will focus on the particular case of MS-foliations, and
     prove that every MS-foliation is characterized up to topological equivalence by any of its
      combinatorial types. This result will play a
 crucial role in the proofs of
Theorems~\ref{t.richer},~\ref{t.attractors},
~\ref{t.embedding},~\ref{t.manyanosov}.

All along this section, we will consider filling MS-laminations on the torus $\TT^2$. Indeed, up 
to diffeomorphism, $\TT^2$ is the only 
closed connected orientable surface which carries 
filling MS-laminations. From now on, we assume that 
an orientation of $\TT^2$ is fixed.

\subsection{Combinatorial type of a filling MS-lamination}

\begin{lemm}
\label{l.parallel-leaves}
Let $\cL$ be a filling MS-lamination on $\TT^2$. 
The compact leaves of $\cL$, regarded as 
non-oriented closed curves on $\TT^2$, are 
non-contractible and pairwise freely homotopic.
\end{lemm}

\begin{proof}
According to Lemma~\ref{l.pre-foliation}, $\cL$ can 
be completed to a MS-foliation 
$\cF$. Lemma~\ref{l.parallel-leaves} is a 
consequence of Poincar\'e-Hopf Theorem applied to 
the foliation $\cF$.
\end{proof}

\begin{defi}[Contracting orientation]
\label{d.contracting-orientation}
Let $\cL$ be a filling MS-lamination on $\TT^2$, and $\gamma$ be 
a compact leaf of $\cL$. The 
\emph{contracting orientation} of $\gamma$ is 
the orientation for which the holonomy of $\gamma$ 
is a contraction.
\end{defi}

Lemma~\ref{l.parallel-leaves} alllows to compare 
the orientations of two compact leaves of a 
filling MS-lamination:

\begin{defi}[Coherently orientated compact leaves]
\label{d.coherent-incoherent}
Let $\cL$ be a filling MS-lamination on $\TT^2$, and 
$\gamma,\gamma'$ be some compact leaves of 
$\cL$, endowed with their contracting orientations. 
We say that $\gamma$ and $\gamma'$ 
are \emph{coherently oriented} if they are 
freely homotopic, when regarded as oriented 
closed curves.
\end{defi}

\begin{figure}[ht]
\begin{center}
\includegraphics[totalheight=4.5cm]{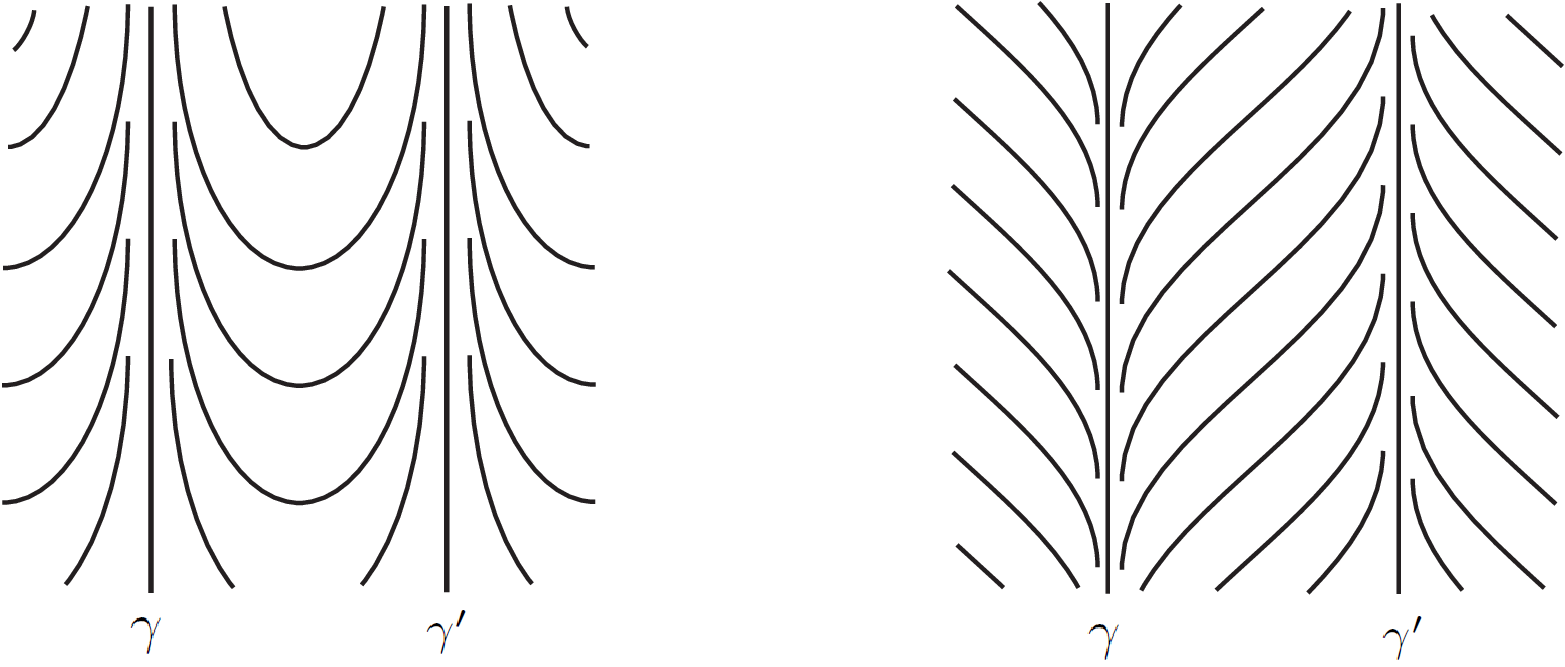}
\caption{\label{f.dynamical-orientations} 
Left: compact leaves with coherent 
contracting orientations. Right: compact leaves 
with incoherent contracting orientations.}
\end{center}
\end{figure}

Lemma~\ref{l.parallel-leaves} implies that the 
compact leaves of a filling MS-lamination are 
``cyclically ordered"; let us formalize this:

\begin{defi}[Geometrical enumeration]
\label{d.geom-enum}
 Let $\cL$ be a filling MS-lamination of $\TT^2$, with 
 $n$ compact leaves. An enumeration 
 $\gamma_0,\dots,\gamma_{n-1}$ of the compact leaves 
 of $\cL$ is called a 
 \emph{geometrical enumeration} if it satisfies 
 the 
 following properties:
\begin{itemize}
\item for $i=0\dots n-1$, the leaves $\gamma_i$ 
and $\gamma_{(i+1)\mathrm{ mod }n}$ bound a 
connected component $A_i$ of 
$\TT^2\setminus \bigcup_k \gamma_k$,
\item $A_1$ is on the right-hand side
\footnote{The orientation of $\TT^2$ provides 
an notion of \emph{local right-hand side of 
an oriented closed curve}. To include the 
particular case where $\cL$ has a single compact 
leaf, we allow $A_0$ to be on both sides 
of $\gamma_0$.} of $\gamma_0$, with respect 
to the contracting orientation of $\gamma_0$.
\end{itemize}
\end{defi}

\begin{defi}[Combinatorial type]
Let $\cL$ be a filling MS-lamination on $\TT^2$, 
and $\gamma_0,\dots,\gamma_{n-1}$ be a 
geometrical enumeration of the compact leaves 
of $\cL$. The \emph{combinatorial type} of 
the lamination $\cL$ (associated to the 
enumeration $\gamma_0,\dots,\gamma_{n-1}$) is the 
map
$$\sigma:\{0,\dots,n-1\}\to \{-,+\}$$
defined as follows: $\sigma(i)=+$ if and only if 
the contracting orientations of $\gamma_i$ 
and $\gamma_0$ are coherent
\footnote{In particular, $\sigma(0)$ is always 
equal to $+$ .}.
\end{defi}

\begin{rema}
Let $\cL$ be a filling MS-lamination on $\TT^2$, with 
$n$ compact leaves. There are $n$ possible 
geometrical enumeration of the compact leaves 
of $\cL$. To each geometrical enumeration 
is associated a combinatorial type of $\cL$. 
These combinatorial types can easily be deduced 
from one another.
\end{rema}

\subsection{MS-foliations are characterized 
by their combinatorial types}

\begin{defi}[Topological equivalence on 
oriented surfaces]
Let $\cL,\cL'$ be laminations on oriented 
surfaces $S,S'$. We will say that $\cL$ and $\cL'$ 
are \emph{topologically equivalent} if there exists 
an orientation-preserving homeomorphism 
$h:S\to S'$ such that $h_* (\cL)=\cL'$
\end{defi}

Keep in mind that we only consider 
topological equivalences induced by 
\emph{orientation preserving} homeomorphism. 
Now, let us focus our attention on MS-foliations.

\begin{prop}
\label{p.combinatorial-type}
A MS-foliation on $\TT^2$ is characterized 
up to topological equivalence by any of 
its combinatorial types.
\end{prop}

\begin{proof}
Consider two MS-foliations $\cF^1$ and $\cF^2$ 
on $\TT^2$. Fix some geometrical 
enumerations $\gamma_0^1,\dots,\gamma_{n-1}^1$ 
and $\gamma_0^2,\dots,\gamma_{n-1}^2$ of the 
compact leaves of $\cF^1$ and $\cF^2$, and denote 
by $\sigma^1$ and $\sigma^2$ the correspondng combinatorial types of $\cF^1$ and $\cF^2$. 
Assume that $\sigma^1=\sigma^2$. We will prove 
that $\cF^1$ and $\cF^2$ are topologically equivalent.

If we endow the compact leaf $\gamma_i^j$ with 
its contracting orientation, then the holonomy 
of $\cF^i$ along $\gamma_i^j$ is a 
contraction. Therefore, we can find an 
arbitrarily small tubular neighbourhood $U_i^j$ 
of $\gamma_i^j$, such that $\cF^j$ is transverse 
to $\partial U_i^j$. For $j=1,2$, we can assume 
that the neighbourhoods $U_1^j,\dots,U_n^j$ 
are pairwise disjoint. We denote by $A_i^j$ 
the connected component of 
$\TT^2\setminus \bigcup_i \mathrm{int}(U_i^j)$ 
which lies between $U_i^j$ and $U_{i+1}^j$. We 
denote by $\partial^{\ell} U_i^j$ (resp. 
$\partial^{r} U_i^j$) the boundary component 
of $U_i^j$ which is also a boundary component 
of $A_i^j$ (resp. $A_i^{(j-1)}$). The 
contracting orientation of the leaf 
$\gamma_1^j$ induces an orientation of the 
closed curves $\partial^{\ell} U_i^j$ 
and $\partial^{r} U_i^j$. Note that 
$\partial A_i^j=\partial^{r} U_i^j\cup 
\partial^{\ell} U_{i+1}^j$.

\begin{claim}
For each $i$, there is 
an orientation-preserving 
homeomorphism $\psi_i:A_i^1\to A_i^2$ such that 
which maps $\partial^{r} U_i^1$ and 
$\partial^{\ell} U_{i+1}^1$ on 
$\partial^{r} U_i^2$ and 
$\partial^{\ell} U_{i+1}^2$ respectively, and 
such that $(\psi_i)*(\cF^1)=\cF^2$.
\end{claim}

\begin{proof} Indeed, $A_i^j$ is a compact annulus, disjoint 
from the compact leaves of $\cF^j$, whose boundary 
is transverse to $\cF^j$. Since every half-leaf 
of $\cF^j$ accumulates on a compact leaf, this 
implies that the restriction of $\cF^j$ to $A_i^j$ 
is topologically conjugate to the vertical 
foliation on the annulus $(\RR/\ZZ)\times [-1,1]$. 
The claim follows.
\end{proof}

\begin{claim}
For each $i$, there is 
an orientation-preserving 
homeomorphism $\phi_i:U_i^1\to U_i^2$, which 
coincides with $\psi_{i}$ (resp. $\psi_{i-1}$) 
on $\partial^{r} U_i^1$ 
(resp. $\partial^{\ell} U_i^1$), and such 
that $\phi_*(\cF^1)=\cF^2$.
\end{claim}

\begin{proof}
We endow $\gamma_i^1$ (resp. $\gamma_i^2$) with 
the orientation that is coherent with the 
contracting orientation of $\gamma_1^1$ 
(resp. $\gamma_1^2$). Since 
$\sigma^1(i)=\sigma^2(i)$, there are 
two possibilities: either both the holonomies 
of $\gamma_i^1$ and $\gamma_i^2$ are contraction, 
or both the holonomies of $\gamma_i^1$ 
and $\gamma_i^2$ are dilations. Assume for 
example that they both are contractions.

Choose an oriented arc $\alpha_i^j$ in 
$U_i^j$,  tranverse to $\cF^j$, going 
from $\partial^{l} U_i^j$ to 
$\partial^{\ell} U_i^j$. The arc $\alpha_i^j$ is 
a cross section for the restriction of $\cF^j$ 
to $U_i^j$. Denote by $f_i^j$ the first return map 
of the leaves of $\cF^j$ on $\alpha^j$. The 
maps $f_i^1$ and $f_i^2$ are contractions. Hence 
they are topologically conjugate by an 
orientation-preserving homeomorphism 
$h_i:\alpha_i^1\to\alpha_i^2$. One deduces easily 
that there exists an orientation-
preserving homeomorphism $\phi_i:U_i^1\to U_i^2$ 
which maps $\partial^{\ell} U_i^1$ and 
$\partial^{r} U_i^j$ on $\partial^{\ell} U_i^2$ 
and  $\partial^{r} U_i^2$ respectively, such 
that $\phi_*(\cF^1)=\cF^2$. There is some freedom 
for the choice of the conjugating homeomorphism 
$h_i$: the restriction of $h_i$ to a 
fundamental domain of the contraction $f_i^j$ can 
be choosen arbitrarily. Hence, the restriction 
of $\phi_i$ to $\partial U_i^j$ can also be 
choosen arbitrarily (since every leaf of 
$\cF^j$ intersects $\partial U_i^j$ at most once).
The claim is proved.
\end{proof}

The homeomorphism 
$\phi_1,\psi_1,\dots,\phi_n,\psi_n$ provided by 
the two claims can be glued together to obtain 
a global orientation-preserving homeomorphism 
$\phi:\TT^2\to\TT^2$ such that $\phi_*(\cF^1)=\cF^2$.
This completes the proof of Proposition~\ref{p.combinatorial-type}.
\end{proof}

\begin{rema}
Proposition~\ref{p.combinatorial-type} is false 
for filling MS-laminations: by removing non-compact 
leaves to a given filling MS-lamination, one can 
easily find infinitely many filling MS-laminations 
with the same combinatorial type which are 
pairwise not topologically equivalent. 
Nevertheless, every filling MS-lamination $\cL$ can 
be embedded in a MS-foliation $\cF$ wtih the 
same compact leaves as $\cL$, the combinatorial 
types of $\cF$ are the same as those of $\cL$, 
and $\cF$ is characterized up to 
topological equivalence by these combinatorial types.
\end{rema}

\begin{defi}[Zipped Reeb lamination/foliation]
\label{d.zipped-Reeb}
We call \emph{zipped Reeb lamination} 
(resp. \emph{foliation}) a filling MS-lamination 
(resp. foliation) on $\TT^2$ with a single 
compact leaf. See figure~\ref{l.zipped-Reeb}.
\end{defi}

\begin{exam}
\label{e.zipped-Reeb}
Consider the vector field $X$ on $\RR^2$ defined 
by $X(x,y):=\sin(\pi.x)\frac{\partial}
{\partial x}+\cos(\pi x)\frac{\partial}
{\partial y}$. The orbits of this vector field 
defined a foliation on $\RR^2$. This foliation 
is invariant under the standard action of 
$\ZZ^2$. Therefore, it induces a foliation 
on $\TT^2=\RR^2/\ZZ^2$. One easily checks that this 
a zipped Reeb foliation.
\end{exam}

\begin{figure}[ht]
\begin{center}
\includegraphics[totalheight=4.2cm]{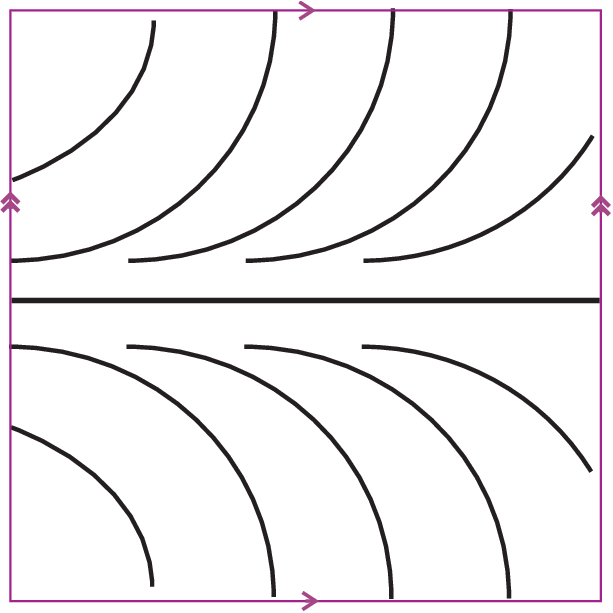}
\caption{\label{l.zipped-Reeb} A zipped 
Reeb lamination}
\end{center}
\end{figure}

\begin{defi}
\label{d.adding-leaves}
Let $\cF,\cF'$ be MS-foliations on 
$\TT^2$. Suppose that $\cF'$ has one more compact 
leaf than $\cF$, and suppose that there is 
a combinatorial type $\sigma:\{1,\dots,n\}\to
\{+,-\}$ of $\cF$ and a combinatorial  type 
$\sigma':\{1,\dots,n+1\}\to\{+,-\}$ of $\cF'$ 
such that $\sigma'_{|\{1,\dots,n\}}=\sigma$. We 
say that the foliation $\cF'$ is obtained 
by \emph{adding a compact leaf} to $\cF$.
\end{defi}

\begin{figure}[ht]
\begin{center}
\includegraphics[totalheight=3.4cm]{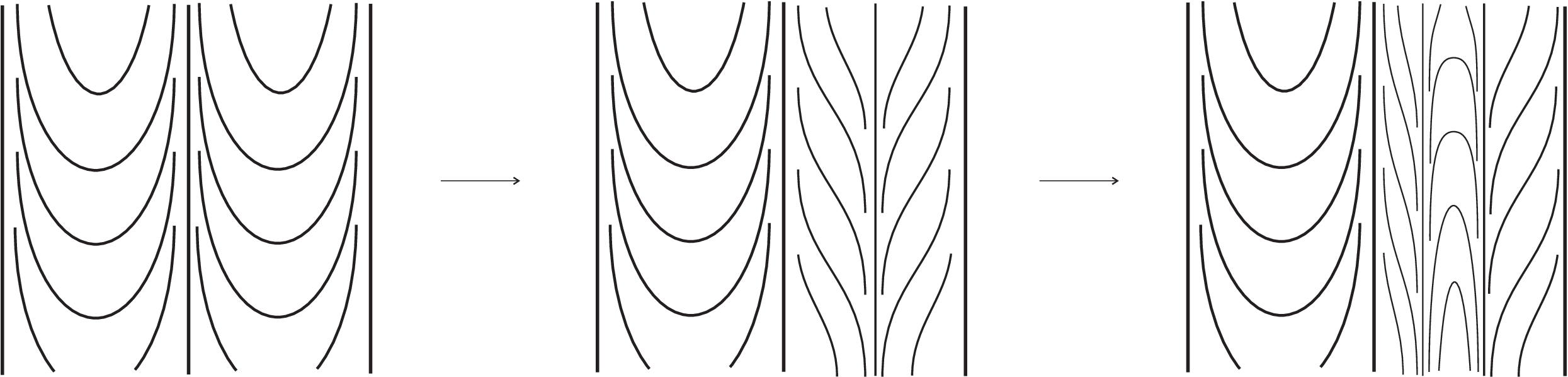}
\caption{\label{f.adding} Adding compact leaves 
to  MS-foliations.}
\end{center}
\end{figure}

The two following statements are 
immediate consequences 
of Proposition~\ref{p.combinatorial-type}:

\begin{coro}
\label{c.zipped-Reeb}
All zipped Reeb foliations are topological equivalent.
\end{coro}

\begin{coro}
\label{c.adding-leaves}
Up to topological equivalence, every simple 
foliation on $\TT^2$ can be obtained by 
adding inductively a finite number of compact 
leaves to a zipped Reeb foliation.
\end{coro}

\subsection{Contracting orientation versus 
dynamical orientation}

Consider a hyperbolic plug $(U,X)$. The compact 
leaves of the laminations $\cL^s_X$ and $\cL^u_X$ 
can be equipped with their contracting orientation. 
We will define another natural orientation for 
these compact leaves.

\begin{defi}[dynamical orientation]
\label{d.flow-orientation}
Let $\gamma$ be a compact leaf of the 
lamination $\cL^u_X$. According 
to (the proof of) Proposition~\ref{p.hyperbolicplug}, there exist 
a periodic orbit $O$ of $X$, such that $\gamma$ is 
a connected component of 
$W^u(O)\cap\partial^{out} U$. The orbit $O$ has 
a natural orientation defined by the vector field $X$.
\begin{itemize}
\item[--] If $O$ has positive multipliers, 
then $W^u(O)$ is a cylinder, and both $\gamma$ and 
$O$ are non-contractible closed curves on 
this cylinder.  The \emph{dynamical orientation} 
of $\gamma$ is the orientation for which $\gamma$ 
is freely homotopic to the orbit $O$ endowed with its natural orientation, in the cylinder $W^u(O)$. See figure~\ref{f.flow-orientation}.
\item[--] If  $O$ has positive multipliers, 
then $W^u(O)$ is a M\"obius band. The 
\emph{dynamical orientation} of $\gamma$ is 
the orientation for which $\gamma$ is freely 
homotopic to two times the orbit $O$ endowed with its natural orientation, in the cylinder $W^u(O)$.
\end{itemize}
We define similarly the dynamical orientation 
of a compact leaf of $\cL^s_X$.
\end{defi}

\begin{prop}
\label{p.flow-vs-contracting}
Let $(U,X)$ be a hyperbolic plug. If $\gamma$ is 
a compact leaf of $\cL^u_X$, the 
contracting orientation and the dynamical 
orientation of $\gamma$ coincide. If $\gamma$ is 
a compact leaf of $\cL^s_X$, the 
contracting orientation and the dynamical 
orientation of 
$\gamma$ are opposite.
\end{prop}

\begin{proof}
Let $\Lambda$ be the maximal invariant set of 
$(U,X)$. Recall 
that $\cL^u_X=W^u(\Lambda)\cap\partial^{out} U$ 
and  $\cL^s_X=W^s(\Lambda)\cap\partial^{in} U$. 
The Proposition is a consequence of the 
definitions together with the following fact: if 
$O$ is a periodic orbit of $X$, the holonomy of 
the two-dimensional lamination $W^u(\Lambda)$ 
along $O$ (where the orbit $O$ is equipped with 
its natural orientation induced by $X$) is 
a contraction, and the holonomy of the 
two-dimensional lamination $W^s(\Lambda)$ along $O$ 
is a dilation.
\end{proof}

\begin{figure}[ht]
\begin{center}
\includegraphics[totalheight=4.5cm]
{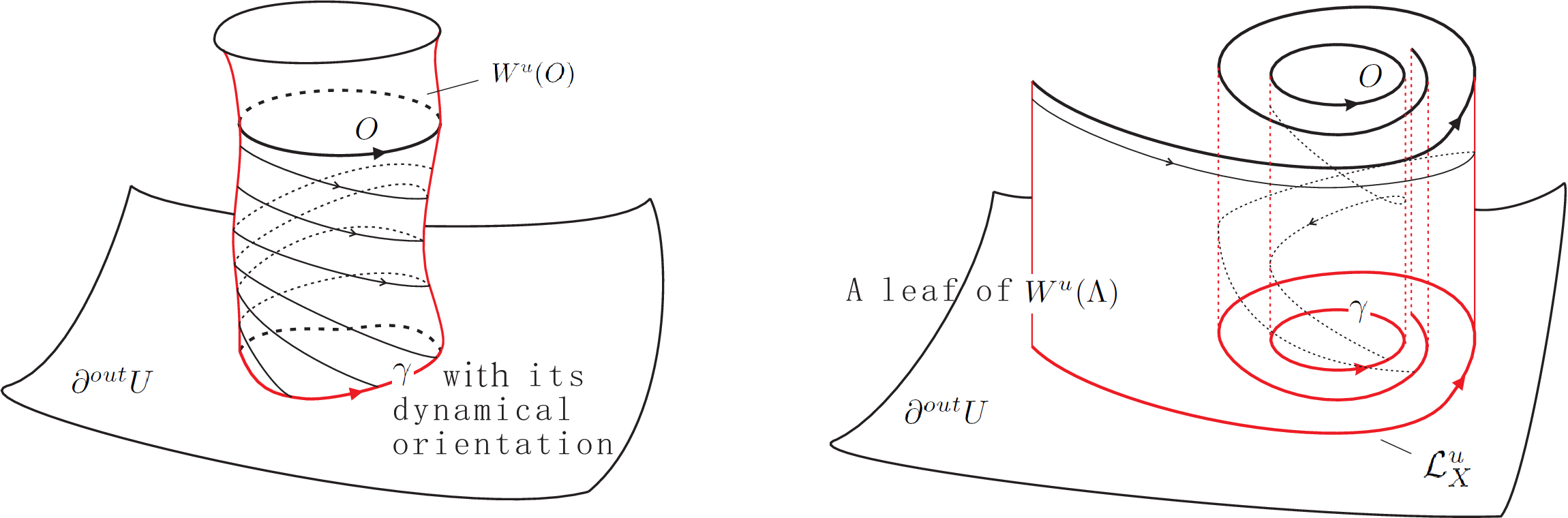}
\caption{\label{f.flow-orientation}
Dynamical orientation (left). Proof 
of Proposition~\ref{p.flow-vs-contracting} (right).}
\end{center}
\end{figure}

\subsection{Simplification of a MS-foliation}

The following elementary proposition provides a kind of ``normal form" for a filling MS-lamination of $\TT^2$.

\begin{prop}
\label{p.simplification}
Let $\cL$ be a filling MS-lamination of class $C^1$ on the torus $\TT^2$. Let $\gamma_0,\dots,\gamma_{n-1}$ be a geometrical enumeration
of its compact leaves and $\sigma\colon\{\gamma_0,\dots,\gamma_{n-1}\}\to \{+,-\}$ be a combinatorial type. Denote $\sigma_i:=\sigma(\gamma_i)\in\{+,-\}$.
We endow $\TT^2=\RR^2/ZZ^2$ with its standard eulidean coordinates, and we assume that $\gamma_0$ is isotopic to $\{0\}\times \SS^1$.  Then there is a 
diffeomorphism $\varphi\colon \TT^2\to \TT^2$ isotopic to the identity map,  so that the lamination $\varphi_*(\cL)$ has the following properties: 
\begin{itemize}
 \item $\varphi_*(\gamma_i)=\{\frac in\}\times \SS^1$;
 \item on the annulus $(\frac in,\frac{i+1}n)\times\SS^1$,  the leaves of the lamination $\varphi_*(\cL)$ are graphs of  $C^1$ functions from $(\frac in,\frac{i+1}n)$ to $\SS^1$; moreover, the derivative of these functions are:
 \begin{itemize}
 \item positive on the whole interval $(\frac in,\frac{i+1}n)$  if $\sigma_i<0$ and $\sigma_{i+1}>0$;
 \item negative on the whole interval $(\frac in,\frac{i+1}n)$  if $\sigma_i>0$ and $\sigma_{i+1}<0$;
 \item positive on $(\frac in,\frac in+\frac12)$ and negative on $(\frac in+\frac12,\frac{i+1}n)$  if $\sigma_i<0$ and $\sigma_{i+1}<0$;
  \item negative on $(\frac in,\frac in+\frac12)$ and positive on $(\frac in+\frac12,\frac{i+1}n)$ if $\sigma_i>0$ and $\sigma_{i+1}>0$.
\end{itemize}
\end{itemize}
\end{prop}

Proposition~\ref{p.simplification} can be thought as a kind of "differentiable version" of Proposition~\ref{p.combinatorial-type} (indeed, Proposition~\ref{p.combinatorial-type} shows that a MS-foliation is \emph{fully} characterized \emph{up to topological equivalence} by its combinatorial types, whereas Proposition~\ref{p.simplification} shows that a $C^1$ filling MS-lamination is \emph{partially} characterized \emph{up to differentiable equivalence} by its combinatorial types. 

\begin{proof}[Idea of the proof]
The proof of Proposition~\ref{p.simplification} roughly follows the same scheme as those of Proposition~\ref{p.combinatorial-type}. One first embeds $\cL$ in a MS-foliation $\cF$, using Lemma~\ref{l.embed-in-foliation}. Then one chooses a diffeomorphism mapping the compact leaf $\gamma_i$ on $\{\frac in\}\times \SS^1$ for every $i$. To get the normal form on a small tubular neighbourhood $U_i$ of the compact leaf $\gamma_i$, one uses the fact that a foliation of a surface is $C^1$-equivalent on a neighbourhood of a compact leaf to the suspension of the holonomy of this compact leaf. To conclude, it remains to get the announced normal form on a compact annulus $A_i$ lying between the tubular neighbourhoods $U_i$ and $U_{i+1}$; these is an easy task since the restriction of $\cF$ to the compact annulus $A_i$ is a trivial foliation by segments joining one boundary component of $A_i$ to the other one. We leave the details to the reader.
\end{proof}
\section{The ``blow-up, excise and glue surgery"}
\label{s.surgery}

 The purpose of this section is to describe the ``blow-up, excise and glue surgery" which was sketched
 in the introduction. As immediate applications, we will prove theorems~\ref{t.richer}
 and~\ref{t.both-transitive-non-transitive}.

\subsection{DA bifurcations}
\label{ss.DA}

In his seminal paper~\cite{Sm}, S. Smale constructed one of the first examples of surface diffeomorphism
displaying a one-dimensional hyperbolic attractor. This diffeomorphism was obtained by bifurcating a linear
Anosov diffeomorphism of $\TT^2$. Smale's construction is known as a \emph{DA bifurcation}
\footnote{``DA" stands for ``derived from Anosov".}. Since then DA bifurcations have been generalized
to various contexts, including axiom A vector fields in dimension~3 (a good reference for this purpose
is~\cite[subsection 2.2.2]{GHS}).

Given some hyperbolic plug $(U,X)$, one can build another hyperbolic plug $(U',X')$ by performing a DA
bifurcation on a periodic orbit of $X$ and excising a small tubular neighborhood of this orbit. We shall
describe this operation in details.

\subsubsection{Attracting DA bifurcation on a periodic orbit with positive multipliers}
\label{sss.DA}

We consider a hyperbolic plug\footnote{Note that the entrance boundary $\partial^{in} U$ or/and the exit
boundary $\partial^{out} U$ can be empty.} with filling MS-laminations $(U,X)$, and a periodic orbit $O$ of
the vector field $X$. We assume that $O$ has positive multipliers. To avoid dealing with some particular
cases, we assume moreover that $O$ has no free separatrix\footnote{Recall that a \emph{stable separatrix}
of $O$ is a connected component of $W^s_X(O)\setminus O$. Since $O$ has positive multipliers, $W^s_X(O)$
is a cylinder, and $O$ has two stable separatrices. A stable separatrix is said to be \emph{free} if it
is disjoint from the maximal invariant set of $(U,X)$. \emph{Free unstable separatrices} are defined in
the same way. Note that the assumption ``$O$ has no free separatrix" is not very restrictive since it is
satisfied by every but finitely many periodic orbits. See the proof of
Proposition~\ref{p.hyperbolicplug}.}. We denote by $\Lambda$ the maximal invariant set of
$(U,X)$.

The vector field $X$ is structurally stable. Therefore, up to perturbing $X$ within its topological
equivalence class, we can assume that $X$ is $C^1$-linearizable on a neighborhood of the periodic
orbit~$O$. This means that there exists a coordinate system
$(x,y,\theta) : V\to [-1,1] \times[-1,1]\times \RR/\ZZ$, defined on a neighborhood
$V\subset \mathrm{int}(U)$ of the orbit $O$, such that
$$X(x,y,\theta)=\lambda\, x\, \frac{\partial}{\partial x}+\mu\, y\,
\frac{\partial}{\partial y}+\frac{\partial}{\partial\theta}$$
for some constant $\lambda<0<\mu$. For $0<\eta<1$, we consider the vector field $X$ which vanishes on $U\setminus V$,
and which is defined on $V$ by:
$$Y_\eta(x,y,\theta)=-2\,\mu\,y\,\phi(x/\eta)\,\phi(y/\eta)\,\frac{\partial}{\partial y},$$
where $\phi:[-1,1]\to \RR^+$ is the bump function defined by $\phi(t)=(1-t^2)^2\mathbf{1}_{[-1,1]}(t)$.
Then we consider the vector field
$$X':=X+Y_\eta.$$
An straightforward computation shows that $O$ is an attracting hyperbolic periodic orbit for the vector
field $X'$.  We say that \emph{$X'$ is derived from $X$ by an attracting DA bifurcation on the orbit $O$}.

\begin{figure}
\begin{center}
\includegraphics[totalheight=4.6cm]{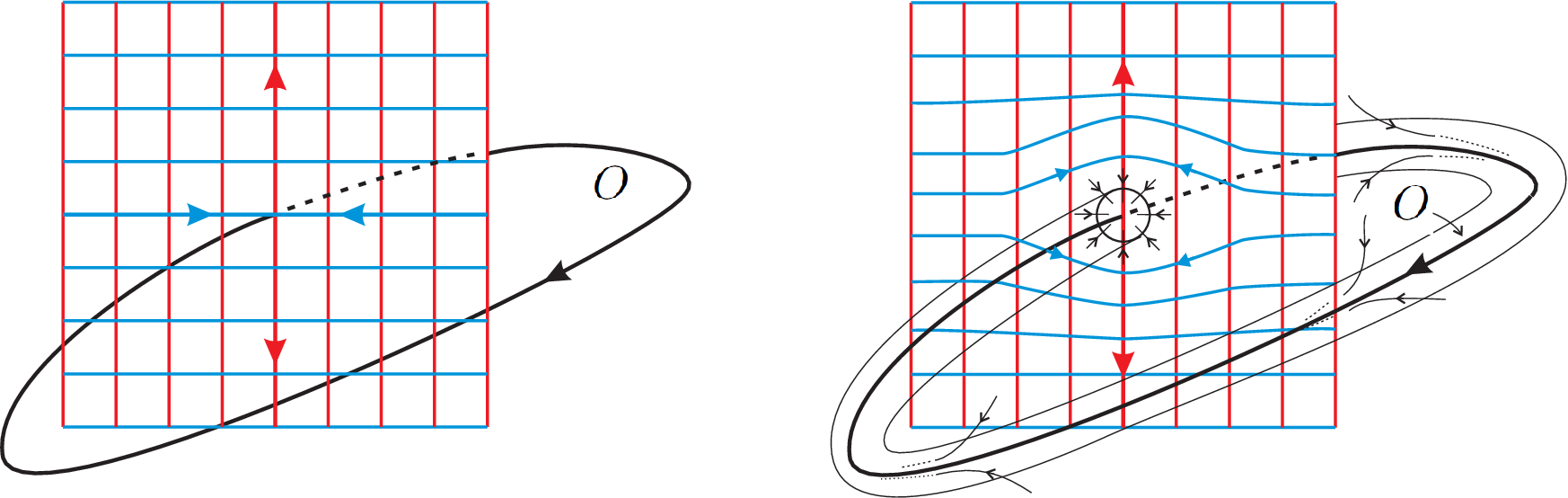}
\caption{\label{f.DA}An attracting DA bifurcation}
\end{center}
\end{figure}

Now we pick a (small) real number $\epsilon>0$, and we consider the solid torus $T\subset V$ defined by
$$T:=\{(x,y,\theta)\, ,\, x^2+y^2<\epsilon^2\}.$$
Obviously, $T$ is a tubular neighborhood of the periodic orbit $O$. We assume that $\epsilon$ is small
enough, so that the two following properties hold: $T$ is included in the basin of attraction (for $X'$)
of $O$, and $X'$ is transverse to $\partial T$. We consider the manifold with boundary
$$U':=U\setminus T$$
endowed with the vector field $X'$. Note that $X'$ is transverse to $\partial U'$ (since $X'$ is transverse
to $\partial T$ and $X'=X$ on $\partial U'\setminus\partial T=\partial U$). Hence $(U',X')$ is a plug.
The following proposition summarizes the relatitionships between the plugs $(U,X)$ and $(U',X')$:

\begin{prop}
\label{p.properties-DA}
The plug $(U',X')$ satisfies the following properties:
\begin{enumerate}
\item $U'=U\setminus T$ where $T$ is a tubular neighborhood of a periodic orbit of $X$;
\item $\partial^{in} U'=\partial^{in} U$ and $\partial^{out} U'=\partial^{out} U\cup \partial T$;
\end{enumerate}
Moreover, if $\eta$ is small enough,
\begin{enumerate}
  \setcounter{enumi}{2}
\item $(U',X')$ is a hyperbolic plug: the maximal invariant set $\Lambda'$ of $(U',X')$ is a saddle
hyperbolic set;
\item $(X')_{|\Lambda'}$ is a topological extension of $X_{|\Lambda}$: there exists a continuous onto
map $\pi:\Lambda'\to\Lambda$ inducing a semi-conjugacy between a reparametrization of the flow of $X'$
and the flow of $X$. Moreover, $\pi$ is ``almost one-to-one": the set $\pi^{-1}(x)$ is a single point
for every $x\in \Lambda\setminus W^s_X(O)$;
\item if $(U,X)$ is a transitive plug, then so is $(U',X')$;
\item $\cL^s_{X'}$ is a filling MS-lamination, with the same combinatorial types as $\cL^s_X$;
\item $\cL^u_{X'}\cap\partial^{out} U$ is a filling MS-lamination, topologically equivalent to $\cL^u_X$;
\item \label{new-lamination} $\cL^u_{X'}\cap \partial T$ is a filling MS-lamination with two coherently
oriented compact leaves (Figure~\ref{f.lamination-DA}).
\end{enumerate}
\end{prop}

\begin{proof}
Let us start by setting some notations. Recall that the periodic orbit $O$ corresponds to the circle
$(x=y=0)$ in the $(x,y,\theta)$ coordinate system. Recall that $O$ is a saddle hyperbolic orbit for the
vector field $X$, and an attracting hyperbolic orbit for the vector field $X'$. We denote by
$B:=W^s_{X',V}(O)$ the bassin of attraction of $O$ for the vector field $X'$. We denote by $O_\pm$ the
circle $(x=0,y=\pm\delta)$, where $\delta$ is the unique positive solution of the equation
$\phi(y)=\frac{1}{2}$. Straightforward computations show that $O_-$ and $O_+$ are sadlle hyperbolic
periodic orbits for the vector field $X'$. One can easily check that, in the $(x,y,\theta)$ coordinate
system, the local stable manifolds $W^s_{X',V}(O^-)$ and $W^s_{X',V}(O^+)$ are ``horizontal" graphs,
and the local basin $W^s_{X',V}(O)$ is the open band between the graphs $W^s_{X',V}(O^-)$ and
$W^s_{X',V}(O^+)$. It follows that the accessible boundary of $B$ is precisely
$W^s_{X'}(O_-)\cup W^s_{X'}(O_+)$. We set $\widehat B:
=B\cup W^s_{X'}(O_-)\cup W^s_{X'}(O_+)$

\medskip

Item~1 and~2 follow immediately from the construction of $U'$ and $X'$.

\medskip

Item~3 and~4 are consequences of well-known properties of DA bifurctions. Let us give more details.
Using classical techniques of hyperbolic theory, one can prove  that, for $\eta$ small enough
the maximal invariant set $\Lambda'$ of $(U',X')$ is a saddle hyperbolic set. The very rough idea is the following. Denote by $W$
the support of the vector field $Y_\eta=X'-X$, and observe that $W$ is contained in a solid torus which gets thiner 
and thiner when $\eta$ goes to $0$. Therefore, when $\eta$ is very small, every orbit spends a long time outside 
$W$ between two visits of $W$. Therefore, the possible loss of hyperbolicity in $W$ is counterbalanced by the
hyperbolicity outside $W$. See~\cite[section 2.2.2]{GHS} for a detailed proof. Moreover, one can prove that the vector field
$X'$ is a topological extension of the vector field $X$: there exists a continuous onto map $\pi:U\to U$, inducing a semi-conjugacy 
between a reparametrization of the flow of $X'$ and the flow of $X$. Moreover, the map $\pi$ admits a concrete description:  
it ``squashes $\widehat B$ onto $W^s_X(O)$".
More precisely,
\begin{itemize}
\item $\pi$ maps $\widehat B$ on $W^s_X(O)$, and maps $U\setminus \widehat B$ on
$U\setminus W^s_X(O)$;
\item $\pi: U\setminus\widehat B \to U\setminus W^s_X(O)$ is a homeomorphism;
\item for $x\in W^s_X(O)$, the set $\pi^{-1}(x)$ is an arc crossing $\widehat B$ from
$W^s_{X'}(O_-)$ to $W^s_{X'}(O_-)$.
\end{itemize}
In particular, the restriction $\pi:\Lambda'\to\Lambda$ is onto, and
$\pi:\Lambda'\setminus\widehat B\to\Lambda\setminus W^s_X(O)$ is a homeomorphism. See again~\cite[section 2.2.2]{GHS} for a 
detailed proof. 
Item 3 and 4 follow.

\medskip

Let us prove item~5. Assume that $(U,X)$ is a transitive plug. By definition, this means that
$\Lambda$ is a transitive hyperbolic set for $X$. Note that $\Lambda$ is not a single orbit
since we have assumed that the orbit $O$ has no free separatrix. Hence, we can find an orbit
$Q$ of $X$, such that $Q\subset \Lambda\setminus W^s_X(O)$ and such that $Q$ is dense in
$\Lambda$. We have seen above that $\pi:\Lambda'\setminus \widehat B\to \Lambda\setminus W^s_X(O)$
is an homeomorphism. It follows that the set $\Lambda'\setminus \widehat B$ is topologically
transitive for the vector field $X'$. On the other hand, since $O$ has no free unstable separatrix,
$W^s_X(O)$ is accumulated on both sides by leaves of $W^s_X(\Lambda)\setminus W^s_X(O)$.
Using the properties of the map $\pi$, it follows that neither $W^s_{X'}(O_-)$ nor $W^s_{X'}(O_+)$
is isolated in $W^s_{X'}(\Lambda')$. In other words,
$\Lambda'\setminus \widehat B=\Lambda'\setminus \left(W^s_{X'}(O_-)\cup W^s_{X'}(O_+)\right)$ is
dense in $\Lambda'$.
Hence, $\Lambda' $ is transitive (for the vector field $X'$). By definition, this means that
$(U',X')$ is a transitive plug.

\medskip

Let us turn to item~6. Recall that the orbit $O$ has no free separatrix. According to the proof of
Proposition~\ref{p.hyperbolicplug}, this implies that $W^s_X(O)$ does not contain any compact leaf of
the lamination $\cL^s_X$. The map $\pi$ induces a homeomorphism from $\cL^s_{X'}\setminus \widehat B$ to
$\cL^s_{X}\setminus W^s_{X}(O)$. Moreover, if $\gamma$ is a (non-compact) leaf of
$\cL^s_{X}\cap W^s_{X}(O)$, then $\pi^{-1}(\cL)$ is a strip, bounded by two (non-compact) leaves
of $\cL^s_{X'}$, whose interior is contained in the bassin $B$ (hence disjoint from $\cL^s_{X'}$).
Item~6 follows.

\medskip

Now we prove item~7. The surface $\partial^{out} U$ is disjoint from the basin $B$ since every orbit of
$X'$ in $B$ must accumulate on $O$ in the future, and therefore must exit from $U'$ by crossing
$\partial T$. The surface $\partial^{out} U$ is also disjoint from the stable manifolds $W^s_{X'}(O_-)$
and $W^s_{X'}(O_+)$ (since every orbit in $W^s_{X'}(O_-)$ and $W^s_{X'}(O_+)$ accumulates on $O_+$ and
$O_-$ in the future, and therefore remains in $U'$ forever). Hence, $\partial^{out} U$ is disjoint from
$\widehat B=B\sqcup W^s_{X'}(O_-)\sqcup W^s_{X'}(O_-)$. But we know that $\pi$ is a homoemophism on the
complement of $\widehat B$. Hence $\pi$ induces a topological equivalence between the laminations
$\cL^u_{X'}\cap\partial^{out} U$ and $\cL^u_{X}\cap\partial^{out} U=\cL^u_X$.

\medskip

We are left to prove item 8. Let $\gamma_\pm$ be the circle $(x=0,y=\pm\epsilon)$ in the $(x,y,\theta)$
coordinate system. Let $W_+$ be the cylinder $(x=0,y>0)$ and $W_-$ be the cylinder $(x=0,y>0)$. It is easy
to check that the cylinder $W_\pm$ is contained in $W^u_{X'}(O_\pm)$. It follows that the circles
$\gamma_+$ and $\gamma_-$ are compact leaves of the lamination $\cL^s_{X'}\cap \partial T$. On the other
hand, let $\gamma$ be a compact leaf of $\cL^s_{X'}\cap \partial T$. According to the proof of
Proposition~\ref{p.hyperbolicplug}, $\gamma=W\cap \partial T$ where $W$ is a free unstable separatrix of
a periodic orbit $P\subset\Lambda'$. Since $\partial T$ is contained in the bassin $B$, the separatrix $W$
must be contained in $B$, and the orbit $P$ must be contained in the accessible boundary of $B$. But $O_+$
and $O_-$ are only the only periodic orbits in the accessible boundary of $B$. Hence the separatrix $W$
must be equal to either $W_+$ or $W_-$. As a further consequence, the compact leaf $\gamma$
must be equal to either $\gamma_+$ or $\gamma_-$. So we have proved that the circles $\gamma^-$ and
$\gamma^-$ are the only compact leaves of the lamination $\cL^u_{X'}\cap\partial T$. For further use,
note that these compact leaves are not homotopic to $0$ is the torus $\partial T$.

By assumption, the periodic orbit $O$ has no free stable separatrix. Hence  $W^u_X(O)$ is accumulated on
both sides by leaves of $W^u_X(\Lambda)$. Using the properties of the map $\pi$, it follows that
$W^u_{X'}(O_\pm)$ is accumulated on both sides by leaves of $W^u_{X'}(\Lambda')$. As a further consequence,
the compact leaf $\gamma_\pm$ is accumulated on both sides by non-compact leaves of
$\cL^s_{X'}\cap \partial T$.

The surface $\partial T$ is a torus, no compact leaf of $\cL^s_{X'}\cap \partial T$ is homotopic
to $0$, and every compact leaf of $\cL^s_{X'}\cap \partial T$ is accumulated on both sides by
non-compact leaves of $\cL^s_{X'}\cap \partial T$. It follows that every connected component of
$\partial T\setminus \cL^s_{X'}$ is a strip bounded by two leaves of $\cL^s_{X'}$ which are asymptotic
to each other at both ends. Hence, $\cL^s_{X'}\cap \partial T$ is a filling MS-lamination.

Using explicit formula for the vector field $X'$, one easily checks that the dynamical orientation of
the compact leaf $\gamma_\pm$ coincides with the orientation induced by the vector field
$\frac{\partial}{\partial\theta}$. According to Proposition~\ref{p.flow-vs-contracting},
the attracting orientation of the leaf $\gamma_\pm$ coincides with its  dynamical orientation.
It follows that  the attracting orientation of $\gamma_-$ and $\gamma_+$ are coherent. The proof is complete.
\end{proof}

\begin{rema}
\label{r.homotopie-class-compact-leaves-1}
For later use, we note that, in the $(x,y,\theta)$ coordinate system, the two compact leaves of the
lamination $\cL^u_{X'}\cap \partial T$ are the circles $(x=0,y=\pm\epsilon)$. Moreover, the attracting
orientation of these leaves is the orientation induced by the vector field
$\frac{\partial}{\partial\theta}$ (see the proof of Proposition~\ref{p.properties-DA}).
\end{rema}

\begin{figure}[ht]
\begin{center}
\includegraphics[totalheight=4.5cm]{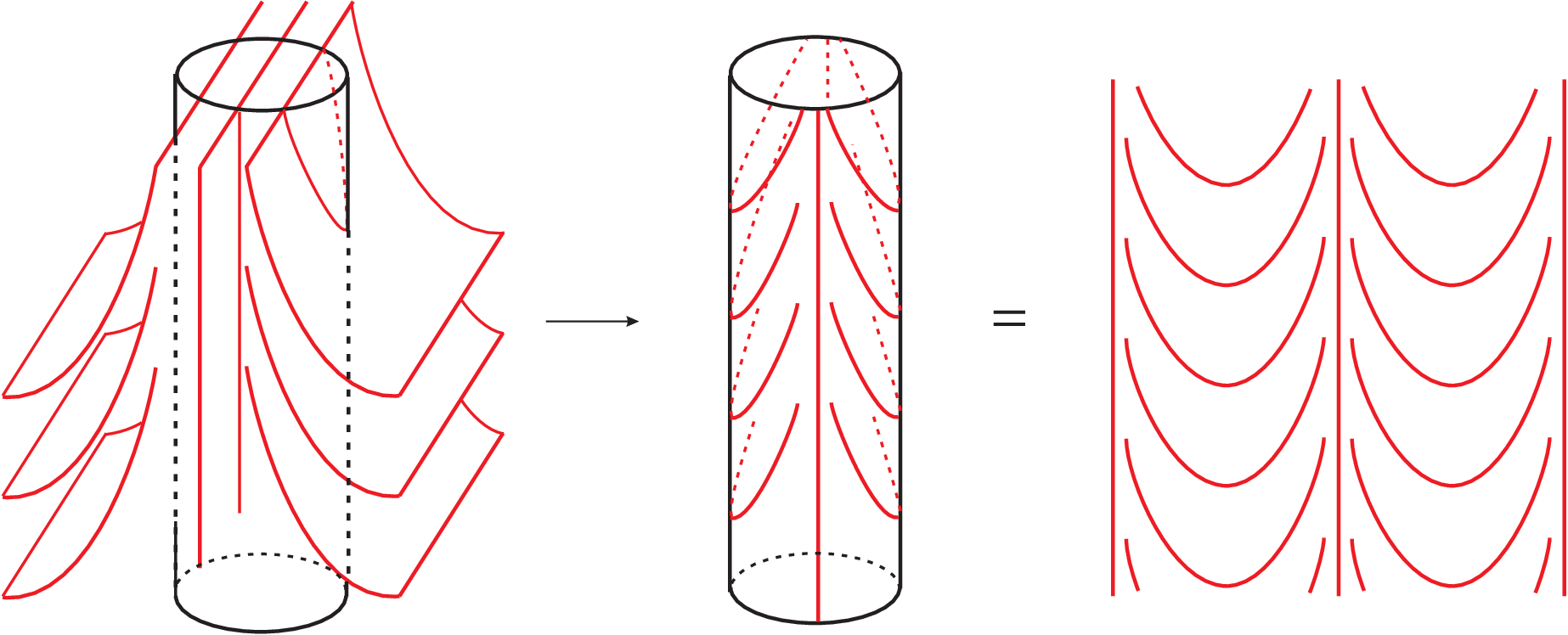}
\caption{\label{f.lamination-DA}The exit lamination $\cL^u_{X'}\cap\partial T$.}
\end{center}
\end{figure}

\subsubsection{Attracting DA bifurcation on orbits with negative multipliers}

In the preceeding paragraph, the orbit $O$ was assumed to have positive multipliers. Actually, we can
also make an attracting DA bifurcation on a periodic orbit $O$ with \emph{negative multipliers}. In
this case, the stable manifold $W^s_X(O)$ is a M\"obius band, and Proposition~\ref{p.properties-DA}
must be replaced by the following statement:

\begin{prop}
\label{p.properties-DA-bis}
Same as proposition~\ref{p.properties-DA}, except for item~\ref{new-lamination}, which is replaced by
\begin{itemize}
\item[\ref{new-lamination}'.] $\cL^u_{X'}\cap \partial T$ is a filling MS-lamination with a single compact
leaf, \emph{i.e.} a zipped Reeb lamination.
\end{itemize}
\end{prop}

\subsubsection{Repelling DA bifurcations}

Instead of an attracting DA bifurcation, it is also possible to make \emph{repelling DA bifurcation}
on a periodic orbit $O$. This bifurcation creates a repelling periodic orbit, instead of an attracting one.
As in \S~\ref{sss.DA}, one can excise a tubular neighborhood of this repelling periodic orbit, and get
a hyperbolic plug $(U',X')$. The properties of this hyperbolic plug are analogous to those listed in
Proposition~\ref{p.properties-DA} and~\ref{p.properties-DA-bis}, after having exchanged the roles of
the stable and the unstable directions, and the roles of the entrance and exit boundaries.
Remark~\ref{r.homotopie-class-compact-leaves-1} must replaced by the following statement:

\begin{rema}
\label{r.homotopie-class-compact-leaves-2}
In the $(x,y,\theta)$ coordinate system, the compact leaves of the lamination
$\cL^s_{X'}\cap \partial T$ are the circles $(x=\pm\epsilon,y=0)$. Moreover, the attracting
orientation of these leaves coincides with the orientation induced by the vector field
$-\frac{\partial}{\partial\theta}$.
\end{rema}

\subsection{The ``blow-up, excise, and glue surgery"}
\label{ss.blow-up-excise-glue}

We will now explain what we call the ``blow-up, excise, and glue surgery", and prove
Theorem~\ref{t.richer}. We shall need the following lemma~:

\begin{lemm}
\label{l.gluing-diffeo}
Let $\cL_1,\cL_2$ be filling MS-laminations on $\TT^2$, with the same number of compact leaves.
Assume that the all the compact leaves of $\cL_1$ (resp. $\cL_2$) are coherently oriented
(see Definition~\ref{d.coherent-incoherent}). Then there exists an orientation preserving
diffeomorphism $\phi:\TT^2\to\TT^2$ such that $\phi_*(\cL_1)$ is strongly transverse to
$\cL_2$. If $\cL_1=\cL_2$, then $\phi$ can be chosen isotopic to the identity.
\end{lemm}

\begin{proof}
Let $n$ be the number of compact leaves of $\cL_1,\cL_2$. According to Proposition~\ref{p.simplification}, there exists two diffeomorphisms
$\phi_1,\phi_2:\TT^2\to\TT^2$ such that, for $i=1,2$, 
\begin{itemize}
\item the compact leaves of the lamination $(\phi_i)_*(\cL_i)$ are the vertical circles $\{\frac in\}\times \SS^1$ for $i=0,\dots,n-1$;
 \item in the open annulus $(\frac in,\frac{i+1}n)\times \SS^1$, the leaves of the lamination $(\phi_i)_*(\cL_i)$ are graphs of  $C^1$ functions from $(\frac in,\frac{i+1}n)$ to $\SS^1$; moreover the derivative of these functions is strictly positive  on $(\frac in,\frac in+\frac12)$, stricly negative on $(\frac in+\frac12,\frac{i+1}n)$ and vanishes precisely on $\{\frac in+\frac12\}$.
 \end{itemize}
 Let $\phi$ be the diffeomorphism of $\TT^2$ given by $\psi_0(x,y)=(x+\frac{1}{2n},y)$. Let $\phi:=(\phi_2)^{-1}\circ\psi\circ \phi_1$. One easily checks that $\phi_*(\cL_1)$ is strongly transverse to $\cL_2$. If $\cL_1=\cL_2$, then one may take $\phi_1=\phi_2$ which implies that $\phi$ is isotopic to the identity.
 \end{proof}
 
We begin with an Anosov vector field $X$ on a closed three-manifold $M$. We consider two distinct periodic
orbits $O,O'$ of $X$,  both of which have positive multipliers\footnote{The surgery can also be made if $O$
and $O'$ both have negative multipliers, but we will not need it}. We proceed as follows:
\begin{itemize}
\item \textbf{Step 1. Blow-up.} We consider a vector field $X'=X'_\eta$ on $M$, derived from $X$ by an attracting
DA bifurcation on the orbit $O$ (see subsection~\ref{ss.DA}). Note that $O$ is an attracting hyperbolic
periodic orbit for this new vector field $X'$.
\item \textbf{Step 2. Excise.} As in subsection~\ref{ss.DA}, we consider a tubular neighborhood $T$ of
the attracting orbit of $O'$, so that $T$ is included in the bassin of attraction of $O'$, and so that
$X'$ is transverse to $\partial T$. We set $U':=M\setminus T$. Clearly, $(U',X')$ is a repelling hyperbolic
plug with $\partial^{out} U'=\partial T$.  According to item~8 of Proposition~\ref{p.properties-DA},
$\cL^u_{X'}$ is a filling MS-lamination with two coherently oriented compact leaves. We denote by $\Lambda'$
the maximal invariant set of $X'$.  According item~4 of Proposition~\ref{p.properties-DA}, there is a
continous onto map $\pi:\Lambda'\to M$ inducing a semi-conjugacy between a reparametrization of the flow
of $X'$ and the flow of $X$. Since $O'\nsubseteq W^s_X(O)$, $\pi^{-1}(O')$ is a periodic orbit of $X'$.
In other words, we can (and we will) regard $O'$ as an orbit of $X'$.
\item \textbf{Step 1'. Blow-up.} Now, we consider a vector field $X''$ on $U'$, derived from $X'$ by an
repelling DA bifurcation on the orbit $O'$. Note that $O'$ is a repelling hyperbolic periodic orbit for
$X''$.
\item \textbf{Step 2'. Excise.}  We consider a tubular neighborhood $T'$ of the repelling orbit of $O'$,
so that $T'$ is included in the bassin of repulsion of $O'$, and so that $X''$ is transverse to
$\partial T'$.  We set $U'':=U'\setminus T'$. Then $(U'',X'')$ is a hyperbolic plug with
$\partial^{out} U''=\partial^{out} U'=\partial T$ and $\partial^{in} U''=\partial T'$.
We denote by $\Lambda''$ the maximal invariant set of $X'$. According to item~6 of
Proposition~\ref{p.properties-DA}, the lamination $\cL^u_{X''}$ is a filling MS-lamination
with the same combinatorial type as $\cL^u_{X'}$, \emph{i.e.} $\cL^u_{X''}$ has two coherently
oriented compact leaves. According to item~8 of Proposition~\ref{p.properties-DA}, the lamination
$\cL^s_{X''}$ is also a filling MS-lamination with two coherently oriented compact leaves.  According
item~4 of Proposition~\ref{p.properties-DA}, there is a continous onto map $\pi':\Lambda''\to \Lambda'$
inducing a semi-conjugacy between a reparametrization of the flow of $X''$ and
the flow of $X'$.
\item \textbf{Step 3. Glue.} The laminations $\cL^s_{X''}$ and $\cL^u_{X''}$ satisfy the hypothesis
of Lemma~\ref{l.gluing-diffeo}. Hence we can find an orientation-preserving diffeomorphism of
$\phi:\partial^{out} U''\to \partial^{in} U''$  such that $\phi_*(\cL^u_{X''})$ is strongly
transverse to $\cL^s_{X''}$. We consider the closed manifold $N:=U''/\phi$ and the vector
field $Z$ induced by $X''$ on $N$.  According to Theorem~\ref{t.transitive}, up to modifying
$X''$ by a topological equivalence and $\phi$ by a strongly transverse isotopy, $Z$ is Anosov.
\end{itemize}
We say that the Anosov vector field $(N,Z)$ are derived from the Anosov vector field $(M,X)$ by a
\emph{``blow-up, excise and glue surgery"}.


\begin{lemm}
\label{l.transitive}
If $X$ is transitive, then so is $Z$.
\end{lemm}

\begin{proof}
Assume that $X$ is transitive. According to item~5 of Proposition~\ref{p.properties-DA},
it follows that $X''_{|\Lambda''}$ is also transitive. Now, using Proposition~\ref{p.transitive},
we deduce that $Z$ is transitive.
\end{proof}

\begin{lemm}
\label{l.richer}
The dynamics of the new vector field $Z$ is ``richer" than the dynamics of the initial vector field $X$.
More precisely, there exists a compact subset $\Lambda''$ of $N$, which is invariant under the flow of $Z$,
and a continuous onto map $\pi'\circ\pi:\Lambda''\to M$  inducing a semi-conjugacy between some
reparametrization of the flow of $Z_{|\Lambda''}$ on the flow of $X$.
\end{lemm}

\begin{proof}
The set $\Lambda''$, the maps $\pi'$ and $\pi$ were defined above. Observe that $\Lambda''$ can
indeed be seen as a subset of $N$, since $\Lambda''\subset \mathrm{int}( U'')\subset N$. Moreover,
the vector field $Z$ coincides with $X''$ on $\Lambda''$. The lemma follows from the properties of the
maps $\pi$ and $\pi'$.
\end{proof}

\begin{proof}[Proof of Theorem~\ref{t.richer}]
The theorem immediately follows from the construction above, lemma~\ref{l.transitive} and
lemma~\ref{l.richer}.
\end{proof}

\subsection{A transitive and a non-transitive Anosov vector field on the same manifold}

The ``blow-up, excise and glue" surgery described in the previous paragraph admits many variants.
We shall use one of these variants to prove Theorem~\ref{t.both-transitive-non-transitive}, \emph{i.e.}
to construct a closed three-manifold $N$ supporting both a non-transitive Anosov vector field $Y$ and a
transitive vector field $Z$.

\begin{proof}[Proof of Theorem~\ref{t.both-transitive-non-transitive}]
We start with a transitive Anosov vector field $X$ on a closed manifold $M$. We pick two periodic orbits
$O,O'$ of $X$ with positive multipliers. Then we consider four vector fields $X_1,\dots,X_4$ on $M$
which are derived from $X$ by DA bifurcations on $O$ and $O'$. More precisely:
\begin{itemize}
\item $X_1$ is obtained by an attracting DA bifurcation on $O$ and an attracting DA bifurcation on $O'$;
\item $X_2$ is obtained by  a repelling DA bifurcation on $O$ and a repelling DA bifurcation on $O'$;
\item $X_3$ is obtained by an attracting DA bifurcation on $O$ and a repelling DA bifurcation on $O'$;
\item $X_4$ is obtained by a repelling DA bifurcation on $O$ and a attracting DA bifurcation on $O'$.
\end{itemize}
Observe that $O$ is an attracting orbit for $X_1$ and $X_3$ and a repelling periodic orbit for $X_2$ and
$X_4$, whereas $O'$ is an attracting orbit for $X_1$ and $X_4$ and a repelling orbit for $X_2$ and $X_3$.
We can find some tubular neighborhoods $T$ and $T'$ of $O$ and $O'$ respectively, so that $T$ and $T'$ are
contained in the bassins\footnote{with respect to each of the four vector fields $X_1,\dots,X_4$} of $O$
and $O'$ respectively, and so that the four vector fields $X_1,\dots,X_4$ are transverse to $\partial T$
and $\partial T'$. We consider the manifold with boundary
$U:=M\setminus (\mathrm{int}(T)\cup\mathrm{int}(T'))$. Note that $(U,X_1)$, $(U,X_2)$, $(U,X_3)$ and
$(U,X_4)$ are hyperbolic plugs (item~3 of Proposition~\ref{p.properties-DA}).

\medskip

We construct a non-transitive Anosov vector field $Y$ by gluing the hyperbolic plugs $(U,X_1)$ and
$(U,X_2)$. The periodic orbits $O$ and $O'$ are attracting for $X_1$. Hence $(U,X_1)$ is a repelling
hyperbolic plug: $\partial^{out}_{X_1} U=\partial U=\partial T\cup \partial T'$.  On the contrary, the
periodic orbits $O$ and $O'$ are repelling for $X_2$.  Hence $(U,X_2)$ is an attracting hyperbolic plug:
$\partial^{in}_{X_2} U=\partial U=\partial T\cup \partial T'$. According to item~8 of
Proposition~\ref{p.properties-DA},  $\cL^s_{X_1}\cap \partial T$
(resp. $\cL^s_{X_1}\cap \partial T'$, $\cL^u_{X_2}\cap T$ and $\cL^u_{X_2}\cap T'$) is a MS-foliation
with two coherently oriented compact leaves. Lemma~\ref{l.gluing-diffeo} provides an orientation-preserving
diffeomorphism
$$\phi:\partial^{out}_{X_2} U=\partial U\longrightarrow \partial^{in}_{X_1} U=\partial U$$
such that $\phi_*(\cL^u_{X_1})$ is transverse to $\cL^s_{X_2}$. We consider the closed manifold
$N_\phi:=(U\sqcup U)/\phi$. The vector fields $X_1$ and $X_2$ induce a vector field $Y$ on $N_\phi$.
According to Proposition~\ref{p.plug}, $(N_\phi,Y)$ is a hyperbolic plug. Since
$\partial N_\phi=\emptyset$, this means that $Y$ is an Anosov vector field. Note that $Y$ is not
transitive, since $(N_\phi,Y)$ was constructed by gluing an attracting plug and a repelling plug.

\medskip

Now, we construct a transitive Anosov vector field $Z$ by gluing the hyperbolic plugs $(U,X_3)$ and
$(U,X_4)$. Recall that $O$ is a attracting orbit for $X_3$ and a repelling orbit for $X_4$, whereas $O'$
is a repelling orbit for $X_3$ and an attracting orbit for $X_4$. Therefore,
$\partial^{in}_{X_3} U=\partial^{out}_{X_4}U=\partial T'$ and
$\partial^{out}_{X_3} U=\partial^{in}_{X_4} U=\partial T$. According to item~8 of
Proposition~\ref{p.properties-DA},  $\cL^s_{X_3}$ (resp. $\cL^u_{X_3}$, $\cL^s_{X_4}$
and $\cL^u_{X_4}$ is a filling MS-lamination with two has two coherently oriented compact leaves.
Lemma~\ref{l.gluing-diffeo} provides an orientation-preserving diffeomorphism
$$\psi: (\partial^{out}_{X_3} U\sqcup\partial^{out}_{X_4} U)=
\partial U\longrightarrow (\partial^{in}_{X_3} U\sqcup\partial^{in}_{X_4} U)=\partial U$$
such that $\psi_*(\cL^u_{X_3})$ is strongly transverse to $\cL^s_{X_4}$ on $\partial T$,
and $\psi_*(\cL^u_{X_4})$ is strongly transverse to $\cL^s_{X_3}$ on $\partial T'$. We
consider the closed manifold $N_\psi:=(U\sqcup U)/\psi$. The vector fields $X_3$ and $X_4$
induce a vector field $Z$ on $N_\psi$. According to Theorem~\ref{t.transitive}, up to
perturbing $X_3$ and $X_4$ by a small topological equivalence and $\psi$ by a strongly
transverse isotopy, we can assume that $Z$ is an Anosov vector field.

Let us prove that the vector field $Z$ is transitive. Since the initial Anosov vector
field $X$ is transitive, item~5 of Proposition~\ref{p.properties-DA} ensures that the
maximal invariant set $\Lambda_3$ of $(U,X_3)$ and the maximal invariant set $\Lambda_4$
both are transitive. Moreover,
$\psi_*(\cL^u_{X_4})=\psi_*(W^u_{X_4}(\Lambda_4))\cap\partial^{out}_{X_4} U$ intersects
$\cL^s_{X_3}=W^s_{X_3}(\Lambda_3)\cap\partial^{in}_{X_3} U$, and
$\psi_*(\cL^u_{X_3})=\psi_*(W^u_{X_3}(\Lambda_3))\cap\partial^{out}_{X_3} U$ intersects
$\cL^s_{X_4}=W^s_{X_4}(\Lambda_4)\cap\partial^{in}_{X_4} U$. Hence, Proposition~\ref{p.transitive}
guarantees that $Z$ is transitive.

\medskip

The proof of Lemma~\ref{l.gluing-diffeo} together with Remarks~\ref{r.homotopie-class-compact-leaves-1}
and~\ref{r.homotopie-class-compact-leaves-2} imply that the maps $\phi$ and $\psi$ can be
choosen isotopic to the $-\mathrm{Id}$ on each of the two connected components of
$\partial U\simeq\TT^2\sqcup\TT^2$. As a consequence, the gluing maps $\phi$ and $\psi$ can be choosen in
such a way that the manifolds $N_\phi$ and $N_\psi$ are diffeomorphic. As a further consequence, $Y$ and
$Z$ can be regarded as vector fields on the same manifold $N$. The proof is complete.
\end{proof}

\begin{figure}[ht]
\begin{center}
\includegraphics[totalheight=5.7cm]{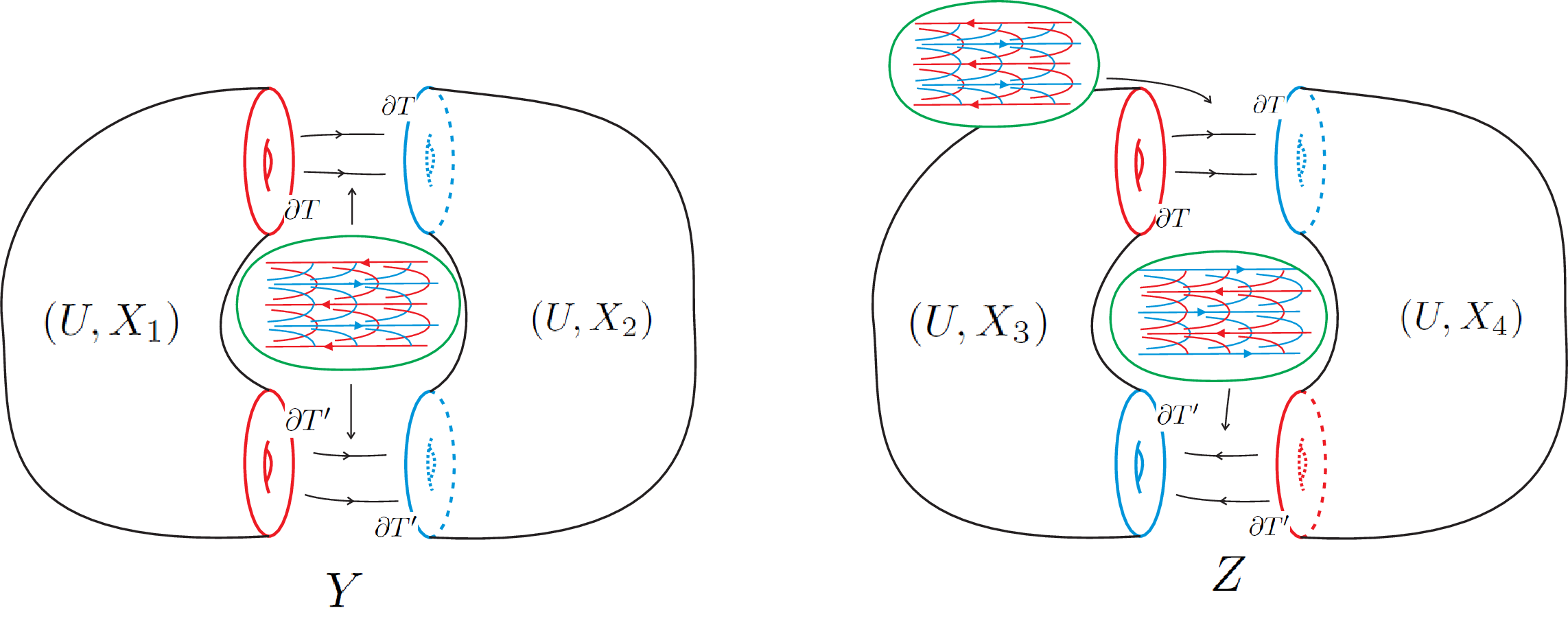}
\caption{\label{f.transitive-non-transitive}Construction of the Anosov vector fields $Y$ and $Z$}
\end{center}
\end{figure}

\section{Attractors with prescribed entrance foliation}
\label{s.attractors}

Let $(V,X)$ be an oriented  plug (\emph{i.e.} a hyperbolic plug such that  $V$ is oriented). Then the entrance boundary $\partial^{in} V$ inherits of a canonical orientation, characterized by the following property: if $(e_1,e_2)$ is a basis of the tangent space $T_p\partial^{in} V$ of $\partial^{in} V$ at some point $p$, then $(e_1,e_2)$ is a direct basis of $T_p\partial^{in} V$ if and only if $(e_1,e_2,X(p))$ is a direct basis of $T_p V$. The canonical orientation of the exit boundary $\partial^{out} V$ is defined similarly. 

\begin{defi}
Let $\cF$ be a MS-foliation on a closed oriented surface $S$, and $(U,X)$ be an oriented
attracting hyperbolic plug. If the entrance foliation $\cL^s_X$ is topological equivalent to $\cF$, then we say that $\cF$ is \emph{realized} by the plug 
$(U,X)$.
\end{defi}

The purpose of the present section is to prove Theorem~\ref{t.attractors}, which states that every 
MS-foliation (on a closed oriented surface) can be realized by a transitive attracting hyperbolic plug. 
For pedagogical results, we will first prove a weaker result (Proposition~\ref{p.not-transitive} below).

\subsection{Non-transitive attracting plugs with prescribed entrance foliation}

As a first step towards Theorem~\ref{t.attractors}, we will prove the following proposition:

\begin{prop}
\label{p.not-transitive}
Assume that some orientation of $\TT^2$ has been fixed. Any MS-foliation on $\TT^2$ can be realized 
by a (not necessarily transitive) attracting hyperbolic plug.
\end{prop}

The proof of Proposition~\ref{p.not-transitive} relies on the results of section~\ref{s.simple-foliations}. 
In particular, we will use the fact that every MS-foliation on $\TT^2$ can be obtained by adding 
compact leaves to a zipped Reeb foliation (corollary~\ref{c.adding-leaves}).

\begin{lemm}
\label{l.realizing-zipped}
There exists an oriented (transitive)  attracting hyperbolic plug $(U,X)$, whose entrance boundary $\partial^{in} U$ 
is connected, and whose entrance foliation $\cL^s_X$ is a zipped Reeb foliation 
(see Definition~\ref{d.zipped-Reeb}).
\end{lemm}

\begin{proof}
Let $X_0$ be a transitive Anosov vector field on a closed oriented three-manifold $M$, such that $X_0$ 
has some periodic orbits with negative multipliers. For example, $X_0$ can be the the suspension of the 
linear Anosov automorphism $A:\TT^2\to \TT^2$ defined by $A(x,y)=(-2x-y,-x-y)$.  Choose a periodic orbit 
$O$ of $X_0$, such that $O$ has negative multiplier. Make a repelling DA bifurcation on $O$ 
(see subsection~\ref{ss.DA}). This gives rise to a vector field $X$ on $M$, for which $O$ is a 
repelling hyperbolic periodic orbit. As in subsection~\ref{ss.DA}, consider a tubular neighbourhood 
$T$ of $O$, such that $T$ is contained in the bassin of $O$, and such that $X$ is transverse to 
$\partial T$. Set $U:=M\setminus T$. Proposition~\ref{p.properties-DA-bis} shows that $(U,X)$ satisfies 
the required properties.
\end{proof}

\begin{rema}
\label{r.several-zipped}
More generally, for every $p\geq 1$, there exists an oriented  transitive attracting hyperbolic plug $(U,X)$, whose 
entrance boundary $\partial^{in} U$ has $p$ connected components, and such that the restriction of the 
entrance foliation $\cL^s_X$ to each connected component of $\partial^{in} U$ is a zipped Reeb foliation. 
The construction of the plug $(U,X)$ is similar to those of the proof of Lemma~\ref{l.realizing-zipped}. 
The only difference is that we need to make a DA bifurcation on $p$ periodic orbits instead of a single one.
\end{rema}

The core of the proof of Proposition~\ref{p.not-transitive} is the following lemma:

\begin{lemm}
\label{l.adding-leaf-1}
Let $\cF,\cF'$ be MS-foliations on $\TT^2$. Suppose that:
\begin{itemize}
\item $\cF'$ can be obtained by adding a compact leaf to $\cF$ (in the sense of 
definition~\ref{d.adding-leaves}),
\item $\cF$ is realized by an attracting hyperbolic plug.
\end{itemize}
Then $\cF'$ can also be realized by an attracting hyperbolic plug.
\end{lemm}

During the proof of the Lemma~\ref{l.adding-leaf-1}, we will need the existence of a hyperbolic plug as 
provided by the following lemma:

\begin{lemma}
\label{a.hyperbolic-plug-twisted-orbit}
There exists an oriented hyperbolic plug $(V,Y)$ with the following properties:
\begin{itemize}
\item $V$ is a Seifert bundle over a 2-sphere minus two discs;
\item the maximal invariant set $\bigcap_{t\in\RR} Y^t(V)$ is a saddle hyperbolic periodic orbit $O$ with 
negative multipliers;
\item the entrance boundary $\partial^{in} V$ is a torus, and the entrance lamination 
$\cL^s_{Y}=W^s_Y(O)\cap\partial^{in} V$ is made of a single (closed) leaf $\gamma^s$, which is essential in 
$\partial^{in} V$;
\item the exit boundary $\partial^{out} V$ is a a torus, and the exit lamination 
$\cL^u_{Y}=W^u_Y(O)\cap\partial^{out} V$ is made of a single (closed) curve $\gamma^u$, which is  
essential in $\partial^{in} V$.
\end{itemize}
\end{lemma}

\begin{proof}[Proof of Lemma~\ref{a.hyperbolic-plug-twisted-orbit}]
Consider a gradient-like diffeomorphism $f:\SS^2\to\SS^2$ such that the non-wandering set of $f$ consists 
in a repelling hyperbolic fixed point $r$, a saddle hyperbolic fixed point (with negative eigenvalues) $o$, 
and an attracting hyperbolic periodic orbit of period two $\{a_1,a_2\}$. Denote by $(N,Y)$ the suspension 
of $(\SS^2,f)$. The non-wandering set of $Y$ is made of a repelling hyperbolic periodic orbit $R$, a saddle 
hyperbolic periodic orbit $O$, and an attracting hyperbolic periodic orbit $A$. Let $V$ be the manifold 
with boundary obtained by excising from $N$ some small tubular neighborhoods of the periodic orbits $A$ and 
$R$ whose boundary are transverse to $Y$. Then $(V,Y)$ is a hyperbolic plug, and one can easily check that 
it satisfies the desired properties.
\end{proof}

\begin{proof}[Proof of lemma~\ref{l.adding-leaf-1}]
By assumption, we can find a connected attracting hyperbolic plug  $(U,X)$ realizing the foliation $\cF$. 
Let $(V,X)$ be the hyperbolic plug provided by Lemma~\ref{a.hyperbolic-plug-twisted-orbit}. We have to 
construct a plug $(U',X')$ realizing the foliation $\cF'$. Thus plug will be obtained by gluing together 
the plugs $(U,X)$ and $(V,Y)$. We proceed to the construction.

\bigskip

By assumption, $\cF'$ can be obtained by adding a compact leaf to $\cF$. Since $\cL^s_X$ is topologically 
equivalent to $\cF$, the foliation $\cF'$ can also be obtained by adding a compact leaf to $\cL^s_X$. By 
definition, this means that, one can find a geometrical enumeration $\gamma_1,\dots,\gamma_n$ of the 
compact leaves of $\cL^s_X$ and a geometrical enumeration  $\gamma_1',\dots,\gamma_{n+1}'$ of the compact 
leaves of $\cF'$, so that the corresponding combinatorial types $\sigma:\{1,\dots,n\}\to\{+,-\}$ and 
$\sigma':\{1,\dots,n+1\}\to\{+,-\}$ of $\cF'$ satisfy $\sigma'_{|\{1,\dots,n\}}=\sigma$.

Since $\gamma_1,\dots,\gamma_n$ is a geometrical enumeration of the compact leaves of $\cL^s_X$, the 
compact leaves $\gamma_n$ and $\gamma_1$ bound an open annulus $A\subset \partial^{in} U$ which is disjoint 
from the compact leaves of  $\cL^s_X$ (in the particular case $n=2$, the leaves $\gamma_n$ and $\gamma_1$  
bound two annuli; we denote by $A$ the annulus which is on the left-hand side of $\gamma_1$ with respect to 
the contracting orientation of $\gamma_1$).

According to item~3 of Lemma~\ref{a.hyperbolic-plug-twisted-orbit}, the surface $\partial^{out} V$ is a 
torus and the lamination $\cL^u_{Y}$ is made of a single (closed) leaf $\gamma^u$. We consider a 
diffeomorphism $\phi:\partial^{out} V\to\partial^{in} U$ with the following properties:
\begin{enumerate}
\item[(i)] the curve $\phi_*(\gamma^u)$ is contained in the annulus $A$;
\item[(ii)] the curve $\phi_*(\gamma^u)$ is transverse to the foliation $\cL^s_X$.
\end{enumerate}
Such a diffeomorphism $\phi$ does exist: indeed, the restriction of the foliation $\cL^s_X$ to the annulus 
$A$ is topologically equivalent to the foliation by vertical lines of the annulus $\SS^1\times\RR$. Now, we
glue the plugs $(U,X)$ and $(V,Y)$, using $\phi$ as a gluing map. In other words, we consider the manifold 
with boundary $U':=(U\sqcup V)/\phi$, and the vector field $X'$ on $U'$ induced by $X$ and $Y$. 
Proposition~\ref{p.plug} ensures that $(U',X')$ is a connected attracting hyperbolic plug.

We want to prove that the foliation $\cF'$ is realized by $(U',X')$, \emph{i.e.} that the foliations 
$\cL^s_{X'}$ and $\cF'$ are topologically equivalent. For this purpose, we will use the crossing map
$$\Gamma_V:\partial^{in} V\setminus\gamma^s\longrightarrow \partial^{out} V\setminus \gamma^u.$$
Recall that $\Gamma_V$ maps a point $x\in \partial^{in} V\setminus\gamma^s$ to the unique point of 
intersection of the orbit of $x$ (for the flow of the vector field $Y$) with the surface 
$\partial^{out} V$. As stated in Proposition~\ref{p.gluedlaminations}:
\begin{equation}
\label{e.lamination}
\cL^s_{X'} =  \gamma^s\sqcup (\Gamma_V^{-1})_*(\phi^{-1})_*(\cL^s_X)\setminus\gamma^u).
\end{equation}
The foliation $\cL^s_X$ has $n$ compact leaves $\gamma_1,\dots,\gamma_n$. By definition of the gluing map 
$\phi$, the closed curve $\phi_*(\gamma^u)$ is disjoint from $\gamma_1,\dots,\gamma_n$. Hence, 
the foliation $\cL^s_{X'}$ has $n+1$ compact leaves $\widehat\gamma_1,\dots,\widehat\gamma_{n+1}$ 
where we denote
$$\widehat\gamma_i:=(\Gamma_V^{-1}\circ \phi^{-1})_*(\gamma_1)\mbox{ for }i=1\dots n\mbox{ and }
\widehat\gamma_{n+1}:=\gamma^s.$$
Let $\widehat A:=(\Gamma_V^{-1}\circ\phi^{-1})(A)$. Since $\phi_*(\gamma^u)$ is contained in $A$, 
the map  $\Gamma_V^{-1}\circ \phi^{-1}$ is defined on $\partial^{in} U\setminus A$ and
\eqref{e.lamination} shows that
$$\cL^s_{X'}\cap (\partial^{in} U'\setminus \widehat A)
\mbox{ is topologically equivalent to }\cL^s_X\cap (\partial^{in} U\setminus A).$$
Since the annulus $\partial^{in} U'\setminus \widehat A$ contains the compact leaves 
$\widehat\gamma_{1},\dots,\widehat\gamma_{n+1}$, this proves that:
\begin{itemize}
\item $\widehat\gamma_1,\dots,\widehat\gamma_{n+1}$ is a geometrical enumeration of the 
compact leaves of $\cL^s_{X'}$,
\item the combinatorial type $\widehat \sigma:\{1,\dots,n+1\}\to \{+,-\}$ of  $\cL^s_{X'}$ 
associated with this geometrical enumeration satisfies 
$\widehat \sigma_{|\{1,\dots,n\}}=\sigma=\sigma'_{|\{1,\dots,n\}}$.
\end{itemize}
We are left to prove that $\widehat\sigma(n+1)=\sigma'(n+1)$. Actually, we will modify the gluing map 
$\phi$ in order to adjust the value of $\widehat\sigma(n+1)$.  Let 
$\tau:\partial^{out} V\to\partial^{out} V$ be an orientation-preserving homeomorphism which maps the 
closed curve $\gamma^u$ on the same curve with the opposite orientation. Observe that 
$\phi\circ\tau$ still satisfies properties~(i) and~(ii); so we may replace $\phi$ by 
$\phi\circ\tau$ in our construction. Replacing $\phi$ by $\phi\circ\tau$ has the following effect:
\begin{itemize}
\item It changes the contracting orientation of $\widehat\gamma_1$. Indeed, the contracting orientation of  
$\widehat\gamma_1$ is the image under $(\Gamma_V^{-1}\circ \phi^{-1})_*$ of the contracting orientation of 
$\gamma_1$, and $\tau_*$ reverses the orientation of $(\phi^{-1})_*(\gamma_1)$ since $\phi^{-1}(\gamma_1)$ 
is in the same free homotopy class as $\gamma^u$).
\item It does not change the contracting orientation of $\widehat\gamma^{n+1}=\gamma^s$. Indeed, 
Proposition~\ref{p.flow-vs-contracting} ensures that the contracting orientation of $\gamma^s$ as a 
leaf of $\cL^s_{X'}$ is opposite to the dynamical orientation of $\gamma^s$. And the dynamical orientation 
of  $\gamma^s$ does not depend on the gluing map.
\end{itemize}
Therefore, up to replacing $\phi$ by $\tau\circ\phi$, we can decide whether the contracting orientation of 
the compact leaves $\widehat\gamma_1$ and $\widehat\gamma_{n+1}=\gamma^s$ are coherent or not. In other 
words, up to replacing $\phi$ by $\tau\circ\phi$, we may assume that $\widehat\sigma(n+1)=\sigma'(n+1)$. 
We have proved that the foliations $\cL^s_{X'}$ and $\cF'$ have the same combintorial type.  According to 
Proposition~\ref{p.combinatorial-type}, this implies that these foliations are topologically equivalent. 
The proof is complete.
\end{proof}

\begin{figure}[ht]
\begin{center}
\includegraphics[totalheight=5cm]{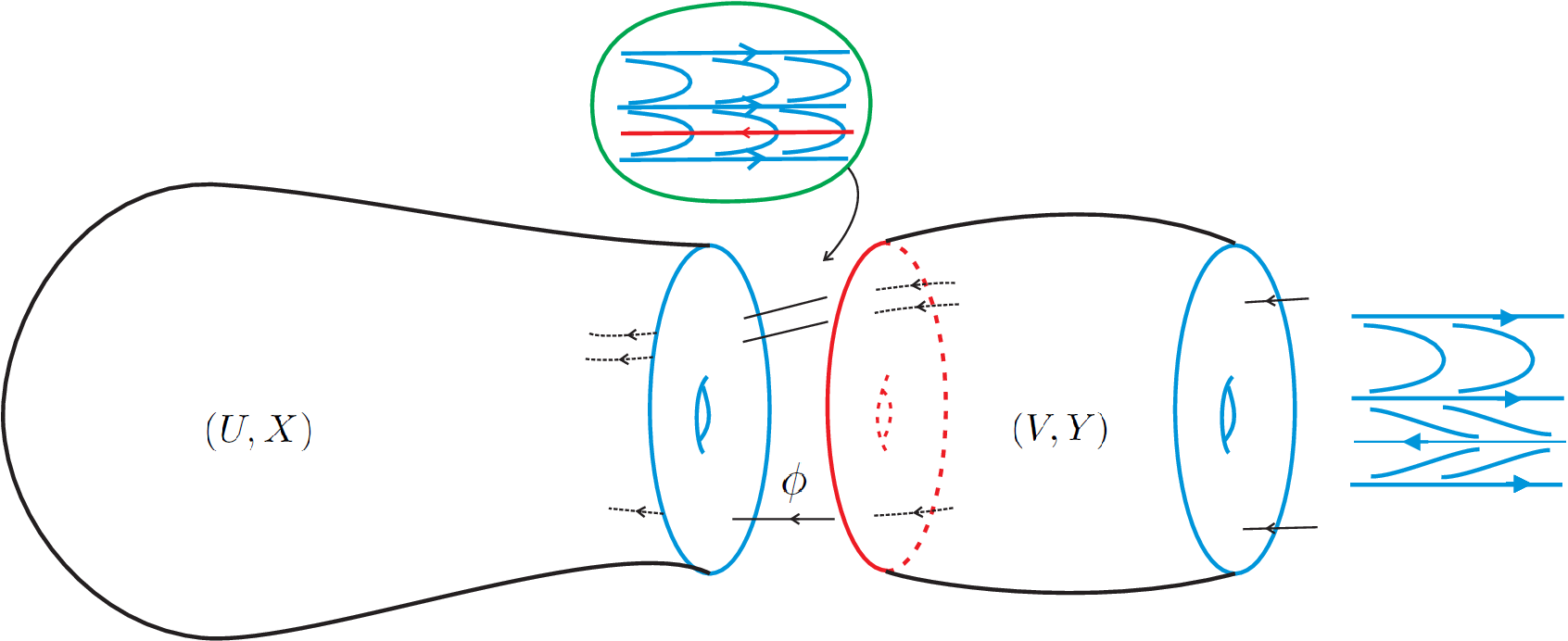}
\caption{Proof of Lemma~\ref{l.adding-leaf-1}}
\end{center}
\end{figure}

\begin{rema}
The attracting plug $(U',X')$ constructed in the proof above is never transitive.
\end{rema}

\begin{proof}[Proof of Proposition~\ref{p.not-transitive}]
The proposition is obtained by putting together corollary~\ref{c.adding-leaves}, 
lemma~\ref{l.realizing-zipped} and lemma~\ref{l.adding-leaf-1}.
\end{proof}

\subsection{Transitive attracting plug with prescribed entrance foliations}

We are now ready to prove Theorem~\ref{t.attractors}. For this purpose, we need an improved version of 
Lemma~\ref{l.adding-leaf-1}:

\begin{lemm}
\label{l.adding-leaf-2}
Let $\cF,\cF'$ be MS-foliations on $\TT^2$. Suppose that:
\begin{itemize}
\item $\cF'$ can be obtained by adding a compact leaf to $\cF$,
\item $\cF$ is realized by a transitive attracting hyperbolic plug $(U,X)$ which has infinitely 
many periodic orbits with negative multipliers.
\end{itemize}
Then $\cF'$ can be realized by a transitive attracting hyperbolic plug $(U',X')$ which has infinitely 
many periodic orbits with negative multipliers.
\end{lemm}

The proof of Lemma~\ref{l.adding-leaf-2} follows the same strategy as those of Lemma~\ref{l.adding-leaf-1}, 
but is slightly more complicated. Instead of using the plug $(V,Y)$ provided by 
Lemma~\ref{a.hyperbolic-plug-twisted-orbit}, we will use a plug $(W,Z)$ with the following 
characteristics:

\begin{lemma}
\label{l.suspension-pair-pants}
There exists a connected oriented hyperbolic plug $(W,Z)$ such that:
\begin{itemize}
\item[$\bullet$] $W$ is diffeomorphic to the product of a pair of pants by a circle;
\item[$\bullet$] the maximal invariant set $\bigcap_{t\in\RR} Z^t(W)$ is an isolated saddle hyperbolic 
periodic orbit (with positive multipliers);
\item[$\bullet$] the entrance boundary $\partial^{in} W$ is made of two tori 
$\partial_1^{in} W,\partial_2^{in} W$; the entrance lamination $\cL^s_{Z}$ is made of two isolated compact leaves 
$\gamma_1^s,\gamma_2^s$; more precisely, $\gamma_1^s$ is an essential simple closed curve in 
$\partial_1^{in} W$ and $\gamma_2^s$ is an essential simple closed curve in $\partial_2^{in} W$;
\item[$\bullet$] the exit boundary $\partial^{out} W$ is a torus, and the exit lamination 
$\cL^u_{Z}$ is made of two closed leaves $\gamma_1^{u},\gamma_2^{u}$, which are essential in the 
torus $\partial^{out} W$; moreover the dynamical orientations of $\gamma_1^{u},\gamma_2^{u}$ coincide.
\end{itemize}
\end{lemma}

\begin{proof}[Proof of Lemma~\ref{l.suspension-pair-pants}]
Consider a gradient-like diffeomorphism $f:\SS^2\to\SS^2$ such that the non-wandering set of $f$ consists 
in two repelling fixed points $r,r'$, one saddle hyperbolic fixed point $o$, and one attracting fixed point 
$a$. Denote by $(N,Z)$ the suspension of $(\SS^2,f)$. The non-wandering set of $Z$ is made of two repelling 
periodic orbits $R,R'$, one saddle hyperbolic periodic orbit $O$, and one attracting periodic orbit $A$. 
Let $W$ be the manifold with boundary obtained by excising from $N$ some small tubular neighborhoods of the 
periodic orbits $R,R',A$. Then $(W,Z)$ is a hyperbolic plug, and one can easily check that it satisfies the 
desired properties.
\end{proof}

\begin{proof}[Proof of Lemma~\ref{l.adding-leaf-2}]
By assumption,  the foliation $\cF$ is realized by a transitive attracting hyperbolic plug $(U,X)$, which 
has infinitely many periodic orbits with negative multipliers. The foliation $\cF'$ can be obtained by 
adding a compact leaf to $\cL^s_X$. This means that we can find a  geometrical enumeration 
$\gamma_1,\dots,\gamma_n$ of the compact leaves of $\cL^s_X$ and a geometrical enumeration 
$\gamma_1',\dots,\gamma_{n+1}'$ of the compact leaves of $\cF'$, so that the corresponding combinatorial 
types $\sigma,\sigma'$ satisfy $\sigma'_{|\{1,\dots,n\}}=\sigma$. We denote by $A$ be the connected 
component of $\partial^{in} U\setminus \cup_i \gamma_i$ which is bounded by the compact leaves 
$\gamma_n$~and~$\gamma_1$, and which is on the left-hand side of $\gamma_1$.

Let $(W,Z)$ be the plug provided by Lemma~\ref{a.hyperbolic-plug-twisted-orbit}, and 
$\Gamma_W:\partial^{in} W\setminus\cL^s_{Z}\rightarrow \partial^{out} W\setminus \cL^u_{Z}$ be the 
crossing map of this plug. Observe that $\partial^{in} W\setminus\cL^s_{Z}$ is the disjoint union of 
the open annuli $A_1^s:=\partial^{in}_1 W\setminus\gamma_1^s$ and 
$A_2^s:=\partial^{in}_2 W\setminus\gamma_2^s$. Therefore, $\partial^{out} W\setminus\cL^u_{Z}$ is the 
disjoint union of the open annuli $A_1^u:=\Gamma_W(A_1^s)$ and $A_2^u:=\Gamma_W(A_2^s)$. Both $A_1^u$ and 
$A_2^u$ are bounded by the compact leaves $\gamma_1^u$ and $\gamma_2^u$.

We consider a gluing diffeomorphism $\phi:\partial^{out} W\to\partial^{in} U$ satisfying the two following 
properties:
\begin{enumerate}
\item $\phi\left(\overline{A_2^u}\right)$ is contained in $A$ (equivalently, $\partial^{in} U\setminus A$ 
is contained in $\phi(A_1^u)$);
\item the compact leaves $\phi_*(\gamma^u_1)$ and $\phi_*(\gamma^u_2)$ are transverse to the foliation 
$\cL^s_X$.
\end{enumerate}
We consider the attracting plug $(U',X')$ obtained by gluing $(W,Z)$ and $(U,X)$ thanks to the gluing map 
$\phi$. Proposition~\ref{p.plug} implies that $(U',X')$ is an hyperbolic plug. Note that 
$\partial^{in} U'=\partial^{in} W=\partial^{in}_1 W\sqcup\partial^{in}_2 W$. Property~1 above implies that 
the pre-image under $\Gamma_W\circ \phi$ of the $n$ compact leaves of $\cL^s_X$ are in $\partial^{in}_1 W$. 
This remark and the same arguments as in the proof of Lemma~\ref{l.adding-leaf-1} show that the foliation 
$\cL^s_{X'}\cap\partial^{in}_1 W$ has the same combinatorial type as $\cF'$. It also shows that 
$\gamma^s_2$ is the only compact leaf of the foliation $\cL^s_{X'}\cap\partial^{in}_2 W$; in other words, 
$\cL^s_{X'}\cap\partial^{in}_2 W$ is a zipped Reeb foliation.

Let us summarize. We have constructed an attracting hyperbolic plug $(U',X')$ with the following 
properties~:
\begin{itemize}
\item the entrance boundary $\partial^{in} U'$ has two connected components $\partial^{in}_1 W$ and 
$\partial^{in}_2 W$,
\item on the first component $\partial^{in}_1 W$, the entrance foliation $\cL^s_{X'}$ is topologically 
equivalent to the foliation $\cF'$,
\item on the other connected component $\partial^{in}_2 W$, the entrance foliation $\cL^s_{X'}$ is a 
zipped Reeb foliation.
\end{itemize}
The plug $(U',X')$ is not transitive. We will use the ``blow-up, excise and glue surgery" of 
section~\ref{s.surgery} to turn $(U',X')$ into a transitive plug.

\bigskip

One hand, the maximal invariant set of $(U,X)$ is a transitive hyperbolic attractor; let us denote it by 
$\Lambda$. On the other hand, the maximal invariant set of $(W,Z)$ is a single hyperbolic periodic orbit 
$O$. Hence $(U',X')$ has two basic pieces: $\Lambda$ and $O$. Observe that $W^u(O)$ intersect 
$W^s(\Lambda)$ in $U'$ (since $\phi(W^u(O)\cap\partial^{out} W)=\phi(\gamma_1^s\cup\gamma_2^s)$ 
intersects $W^s(\Lambda)\cap \partial^{in} U$). By assumption, $(U,X)$ contains some periodic orbits 
with negative multipliers. Choose such a periodic orbit $\Omega$, make a attracting DA bifurcation on 
$\Omega$. This gives rise to a new vector field $X''$ on $U'$ which is a topological extension of $X'$, 
and has three basic pieces: a saddle hyperbolic periodic orbit $O$, a non-trivial saddle basic piece 
$\Lambda'$, and an attracting periodic orbit $\Omega$ (see section~\ref{ss.DA} for more details). 
Let $U''$ be obtained by excising from $U'$ a small tubular neighborhood of the attracting orbit $\Omega$. 
According 
to Proposition~\ref{p.properties-DA}, $(U'',\overline{X'})$ is a hyperbolic plug with the following 
properties:
\begin{itemize}
\item the exit boundary $\partial^{out} U''$ is a torus,
\item the exit lamination $\cL^u_{X''}$ is a zipped Reeb lamination,
\item the entrance boundary $\partial^{in} U''$ coincides with $\partial^{in} U'$,
\item the entrance lamination $\cL^s_{X''}$ has the same combinatorial type as $\cL^s_{X'}$ 
(hence, $\cL^s_{X''}\cap \partial^{in}_1 W$ has the same combinatorial type as $\cF'$ and 
$\cL^s_{X''}\cap \partial^{in}_2 W$ is a zipped Reeb lamination).
\end{itemize}

The laminations $\cL^u_{X''}$ and $\cL^s_{X''}\cap \partial^{in}_2 W$ both are zipped Reeb laminations. 
So, by Lemma~\ref{l.gluing-diffeo}, we can find a strongly transverse gluing map 
$\psi: \partial^{out} U''\to \partial^{in}_2 U''$. We consider the manifold with boundary 
$U''':=U''/\psi$ and the vector field $X'''$ induced by $X''$ on $U'''$. Clearly, $(U''',X''')$ 
is an attracting plug and $\partial^{in} U'''=\partial^{in}_1 W$. According to Theorem~\ref{t.transitive} and Remark~\ref{r.partial-gluing}, 
up to modifying $X''$ by a topological equivalence and $\psi$ by a strongly transverse isotopy, we may 
assume that $(U''',X''')$ is a hyperbolic plug.

The interior of $U''$ is embedded in $U'''$, and $X'''$ coincides with $X''$ in restriction to 
$\mathrm{int}(U'')$. It follows that $\cL^s_{X''}\cap\partial^{in}_1 W$ must be a sublamination of 
the foliation $\cL^s_{X'''}$. Since $\cL^s_{X''}\cap \partial^{in}_1 U''$ is a filling MS-lamination, 
this implies that $\cL^s_{X'''}$ has the same combinatorial type as $\cL^s_{X''}\cap \partial^{in}_1 W$. 
As a further consequence,  $\cL^s_{X'''}$ has the same combinatorial type as $\cF'$. According to 
Proposition~\ref{p.combinatorial-type}, this implies that $\cF'$ is realized by the attracting hyperbolic 
plug $(U''',X''')$.

Clearly, $(U''',X''')$ has infinitely many periodic orbits with negative multipliers. It remains to check 
that the maximal invariant set of $(U''',X''')$ is transitive. The hyperbolic plug $(U'',X'')$ has two
basic pieces: the transitive attractor $\Lambda'$ and the periodic orbit $O$. On the one hand, the unstable
manifold of $O$ intersects the stable manifold of $\Lambda'$. On the other hand, 
$\psi_*(W^u_{X''}(\Lambda')\cap\partial^{out} W)$ intersects $W^s_{X''}(O)\cap\partial^{in}_2 W$, 
since the unique compact leaf of $\cL^s_{X''}\cap\partial^{in}_2 W$ is $\gamma^s_2\subset W^s_{X''}(O)$. 
Therefore, the system $(U'',X'',\psi)$ is combinatorially transitive. 
According to Proposition~\ref{p.transitive}, this implies that $(U''',X''')$ is a transitive plug. 
The proof is complete.
\end{proof}

\begin{figure}[ht]
\begin{center}
\includegraphics[totalheight=4.3cm]{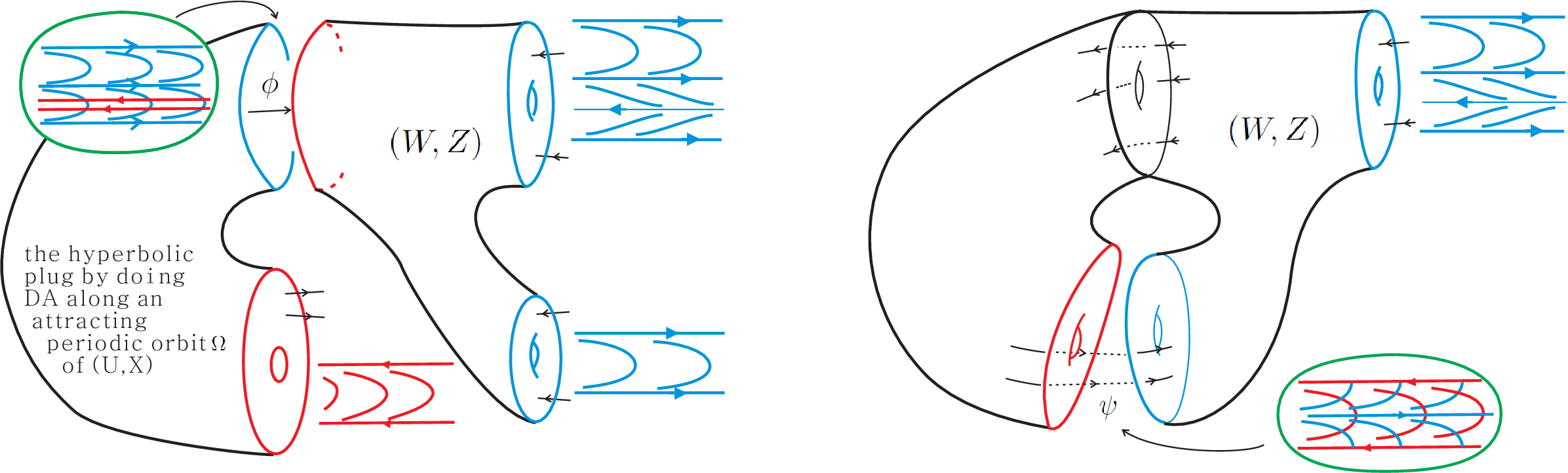}
\caption{Proof of Lemma~\ref{l.adding-leaf-2}}
\end{center}
\end{figure}

\begin{rema}
\label{r.negative-multipliers}
Observe that the plug $(U,X)$ provided by the proof of Lemma~\ref{l.realizing-zipped} has infinitely 
many periodic orbits with negative multipliers.
\end{rema}

\begin{proof}[Proof of Theorem~\ref{t.attractors}]
Let $\cF$ be a MS-foliation on a closed oriented surface $S$. We have to prove that $\cF$ is realized
by a transitive attracting plug. If $S$ is connected (\emph{i.e.} is a torus), this immediately follows 
from  corollary~\ref{c.adding-leaves}, lemma~\ref{l.realizing-zipped}, remark~\ref{r.negative-multipliers} 
and lemma~\ref{l.adding-leaf-2}. If $S$ is has several connected component, the proof is roughly the same, 
except for the fact that we have to use Remark~\ref{r.several-zipped} instead of 
Lemma~\ref{l.realizing-zipped}.
\end{proof}

\begin{rema}
\label{r.repellers}
In $(U,X)$ is an oriented  transitive attracting hyperbolic plug, then $(U,-X)$ is an oriented transitive repelling hyperbolic 
plug, and $\cL^u(U,-X)=\cL^s_X$. Using this observation, we deduce from Theorem~\ref{t.attractors} that 
every MS-foliation can be realized as the exit foliation of a transitive repelling hyperbolic plug.
\end{rema}

\section{Embedding hyperbolic plugs in Anosov flows}
\label{s.embedding}

The purpose of the present section is to prove Theorem~\ref{t.embedding} which states that every hyperbolic plug with filling MS-laminations can be embedded in a transitive Anosov flow. We shall need the folliowing lemma:

\begin{lemma}
\label{l.simple-transverse}
For every MS-foliation $\cF$ on a surface $S$, there exists a MS-foliation $\cG$ on $S$ which is transverse to $\cF$.
\end{lemma}

\begin{proof}
Choose a riemannian metric on $S$, and consider the orthogonal $\cF^\perp$ of $\cF$ for this metric. Obviously, $\cF^\perp$ is a foliation on $S$ which is transverse to $\cF$. In general, $\cF^\perp$ is not a MS-foliation. Nevertheless, a generic $C^1$-small perturbation of $\cF^\perp$ is a  MS-foliation which is still transverse to $\cF$.
\end{proof}

\begin{proof}[Proof of Theorem~\ref{t.embedding}]
Let us prove the first item. We consider a hyperbolic plug with filling MS-laminations $(U_0,X_0)$. We have to build a closed three-manifold $M$ and a (not necessarily transitive) Anosov vector field $X$ on $M$, such that there is a compact submanifold $U$ of $M$, such that $X$ is transverse to $\partial U$ and $(U,X_{|U})$ is topologically equivalent to $(U_0,X_0)$.

According to Lemma~\ref{l.pre-foliation}, the lamination $\cL^s_{X_0}$ can completed to a MS-foliation $\cF^s$. According to Lemma~\ref{l.simple-transverse}, we can find a MS-foliation $\cG$ on $\partial^{in} U_0$ which is transverse to $\cF^s$. Theorem~\ref{t.attractors} provides a transitive repelling hyperbolic plug $(U_R,X_R)$ and a homeomorphism $\phi:\partial^{out} U_1\to\partial^{in} U_0$ such that $\phi_*(\cL^u_{X_R})=\cG$. In particular, the foliation $\phi_*(\cL^u_{X_R})$ is strongly transverse to the  lamination $\cL^s_{X_0}$. We consider the manifold $U_1:=(U_R\sqcup U_0)/\phi$, endowed with the vector field $X_1$ induced by $X_R$ and $X_0$. According to Proposition~\ref{p.plug}, $(U_1,X_1)$ is a repelling hyperbolic plug.

Now we consider the exit foliation $\cL^u_{X_1}$. According to Lemma~\ref{l.simple-transverse}, we can find a MS-foliation $\cH$ on $\partial^{out} U_1=\partial^{out} U_0$ which is transverse to $\cL^u_{X_1}$. Theorem~\ref{t.attractors} provides a transitive attracting hyperbolic plug $(U_A,X_A)$ and a homeomorphism $\psi:\partial^{in} U_A\to\partial^{out} U_1$ such that $\psi_*(\cL^s_{X_A})=\cH$. In particular, the foliation $\psi_*(\cL^s_{X_A})$ is transverse to the foliation $\cL^u_{X_1}$. We consider the closed manifold $M:=(U_1\sqcup U_A)/\psi$, endowed with and the vector field $X$ induced by $X_1$ and $X_A$.  According to Proposition~\ref{p.plug}, $X$ is a (non-transitive) Anosov vector field.

The plug $(M,X)$ was constructed by gluing together the plugs $(U_R,X_R)$, $(U_0,X_0)$ and $(U_A,X_A)$. Therefore, $U_0$ can be regarded as a submanifold with boundary of $M$, and $X_0$ can be regarded as the restriction of $X$ to $U_0$. This completes the proof of the first item of Theorem~\ref{t.embedding}.

\bigskip

Now prove the second item. We assume that the maximal invariant set of $(U_0,X_0)$ contains neither attractors, nor repellers. We will modify use the  blow-up, excise and glue surgery to ``turn $X$ into a transitive vector field".

Recall that $(M,X)$ was obtained by gluing together the hyperbolic plugs $(U_R,X_R)$, $(U_0,X_0)$ and $(U_A,X_A)$. Therefore, $U_R,U_0,U_A$ can be regarded as compact submanifolds with boundary of $M$. We pick two periodic orbits $O_R,O_A$ of $X$, both with positive multipliers, contained respectively in $U_R$ and $U_A$. We make a repelling DA bifurcation on $O_R$, and an attracting DA bifurcation on $O_A$ (see subsection~\ref{ss.DA}). This gives rise to a new vector field $\overline X$ on $M$, which has a repelling periodic orbit $O_R\subset U_R$ and an attracting periodic orbit $O_A\subset U_A$. As in subsection~\ref{ss.DA}, we consider some open tubular neighborhoods $T_R$ and $T_A$ of $O_A$ and $O_R$ respectively, such that $\overline{X}$ is transverse to $\partial T_R$ and $\partial T_A$. We assume that $T_A$ and $T_R$ are contained respectively in  $U_A$ and $U_R$ (therefore $T_A$ and $T_R$ are disjoint from $U_0$). We excise these tubular neighborhoods from $M$, \emph{i.e.} we consider the compact manifold with boundary $U:=M\setminus (T_A\sqcup T_R)$. As explained in section~\ref{ss.blow-up-excise-glue}, $(U,\overline{X})$ is a hyperbolic plug, and we can find a strongly transverse gluing diffeomorphism $\chi:\partial^{out} U\to\partial^{int} U$. Notice moreover that the maximal invariant set of $(U,\overline{X})$ does not contain neither attractors, nor repellors, since we have made a DA bifurcation on a periodic orbit of the unique attractor of $X$ (turning this attractor into a saddle basic piece) and a DA bifurcation on a periodic orbit of the unique repellor of $X$ (turning this repellor into a saddle basic piece). We consider the closed manifold $M':=U/\chi$, and the vector field $X'$ induced by $\overline{X}$ on $M'$. According to Theorem~\ref{t.transitive} (up to perturbing $\overline{X}$ by topological equivalence, and $\chi$ by a  strongly transverse isotopy), the vector field $X'$ is Anosov. Observe that $U_0$ can be regarded as a compact submanifold with boundary of $M'$ (since the solid tori $T_A$ and $T_R$ were assumed to be disjoint from $U_0$), and $X_0$ can be regarded as the restriction of $X'$ to $U_0$ (indeed, $X'$ coincides with $\overline{X}$ on $\mathrm{int}(\overline{M})\supset U_0$, and $\overline{X}$ coincides with $X$ outside a small neighborhood of the orbits $O$ and $O'$). 

We are left to prove that $X'$ is transitive. For this purpose, we will use the criterion provided by Proposition~\ref{p.transitive}. Let $\Lambda_A$ (resp. $\Lambda_R$) be the maximal invariant set of the plug $(U_A,X_A)$ (resp. $(U_R,X_R)$). Let $\Lambda_1,\dots,\Lambda_n$ be the collection of the basic pieces of the hyperbolic plug $(U,X)$. Recall that $(M,X)$ was obtained by gluing the hyperbolic plugs $(U_R,X_R)$, $(U_0,X_0)$, $(U_A,X_A)$, without creating any new recurrent orbit. Therefore, the basic pieces of $(M,X)$ are $\Lambda_R,\Lambda_1,\dots,\Lambda_n,\Lambda_A$. For each $i=1,\dots,n$, $W^s_X(\Lambda_i)$ must intersect $W^u_X(\Lambda_R)$, because $\Lambda_R$ is the only repelling basic piece for $X$. Similarly, $W^u_X(\Lambda_i)$ must intersect $W^s_X(\Lambda_A)$.  Item~4 of Proposition~\ref{p.properties-DA} implies that the basic pieces of $\overline{X}_{|U}$ are in one-to-one correspondence with the basic pieces in $X$. We denote by $\overline{\Lambda}_R,\overline{\Lambda}_1,\dots,\overline{\Lambda}_n,\overline{\Lambda}_A$ the basic pieces of $\overline{X}_{|U}$ (using obvious notations). For $i=1,\dots,n$, $W^s_{\overline{X}}(\overline{\Lambda}_i)$  intersects $W^u_{\overline{X}}(\overline{\Lambda}_R)$, and $W^u_{\overline{X}}(\overline{\Lambda}_i)$  intersects $W^s_{\overline{X}}(\overline{\Lambda}_A)$. Moreover,  item~4 of Proposition~\ref{p.properties-DA} implies that $W^u_{\overline{X}}(\overline{\Lambda}_A)\cap\partial^{out} U$ is dense in $\cL^u_{\overline{X}}$, and $W^s_{\overline{X}}(\overline{\Lambda}_R)\cap\partial^{in} U$ is dense in $\cL^s_{\overline{X}}$. Hence, the image under $\chi_*$ of $W^u_{\overline{X}}(\overline{\Lambda}_A)\cap\partial^{out} U$ intersects $W^s_{\overline{X}}(\overline{\Lambda}_R)\cap\partial^{in} U$. As a consequence, the system $(U,\overline{X},\chi)$ is combinatorially transitive. According to Proposition~\ref{p.transitive}, this implies that the Anosov vector field $X'$ is transitive. This completes the proof of the second item of Theorem~\ref{t.embedding}.
\end{proof}

\section{A manifold supporting with $n$ transitive Anosov flows}

The purpose of this section is to prove Theorem~\ref{t.manyanosov}. We fix an integer $n\geq 1$. In subsection~\ref{ss.construction-several}, we construct of a manifold $M$ supporting $n$ transitive Anosov vector $Z_1,\dots,Z_n$. In subsection~\ref{ss.proof-several}, we prove that these vector fields are pairwise non topologically equivalent.

\subsection{Construction of a manifold $M$ supporting $n$ transitive Anosov flows}
\label{ss.construction-several}

\begin{lemma}\label{M0X0}
There exists a transitive hyperbolic plug with filling MS-laminations $(U,X)$ such that~:
\begin{enumerate}[topsep=1pt, itemsep=1pt, parsep=1pt]
\item $\mathrm{int}(U)$ is a hyperbolic manifold;
\item $\partial^{in} U$ is connected (\emph{i.e.} is a torus), the lamination $\cL^s_X$ has $2n+2$ compact leaves, all of them being coherently oriented;
\item for each connected component $T$ of $\partial^{out} U$, all the compact leaves of the lamination $\cL^u_X\cap T$ are coherently oriented.
\end{enumerate}
\end{lemma}

\begin{proof}
Consider a pseudo-Anosov diffeomorphism $f$ of a closed surface $\Sigma$, such that $f$ has at least two singularities, and such that one of the singularities of $f$ has exactly  $(2n+2)$ prongs. The existence of such a pseudo-Anosov diffeomorphism follows for example from Theorem~2 of~\cite{MaSm}. After possibly replacing $f$ by a power, we can assume that all the prongs of all the singularities of $f$ are fixed by $f$. We denote these singularities by $p_1,\dots,p_m$, where $p_1$ has $(2n+2)$ prongs.  Then we make a repelling DpA (derived from pseudo-Anosov) bifurcation at $p_1$, and some attracting DpA  bifurcations at $p_2,\dots,p_m$. This yields an axiom~A diffeomorphism $g$ of $\Sigma$, whose non-wandering set is composed of a non-trivial saddle basic piece, a repelling fixed point $p_1$, and some attracting fixed points $p_2,\dots,p_m$. Then we consider the suspension $(N,X)$ of this diffeomorphism: $N$ is a closed three-manifold, and $X$ is a non-singular axiom~A vector field on $N$ whose non-wandering set is made of a non-trivial saddle basic piece $\Lambda$, a repelling periodic orbit $\gamma_1$, and some periodic attracting orbits $\gamma_2,\dots,\gamma_m$. We set $U:=M\setminus (T_1\cup\dots\cup T_n)$, where $T_1,\dots,T_m$ are ``small" open tubular neighborhoods of the periodic orbits $\gamma_1,\gamma_2,\dots,\gamma_m$. More precisely, we choose $T_1,\dots,T_m$ to be included in the basins of the orbits $\gamma_1,\gamma_2,\dots,\gamma_m$ respectively, and such that their boundary is transverse to $X$ (just as in subsection~\ref{ss.DA}). By construction, $(U,X)$ is a plug,  $\partial^{in} U=\partial T_1$ (in particular $\partial^{in} U$ is connected as announced), and $\partial^{out} U=\partial T_2\cup\dots\cup \partial T_m$. Moreover, the same arguments as in the proof of Proposition~\ref{p.properties-DA} show that:
\begin{itemize}
\item the maximal invariant set of $(U,X)$ is a transitive and hyperbolic;
\item $\cL^s_X$ is a filling MS-lamination with $2n+2$ compact leaves, all of them being coherently oriented;
\item $\cL^u_X\cap \partial T_k$ is a filling MS-lamination with $s_k$ compact leaves,  where $s_k$ is the number of prongs of the singularity~$p_k$, all of them being coherently oriented.
\end{itemize}
Finally, a well-known theorem of Thurston (see~\cite{Th}) shows that the interior of $U$ is hyperbolic.
\end{proof}

We will use the hyperbolic plug $(V,Y)$ provided by Lemma~\ref{a.hyperbolic-plug-twisted-orbit}.

\begin{lemma}
\label{M1.prefoliation}
There exist a hyperbolic plug with filling MS-laminations $(W,Z)$ such that:
\begin{enumerate}[topsep=1pt, itemsep=1pt, parsep=1pt]
\item $(W,Z)$ is obtained by gluing the hyperbolic plugs $(U,X)$and $(V,Y)$, provided respectively by  Lemma~\ref{M0X0} and Lemma \ref{a.hyperbolic-plug-twisted-orbit}, along $\partial^{in} U$ and $\partial^{out} V$; in particular, $\partial^{in} W=\partial^{in} V$ is connected;
\item the lamination $\cL^s_Z$ has $2n+3$ compact leaves; exactly $2n+2$ of these $2n+3$ compact leaves are coherently oriented;
\item\label{i.exit-lamination} on each connected components of $\partial^{out} W$, all the compact leaves of the lamination $\cL^u_Z\cap T$ are coherently oriented.
\end{enumerate}
\end{lemma}

\begin{proof}
According to Lemma~\ref{M0X0}, $\partial^{in} U$ is a torus, and $\cL^s_X$ is a filling MS-lamination with $2n+2$ compact leaves. Let $\gamma_1,\dots,\gamma_{2n+2}$ be a geometrical enumeration of the compact leaves of $\cL^s_X$ (see definition~\ref{d.geom-enum}). Let $A_1$ be the connected component of $\partial^{in} U\setminus \bigcup_i \gamma_i$ bounded by the compact leaves $\gamma_1$ and $\gamma_2$. Recall that $\partial^{out} V$ is a torus, and that the lamination $\cL^u_Y$ consists in a single (compact) leaf $\gamma^u$. Then we choose a diffeomorphism $\psi:\partial^{out} V\to \partial^{in} U$ such that the closed leaf $\psi_*(\gamma^u)$ is contained in the interior of the annulus $A_1$, and transverse to the lamination $\cL^s_X$ (in particular, the compact leaf $\psi_*(\gamma^u)$ does intersect the lamination $\cL^s(U,X)$). We glue the plugs $(U,X)$ and $(V,Y)$ thanks to the diffeomorphism $\psi$. More precisely, we consider the manifold with boundary $W:=(U\sqcup V)/\psi$ and the vector field $Z$ induced by $X$ and $Y$ on $W$. Proposition~\ref{p.plug} asserts that $(W,Z)$ is a hyperbolic plug.

Let us describe the lamination $\cL^u_Z$. According to Proposition~\ref{p.gluedlaminations}, we have
$$\cL^u(W,Z)=\cL^u_X\sqcup (\Gamma_U)_*(\psi_*(\gamma^u)\setminus \cL^s_X),$$
where $\Gamma_U:\partial^{in} U\setminus \cL^s_X\to \partial^{out} U\setminus \cL^u_X$ is the crossing map associated to the plug $(U,X)$. Now observe that $\psi_*(\gamma^u)\setminus \cL^s_X$ does not contain any compact leaf (recall that the compact leaf $\psi_*(\gamma^u)$ does intersect the lamination $\cL^s_X$). Therefore the compact leaves of $\cL^u_Z$ are exactly the same as the compact leaves of $\cL^u_X$. In particular, for each connected component $T$ of $\partial^{out} W=\partial^{out} U$,
all the compact leaves of $\cL^u_Z\cap T$ are coherently oriented.

Now we describe the entrance lamination $\cL^s_Z$. The arguments are very similar to those of the proof of Proposition~\ref{l.adding-leaf-1}. We have
$$\cL^s_Z=\cL^s_Y\sqcup (\Gamma_V^{-1})_*(\psi^{-1}_*(\cL^s_X)\setminus \gamma^u),$$
where $\Gamma_V: \partial^{in} V\setminus \cL^s_Y\to\partial^{out} V\setminus \cL^u_Y$ is the crossing map associated to the plug $(V,Y)$.
The lamination $(\psi^{-1})_*(\cL^s_X)$ has $2n+2$ compact leaves, and we have chosen $\psi$ so that these leaves are disjoint from $\gamma^u$. Therefore $(\Gamma_V^{-1})_*((\psi^{-1})_*(\cL^s_X)\setminus \cL^u_Y)$ has $2n+2$ compact leaves. The lamination $\cL^s_Y$ consists in a single isolated compact leaf $\gamma^s$. This proves that the  lamination $\cL^s(W,Z)$ has exactly $2n+3$ compact leaves. Let us examine the contracting orientations of the leaves. The compact leaves of $\cL^s_X$ are coherently oriented. Moreover, these compact leaves are contained in the annulus $\partial^{in} U\setminus A_1$, and the map $\Gamma_V^{-1}\circ \psi^{-1}$ is well-defined on $\partial^{in} U\setminus A_1$ (since  $\psi_*(\gamma^u)$ is contained in $A_1$). It follows that the $2n+2$ compact leaves of $\cL^s_Z$ contained in  $(\Gamma_V^{-1})_*((\psi^{-1})_*(\cL^s_X))$ are coherently oriented. It remains to check that the orientation of the last compact leaf of $\cL^s_Z$ is not coherent
with the orientations of the $2n+2$  other compact leaves. Actually, we can use exactly the same trick as in the proof of Proposition~\ref{l.adding-leaf-1}: we consider an orientation-preserving homeomorphism $\tau:\partial^{out} V\to\partial^{out} V$ which reverses the orientation of the compact leaf $\gamma^u$. Exactly the same arguments as in the proof of Proposition~\ref{l.adding-leaf-1} show that either $\psi$ or $\psi\circ\tau$ lead to a plug $(W,Z)$ satisfying the desiered property.
\end{proof}

Now we consider two copies $W_-,W_+$ of the manifold with boundary $W$ provided by Lemma~\ref{M1.prefoliation}. We endow $W_+$ with the vector field $Z_+:=Z$, and we endow $W_-$ with the vector field $Z_-:=-Z$. There are some natural identifications:
\begin{itemize}
\item $\partial^{in} W_+\simeq \partial^{out} W_-\simeq \partial^{in} W$ and $\partial^{out} W_+\simeq \partial^{in} W_-\simeq \partial^{out} W$,
\item $\cL^s_{Z_+}\simeq\cL^u_{Z_-}\simeq\cL^s_{Z}$ and $\cL^u_{Z_+}\simeq\cL^s_{Z_-}\simeq\cL^u_{Z}$.
\end{itemize}
According to Lemma~\ref{M1.prefoliation}, there is one (and only one) compact leaf $c$ of the lamination $\cL^s_{Z}$, such that the contracting orientation of $c$ is incoherent with the contracting orientations of the other compact leaves of $\cL^s_{Z}$. We denote by $c_+$ (resp. $c_-$) the corresponding compact leaf of $\cL^s_{Z_+}$ (resp. $\cL^u_{Z_-}$).

\begin{lemma}\label{diff1}
There exists a diffeomorphism $\phi:\partial^{out} W_+\rightarrow \partial^{in} W_-$ such that:
\begin{enumerate}
  \item the filling MS-laminations $\phi_{\ast} (\cL^u_{Z_+})$ and $\cL^s_{Z_-}$ are strongly transverse;
  \item if we see $\phi$ as a self-homeomorphism of $\partial^{out} W$ (using the natural identifications of $\partial^{out} W_+$ and $\partial^{in} W_-$ with $\partial^{out} W$), then $\phi$ is isotopic to the identity.
\end{enumerate}
\end{lemma}

\begin{proof}
This follows immediately from Lemma~\ref{l.gluing-diffeo} and item~\ref{i.exit-lamination} of Lemma~\ref{M1.prefoliation}.
\end{proof}

\begin{lemma}
\label{phik}
 For every $k\in \{1,\dots,n\}$, there exists a diffeomorphism $\phi_k:\partial^{out} W_-\rightarrow \partial^{in} W_+$ with the following properties:
\begin{enumerate}
  \item $(\phi_k)_* (\cL^u_{Z_-})$ and $\cL^s_{Z_+}$ are strongly transverse;
  \item\label{i.position-special-leaf}  the compact leaves $(\phi_k)_\ast(c^-)$ and $c^+$ bound two open annuli\footnote{In order avoid unnecessary complications, we will not try to distinguish these two annuli (although this can be done using the orientations of the leaves $(\phi_k)_\ast(c^-)$ and $c^+$).} in the torus $\partial^{in} W_+$, which contain respectively $k$ and $2n+2-k$ compacts leaves of the lamination $\cL^s_{Z_+}$;
  \item  if we see $\phi$ as a self-homeomorphism of $\partial^{out} W$, then $\phi_k$ is isotopic to the identity.
  \end{enumerate}
\end{lemma}

\begin{rema}
\label{r.non-homeomorphic}
Item~\ref{i.position-special-leaf} of Lemma~\ref{phik} implies in particular that  $\cL^s_{Z_+}\cup (\phi_i)_*(\cL^u_{Z_-})$ is homeomorphic to $\cL^s_{Z_+}\cup (\phi_j)_\ast (\cL^u_{Z_-})$ only if $i=j$.
\end{rema}

\begin{proof}[Proof of Lemma~\ref{phik}]
See figure~\ref{f.phi}. Proposition~\ref{p.simplification} provides a diffeomorphism $\psi_-^0:\partial^{out} W_-\to\TT^2$ such that:
\begin{itemize}
\item the $2n+3$ compact leaves of the lamination $(\psi_-)_*(\cL^u_{Z_-})$ are the vertical circles $\{\frac{i}{2n+3}\}\times \SS^1$ for $i=0,\dots,2n+2$, the leaf $(\psi_-)_*(c^-)$ being the circle $\{\frac{1}{2n+3}\}\times\SS^1$
 \item in the open annulus $(\frac{i}{2n+3},\frac{i+1}{2n+3})\times \SS^1$, the leaves of the lamination $(\psi_-^0)_*(\cL^u_{Z_-})$ are graphs of  $C^1$ functions from $(\frac i{2n+3}n,\frac{i+1}{2n+3})$ to $\SS^1$; 
 \item the derivatives of these functions are positive on $(0,\frac{1}{2n+3})$, negative on $(\frac{1}{2n+3},\frac{2}{2n+3})$, 
  \item  for $i=2,\dots,2n+2$, the derivatives of these functions are positive on $(\frac{i}{2n+3},\frac{i}{2n+3}+\frac{1}{2})$ and negative on $(\frac{i}{2n+3}+\frac{1}{2},\frac{i+1}{2n+3})$.
 \end{itemize}
Similarly, one gets a a diffeomorphism $\psi_+:\partial^{in} W_+\to\TT^2$ so that the lamination $(\psi_+)_*(\cL^s_{Z_+})$ satisfies similar properties (with the leaf $(\psi_+)_*(c^+)$ insteaf of the leaf  $(\psi_-)_*(c^-)$). Now pick $\epsilon<<1$ and consider the diffeomorphism $\Theta_\epsilon:\TT^2\to\TT^2$ defined by 
$$\Theta_\epsilon(x,y) =  (\theta_\epsilon(x),y)$$
where $\theta_\epsilon:[0,1]\to [0,1]$ is the function which maps affinely $[0,\frac{2}{2n+3}]$ on $[0,\epsilon]$ and maps affinely $[\frac{2}{2n+3},1]$ on $[\epsilon,1]$ (in other words, $\Theta_k$ shrinks the annulus $(0,\frac{2}{2n+3})\times \SS^1$ to a very thin annulus).
For $k\in \{1,\dots,n\}$, also consider the diffeomorphism $\xi_k:\TT^2\to\TT^2$ defined by 
$$\xi_k(x,y)=\left(x+\frac{k+\frac{1}{2}}{2n+1},y\right).$$
One easily checks that the diffeomorphism $(\Theta_\epsilon\circ\psi_+)^{-1}\circ\xi_k\circ(\Theta_\epsilon\circ\psi_-)$ satisfies the desired properties provided that $\epsilon$ is small enough.
\end{proof}

\begin{figure}[ht]
\begin{center}
\includegraphics[totalheight=5.5cm]{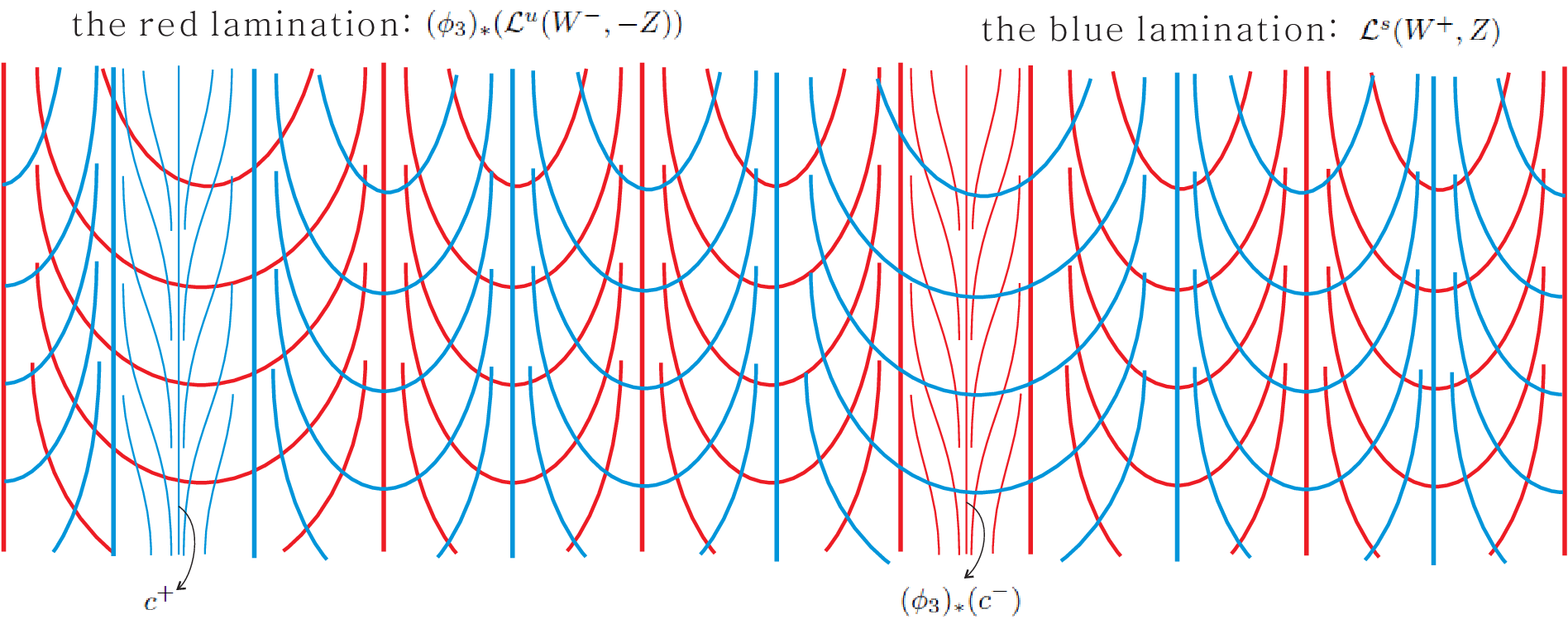}
\caption{\label{f.phi}The diffeomorphism $\phi_k$ in the case $n=2$ and $k=3$.}
\end{center}
\end{figure}

\begin{defi}[The vector fields $Z_1,\dots,Z_n$]
\label{d.several-anosov}
Consider the hyperbolic plug $(W_+,Z_+)\sqcup (W_-,Z_-)$. For $k=1\dots n$,  consider the diffeomorphism
$$\Phi_k: \partial^{out} W_+\sqcup \partial^{out} W_- \longrightarrow \partial^{in} W_+\sqcup \partial^{in} W_-$$
defined by $\Phi_k:=\phi\mbox{ on }\partial^{out} W_+\mbox{ and }\Phi_k:=\phi_k\mbox{ on }\partial^{out} W_-.$
According to Lemma~\ref{diff1} and~\ref{phik}, $\Phi_k$ is a strongly transverse gluing diffeomorphism. Now consider the closed manifold $M_k:=(W_+\sqcup W_-)/\Phi_k$
and the vector field $Z_k$ on $M_k$ induced by the vector fields $Z$ and $-Z$ (more precisely, $Z_k:=Z \mbox{ on }W_+ \mbox{ and }Z_k:=-Z\mbox{ on }W_-$).
According to Theorem~\ref{t.transitive}, up to modifying $Z$ by a topological equivalence and $\Phi_k$ by a strongly transverse isotopy, $Z_k$ is an Anosov vector field.
\end{defi}


Since the gluing map $\Phi_k$ is isotopic to the identity for every $k$, the manifolds $M_1,\dots,M_n$ are pairwise diffeomorphic. From now on, we identify the manifolds $M_1,\dots,M_n$ with a single manifold $M$, and see $Z_1,\dots,Z_n$ as vector fields on this manifold $M$.

The Anosov vector field $(M,Z_k)$ was obtained by gluing cyclically four hyperbolic plugs $(U_-,X_-)$, $(V_-,Y_-)$, $(V_+,Y_+)$ and $(U_+,X_+)$, where $U_+,U_-$ are two copies of $U$, $V_+,V_-$ are two copies of $V$, $X_+=X$, $X_-=-X$, $Y_+=Y$ and $Y_-=-Y$. See figure~\ref{f.gluing-four-plugs}.

\begin{figure}[ht]
\begin{center}
\includegraphics[totalheight=6cm]{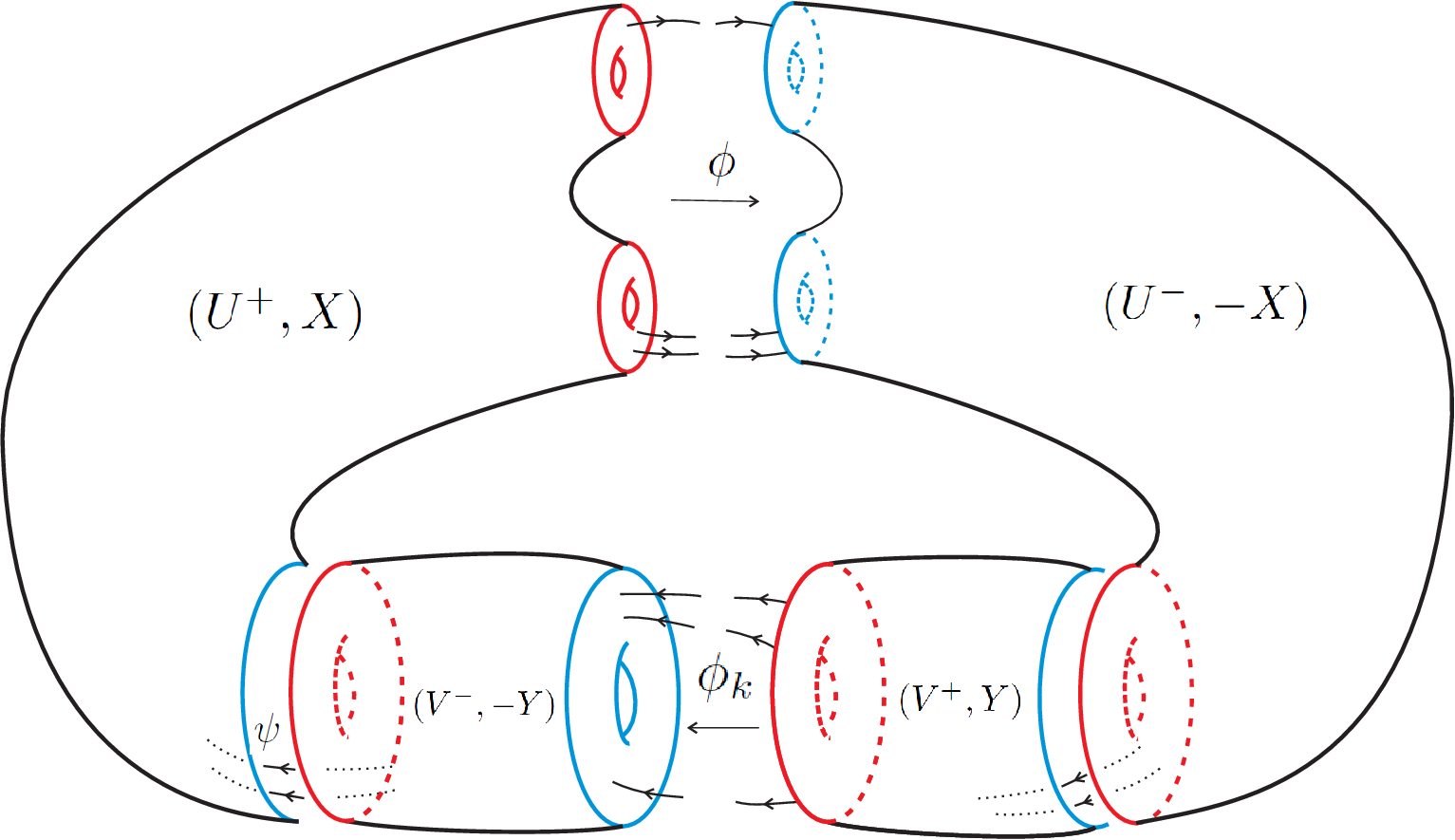}
\caption{\label{f.gluing-four-plugs}The hyperbolic plugs used to build the manifold $M$ and the Anosov vector field $Z_k$}
\end{center}
\end{figure}

\begin{prop}
\label{p.JSJ}
The JSJ decomposition of manifold $M$ has three pieces: two hyperbolic pieces $U_-$ and $U_+$, and one Seifert piece $S:=V_-\cup V_+$.
\end{prop}

\begin{proof}
As explained above, $M$ was obtained by gluing two copies $U_-,U_+$ of $U$ and two copies $V_-,V_+$ of $V$.  The interior of $U$ is hyperbolic; therefore, $U_-$ and $U_+$ must be hyperbolic pieces of JSJ decomposition of $M$. The manifold $V$ is a Seifert fiber bundle. During the construction of the manifold $M$, the two copies $V_-,V_+$ of $V$ were glued together using the map $\psi_k$. The map $\psi_k$ is isotopic to identity, and therefore maps the regular fibers of $V_-$ on the regular fibers of $V_+$ (up to free homotopy). Therefore $S:=V_-\cup V_+$ is a Seifert bundle, and corresponds to a single piece in the JSJ decomposition of $M$.
\end{proof}

\begin{rema}
The vector fields $Z_1,\dots,Z_n$ are pairwise homotopic through non-zero vector fields on $M$: this follows easily from the construction.
\end{rema}

\begin{prop}
For every $k\in\{1,\dots,n\}$, the Anosov vector field $Z_k$ is transitive.
\end{prop}

\begin{proof}
As explained above, $(M_k,Z_k)$ was obtained by gluing  the four hyperbolic plugs $(U_-,X_-)$, $(V_-,Y_-)$, $(V_+,Y_+)$ and $(U_+,X_+)$. These four plugs form a cycle as shown on figure~\ref{f.gluing-four-plugs}. The choice of the gluing maps ensures that the unstable manifold of the maximal invariant set of any of these four plugs intersects the stable manifold of the maximal invariant set of the next plug in the cycle. Moreover, each of the four hyperbolic plugs is transitive. Therefore, the graph associated to the gluing procedure has four vertices, and these four vertices belong to an oriented cycle; in particular, the gluing procedure is combinatorially transitive. By Proposition~\ref{p.transitive}, it follows that $Z_k$ is topologically transitive.
\end{proof}

\subsection{The vector fields $Z_1,\dots,Z_n$ are not topologically equivalent}
\label{ss.proof-several}

The strategy to prove the vector fields  $Z_1,\dots,Z_n$ are pairwise non topologically equivalent is the following. First, we prove that a topological equivalence between $Z_i$ and $Z_j$ must leave invariant the submanifolds $W_-$ and $W_+$. Then we use remarks~\ref{r.non-homeomorphic} to conclude that such a topological equivalence cannot exist, unless $i=j$.

We will use the following result, which was proved by Barbot (\cite[Th\'eor\`eme~A]{Ba3}), elaborating on some arguments of Brunella (\cite{Br}):

\begin{lemma}[Barbot]
\label{l.along-the-flow}
Let $Z$ be an Anosov vector field on a closed three-manifold $M$, and $T,T'$ be some tori embedded in $M$ and transverse to $Z$. If $T$ is homotopic to $T'$, then $T$ is isotopic to $T'$ along the orbits of $Z$.
\end{lemma}

The phrase ``$T$ is isotopic to $T'$ along the orbits of $Z$" means that there exists a continuous function $u:T\to \RR$ such that $x\mapsto Z^{u(x)}(x)$ maps $T$ on $T'$. This implies that there is a homeomorphism $g:M\to M$ preserving preserving each orbit of $Z$ and mapping $T$ on $T'$.

\begin{lemma}
\label{l.preserves-the-plugs}
Assume that the vector fields $Z_i$ and $Z_j$ are topologically equivalent. Then, one can choose the topological equivalence so that it preserves $W_-$ and $W_+$.
\end{lemma}

\begin{proof}
By assumption $Z_i$ and $Z_j$ are topological equivalent: there exists a homeomorphism $h:M\to M$ mapping the oriented orbits of $Z_i$ on the oriented orbits of $Z_j$.

Recall that the JSJ decomposition of $M$ comprises two hyperbolic pieces $U_-,U_+$ and one Seifert piece $S=V_-\cup V_+$ (see Proposition~\ref{p.JSJ}). The homeomorphism $h$ permutes the three JSJ pieces up to isotopy, mapping a hyperbolic piece on a hyperbolic piece, and the Seifert piece on a Seifert piece. Morover, $h$ cannot map (even up to isotopy) $U_-$ on $U_+$, because $h$ preserves the orientation of the orbits, and the orbits of $Z_i$ and $Z_j$ go from $U_+$ to $U_-$ (see figure~\ref{f.gluing-four-plugs}). Therefore, $h$ must leave invariant the three JSJ pieces $U^-,U^+,S$ up to isotopy.

Now, recall that the boundaries of $U_-, U_+,S$ are transverse to the vector fields $Z_i$ and $Z_j$ (see the proof of Proposition~\ref{p.JSJ}). Hence, using to Lemma~\ref{l.along-the-flow}, we can modify $h$ so that it leaves invariant $U_-,U_+,S$ in the set theoretic sense (not ``up to isotopy"). And since $\partial^{out} W_+=\partial^{in} W_-=\partial ^{out}U_+=\partial^{in} U_-$ (see figure~\ref{f.gluing-four-plugs}), it follows that $h$ preserves the surface $\partial^{out} W_+=\partial^{in} W_-$.

It remains to prove that $h$ also preserves the surface $\partial^{out} W_-=\partial^{in} W_+$. Recall that $\partial^{out} W_-=\partial^{in} W_+=\partial^{out} V_-=\partial^{in} V_+$ is an incompressible torus in the interior the Seifert piece $S=V_-\cup V_+$  (see again figure~\ref{f.gluing-four-plugs}). But the topology of $S$ is quite simple. Indeed, $V_-$ and $V_+$ are Seifert bundles over the projective plane minus two discs. It follows that, up to homotopy, there are only three incompressible tori in the Seifert piece $S$: the two connected components $\partial^{in} V_-,\partial^{out} V_+$ of the boundary of $S$, and the torus $\partial^{out} W_-=\partial^{in} W_+=\partial^{out} V_-=\partial^{in} V_+$. As a consequence, $h$ must preserve the torus $\partial^{out} W_-=\partial^{in} W_+$ up to isotopy. And using once again lemma~\ref{l.along-the-flow}, we can modify $h$ so that it leaves invariant $\partial^{out} W_-=\partial^{in} W_+$ in the set theoretic sense.

Now $h$ leaves invariant $\partial^{out} W_+=\partial^{in} W_-$ and $\partial^{out} W_-=\partial^{in} W_+$. Henceforth, it must leave invariant $W_-$~and~$W_+$.
\end{proof}

\begin{prop}
The vector fields $Z_i$ and $Z_j$ are not topologically equivalent, unless $i=j$.
\end{prop}

\begin{proof}
Consider two integers $i,j\in\{1,\dots,n\}$ and assume that the vector fields $Z_i$ and $Z_j$ are topologically equivalent. According to lemma~\ref{l.preserves-the-plugs}, there exists a homeomorphism $h:M\to M$ mapping the orbits of $Z_i$ of the oriented orbits of $Z_j$, and leaving invariant the two submanifolds $W^-$ and $W^+$. This homeomorphism $h$ maps the laminations $\cL^s_{Z^-}$ and $(\phi_i)_*(\cL^u_{Z_+})$ on the laminations  $\cL^s_{Z^-}$ and $(\phi_j)_*(\cL^u_{Z_+})$ respectively. According to remark~\ref{r.non-homeomorphic}, this is possible only if $i=j$.
\end{proof}

\section{An Anosov flow with infinitely many transverse tori}

The purpose of this last section is to prove Theorem~\ref{t.infinitely-many-tori}, \emph{i.e.} to build a transitive Anosov vector field $Z$ on a closed three-manifold $M$ such that there exists infinitely many pairwise non-isotopic tori embedded in $M$ which are transverse to $Z$.

We first consider the vector field $X_0$ on the torus $\TT^2$ defined as follows
$$X_0(x,y)=\sin(2\pi y)\frac{\partial}{\partial x}+\sin(2\pi y)\frac{\partial}{\partial y}.$$
One can easily check that the non-wandering set of $X_0$ consists in four hyperbolic singularities: a source~$\alpha:=(0,0)$, two saddles $\sigma_1:=(\frac{1}{2},0)$ and $\sigma_2:=(0,\frac{1}{2})$, and a sink $\omega:=(\frac{1}{2},\frac{1}{2})$. Moreover,  the invariant manifolds of $\sigma_1$ are disjoint from the invariant manifold of $\sigma_2$. See figure~\ref{f.gradient-like}.

\begin{figure}[ht]
\begin{center}
  \includegraphics[totalheight=5cm]{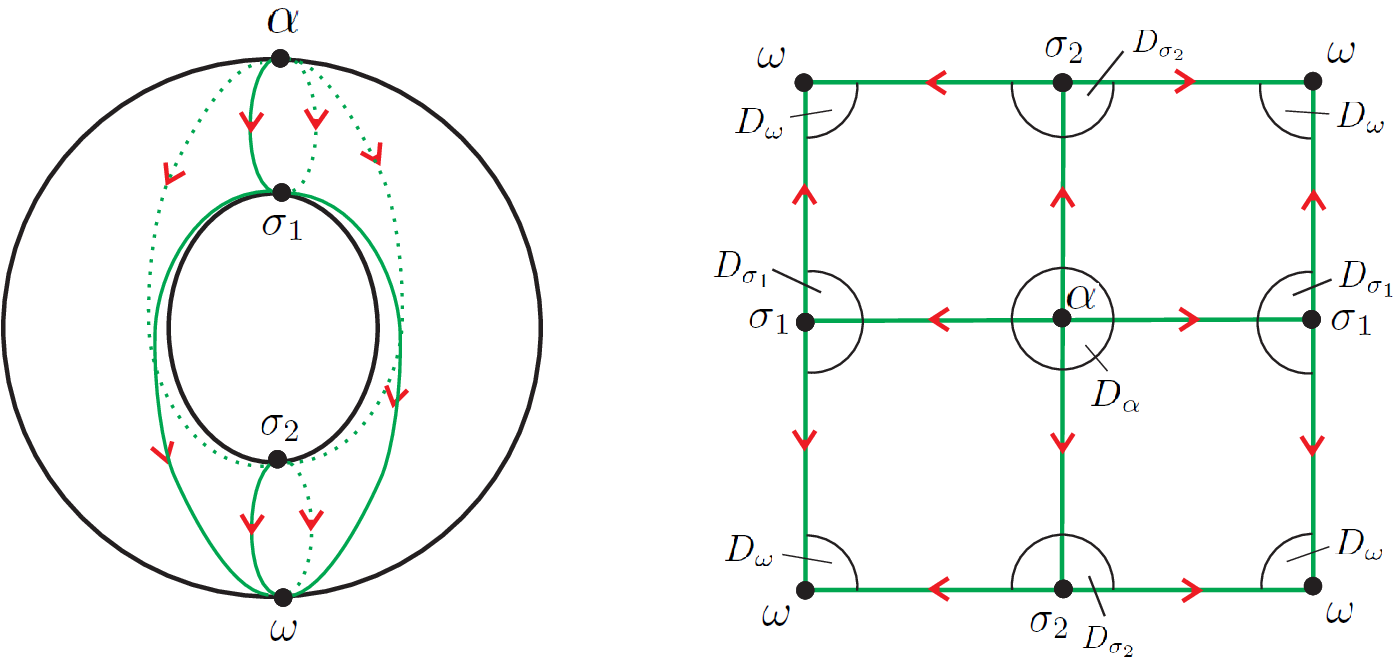}
  \caption{\label{f.gradient-like}The gradient-like vector field $X_0$}
  \end{center}
\end{figure}

Now, we consider some pairwise disjoint (small) open discs $D_\alpha,D_{\sigma_1},D_{\sigma_2},D_{\omega}$ centered at $\alpha,\sigma_1,\sigma_2,\omega$ respectively, such that the vector field $X_0$ is transverse to the boundaries of $D_\alpha$ and $D_{\omega}$. We consider a smooth function $\varphi:\TT^2\to\RR$ such that $\varphi>0$ on $D_{\sigma_1}$, $\varphi<0$ on $D_{\sigma_2}$, and $\varphi=0$ on $\TT^2\setminus (D_{\sigma_1}\cup D_{\sigma_2})$ (in particular, $\varphi=0$ on $D_\alpha\cup D_\omega$). Then we consider the vector field $X$ on $\TT^2\times \SS^1$ defined by
$$X(x,y,t)= X_0(x,y)+ \varphi(x,y)\frac\partial{\partial t}.$$
We consider the compact three-manifold with boundary $U:=(\TT^2\setminus (D_\alpha\cup D_\omega))\times \SS^1$.

\begin{lemma}
The pair $(U,X)$ is a hyperbolic plug with the following characteristics:
\begin{enumerate}[topsep=1pt, itemsep=1pt, parsep=1pt]
\item the maximal invariant set of $(U,X)$ consists in two saddle hyperbolic periodic orbits;
\item both the entrance boundary $\partial^{in} U$ (resp. the exit boundary $\partial^{out} U$) of $(U,X)$ is a torus;
\item the lamination $\cL^s_X$ (resp. $\cL^u_X$) consists in four closed leaves. These leaves are parallel essential curves in $\partial^{in} U$ (resp. $\partial^{out} U$). Moreover,  the dynamical orientations (see definition~\ref{d.flow-orientation}) of these leaves are ``alternating":  the dynamical orientations of two adjacent leaves of $\cL^s_X$ (resp. $\cL^u_X$) are always incoherent.
\end{enumerate}
\end{lemma}

\begin{proof}
The vector field $X_0$ is transverse to $\partial D_\alpha\cup \partial D_\omega$, and the function $\varphi$ vanishes on $D_\alpha\cup D_\omega$. Hence, the vector field $X$ is transverse to $\partial U=(\partial D_\alpha\times\SS^1)\sqcup (\partial D_\omega\times\SS^1)$. In other words, $(U,X)$ is a plug. Moreover, since $\alpha$ is a source and $\omega$ is a sink, the vector field $X_0$ is pointing outwards $D_\alpha$ and inwards $D_\omega$. It follows that $\partial^{in} U=\partial D_\alpha\times\SS^1$ and $\partial^{out} U=\partial D_\omega\times\SS^1$. In particular, both $\partial^{in} U$ and $\partial^{out} U$ are tori.

The definitions of $X_0$ and $S$ imply that maximal invariant set of $(S,X_0)$ is made of the two saddles $\sigma_1,\sigma_2$. It follows that the maximal invariant set of $(U,X)$ consists in the saddle hyperbolic periodic orbits $\sigma_1\times\SS^1$ and $\sigma_2\times\SS^1$. In particular, $(U,X)$ is a hyperbolic plug.

The intersection of $W^s(\sigma_1)\cup W^s(\sigma_2)$ with $\partial D_\alpha$ consists in four points. By definition of the vector field $X$, the lamination $\cL^s_X$ is just the product by $\SS^1$ of $(W^s(\sigma_1)\cup W^s(\sigma_2))\cap \partial D_\alpha$. This shows that $\cL^s_X$ consists in four closed leaves, which are parallel essential curves in $\partial^{in} U=\partial D_\alpha\times\SS^1$.

Let $\gamma_1,\gamma_2$ be two adjacent leaves in $\cL^s(U,X)$. Observe that the points of $W^s(\sigma_1)\cap \partial D_\alpha$ and $W^s(\sigma_2)\cap \partial D_\alpha$ are alternating with respect to the cyclic order of $\partial D_\alpha$. Therefore, up to exchanging the names, $\gamma_1$ and $\gamma_2$ belong respectively to $W^s(\sigma_1\times\SS^1)$ and $W^s(\sigma_2\times\SS^1)$. And since the function $\varphi$ is positive on $D_{\sigma_1}$ and negative on $D_{\sigma_2}$, it follows that
the dynamical orientations of $\gamma_1,\gamma_2$ are incoherent.
\end{proof}

Now we consider two copies $(U_1,X_1)$ and $(U_2,X_2)$ of the plug $(U,X)$. We choose a diffeomorphism $\psi:\partial^{out} U_1\to \partial^{in} U_2$ so that each of the four compact leaves of $\psi_*(\cL^u_{X_1})$ intersects transversally each of the four leaves of $\cL^s_{X_2}$ at exactly one point (see figure~\ref{f.gluings}, left). We consider manifold with boundary $V:=(U_1\sqcup U_2)/\psi$. We denote by $Y$ the vector field on $V$ induced by $X_1$ and $X_2$. According to proposition~\ref{p.plug}, $(V,Y)$ is a hyperbolic plug. By construction, $\partial^{in} V=\partial^{in} U_1$ and $\partial^{out} V=\partial^{out} U_2$. In particular, both $\partial^{in} V$ and $\partial^{out} V$ are tori.

\begin{figure}[ht]
\begin{center}
 \includegraphics[totalheight=5cm]{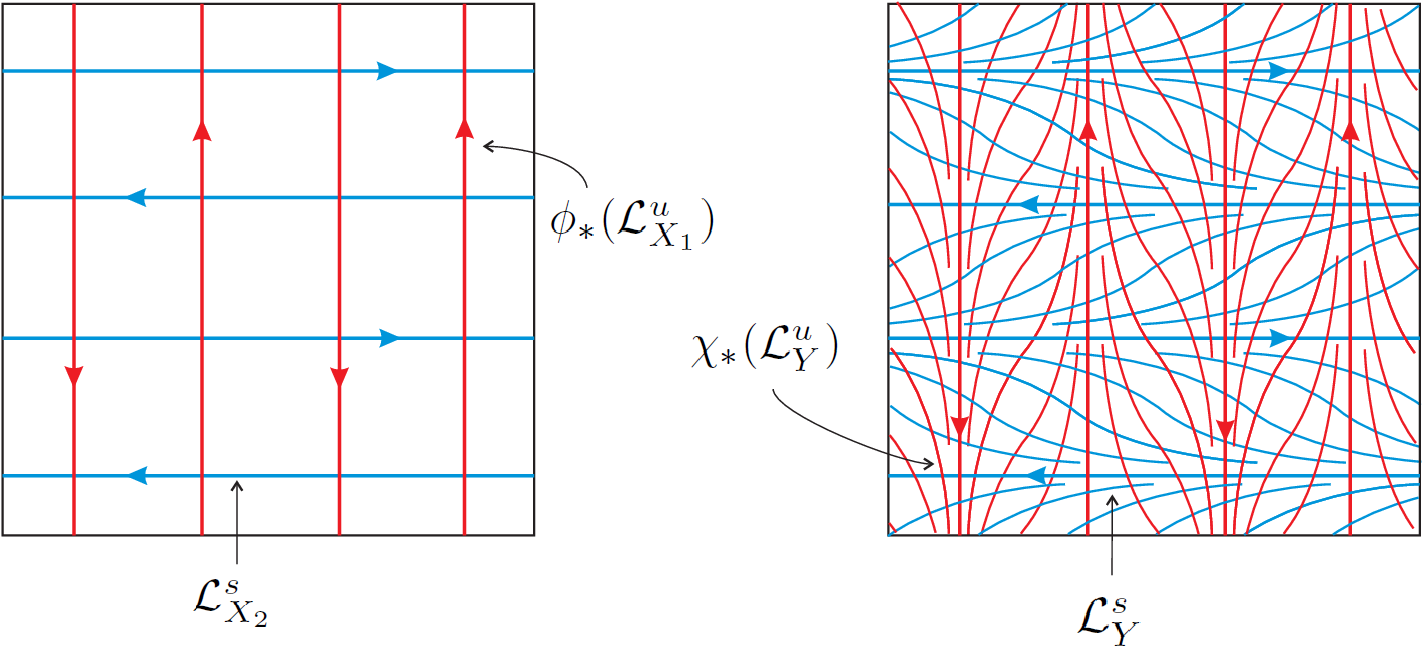}
\caption{\label{f.gluings} The gluing maps $\psi$ (on the left) and $\chi$ (on the right)}
\end{center}
\end{figure}

\begin{rema}
\label{r.arc-crossing-annulus}
Each connected component $A^s$ of $\partial^{in} U_2\setminus \cL^s_{X_2}$ is an annulus bounded by two (compact) leaves of $\cL^s_{X_2}$. The assumptions on the gluing map $\psi$ imply that $\psi_*(\cL^u_{X_2})\cap A^s$ consists in four open arcs, each of which is ``crossing" the annulus $A^s$ (\emph{i.e} going from one end of $A^s$ to the other).
\end{rema}

\begin{lemma}
\label{l.properties-lamination}
$\cL^u_Y$ (resp. $\cL^s_Y$) is a filling MS-lamination.  It has four compact leaves. These leaves have ``alternating contracting orientations": the contracting orientations of two adjacent leaves are always incoherent.
\end{lemma}

\begin{proof}
As usual, we use Proposition~\ref{p.gluedlaminations} to write $\cL^s_Y$ as a disjoint union:
$$\cL^u_Y=\cL^u_{X_2}\sqcup (\Gamma_2)_*(\phi_*(\cL^u_{X_1})\setminus \cL^s_{X_2}$$
where $\Gamma_2:\partial^{in} U_2\setminus\cL^s_{X_2}\to \partial^{out} U_2\setminus\cL^u_{X_2}$ is the crossing map of the plug $(U_1,X_1)$.

Each of the four leaves of $\phi_*(\cL^u_{X_1})$ intersects $\cL^s_{X_2}$. Therefore, $\Gamma_*(\phi_*(\cL^u_{X_1})\setminus \cL^s_{X_2})$ does not contain any compact leaf. As a further consequence, the compact leaves of the lamination $\cL^u_Y$ are exactly those of the lamination $\cL^u_{X_2}$. Hence, the lamination $\cL^u_Y$  has four compact leaves (which are parallel essential curves in the torus $\partial^{out} V=\partial^{out} U_2$), and the dynamical orientation of two adjacent compact leaves are incoherent. By Proposition~\ref{p.flow-vs-contracting}, this is equivalent to the analoguous statement with the contracting orientation instead of the dynamical orientations.

We are left to prove that $\cL^u_Y$ is a filling MS-lamination. We already know that this is a MS-lamination thanks to Propositon~\ref{p.hyperbolicplug}. So we are left to prove that every connected component of $\partial^{out} V\setminus \cL^u_Y$ is a strip in the sense of definition~\ref{d.strip-exceptionnal}. For this purpose, we consider a connected component $A^u$ of $\partial^{out} U_2\setminus\cL^u_{X_2})$~; this is  an open annulus bounded by two compact leaves $\gamma_1^u,\gamma_2^u$ of $\cL^u_{X_2}$. The set $A^s:=\Gamma^{-1}(A^u)$ is a connected component of $\partial^{in} U_2\setminus\cL^s_{X_2}$. The leaves of $\cL^u_Y$ contained in the annulus $A^u$ are exactly the images under $\Gamma_*$ of the connected components of $\psi_*(\cL^u_{X_1}))\cap A^s$. Together with remark~\ref{r.arc-crossing-annulus}, this implies that there are exactly four leaves of $\cL^u_Y$ in the annulus $A^u$, and that each of this four leaves is ``crossing" the annulus $A^u$, \emph{i.e.} is accumulating on both $\gamma^u_1$ and $\gamma^u_2$. As a further consequence, every connected component of $A^u\setminus \cL^u_Y$ is a strip bounded by two compact leaves of $\cL^u_Y$ which are asymptotic at both ends. In other words, $\cL^u_Y$ is a filling MS-lamination.
\end{proof}

\begin{lemma}
\label{l.super-transverse-gluing-map}
There exists a diffeomorphism $\chi:\partial^{out} V\to\partial^{in} V$ such that:
\begin{itemize}
\item the laminations $\chi_*(\cL^u_Y)$ and $\cL^s_Y$ are strongly transverse,
\item every leaf of $\chi_*(\cL^u_Y)$ intersects every leaf of $\cL^s_Y$ (see figure~\ref{f.gluings}, right).
\end{itemize}
\end{lemma}

\begin{proof}
 See figure~\ref{f.gluings}. Proposition~\ref{p.simplification} provides a diffeomorphism $\psi^{out}:\partial^{out} V\to\TT^2$ such that:
\begin{itemize}
\item the compact leaves of the lamination $(\psi^{out})_*(\cL^u_Y)$ are the vertical circles $\{\frac{i}{4}\}\times \SS^1$ for $i=0,\dots,3$;
 \item in the open annulus $(\frac{i}{4},\frac{i+1}{4})\times \SS^1$, the leaves of the lamination $(\psi^{out})_*(\cL^u_Y)$ are graphs of $C^1$ functions from $(\frac{i}{4},\frac{i+1}{4})$ to $\SS^1$; 
 \item the derivatives of these functions are positive for $i=0$ and $2$, and negative for $i=1$ and $3$. Given any constant $A$, an elementary  modification of the proof fo Proposition~\ref{p.simplification} allows to assume that
 the derivatives of these functions are larger than $A$ for $i=0$ and $2$, and smaller than $-A$ for $i=1$ and $3$. 
 \end{itemize}
There is a diffeomorphism $\psi^{in}:\partial^{in} V\to\TT^2$ such that  the lamination $(\psi^{in})_*(\cL^s_Y)$ satisfies analoguous properties.  Now, let $\chi_0:\TT^2\to\TT^2$ be the orientation-preserving diffeomorphism defined by $\chi_0(x,y)=(-y,x)$. A straightforward computation shows that the diffeomorphism $\chi:\partial^{out} V\to\partial^{in} V$ defined by $\chi:=(\psi^{out})^{-1}\circ\chi_0\circ\psi^{in}$ satisfies the desired properties.
\end{proof}

We consider the closed manifold $M:=V/\chi$. We denote by $Z$ the vector field induced by $Y$ on $M$. The first item of Lemma~\ref{l.super-transverse-gluing-map} and Theorem~\ref{t.transitive} imply that the vector field $Z$ is Anosov (up to perturbing $Y$ within its topological equivalence class and modifying $\chi$ by a strongly transverse isotopy). The second item of Lemma~\ref{l.super-transverse-gluing-map} and Proposition~\ref{p.transitive} imply that the vector field $Z$ is topologically transitive.

\begin{figure}[htp]
\begin{center}
 \includegraphics[totalheight=5cm]{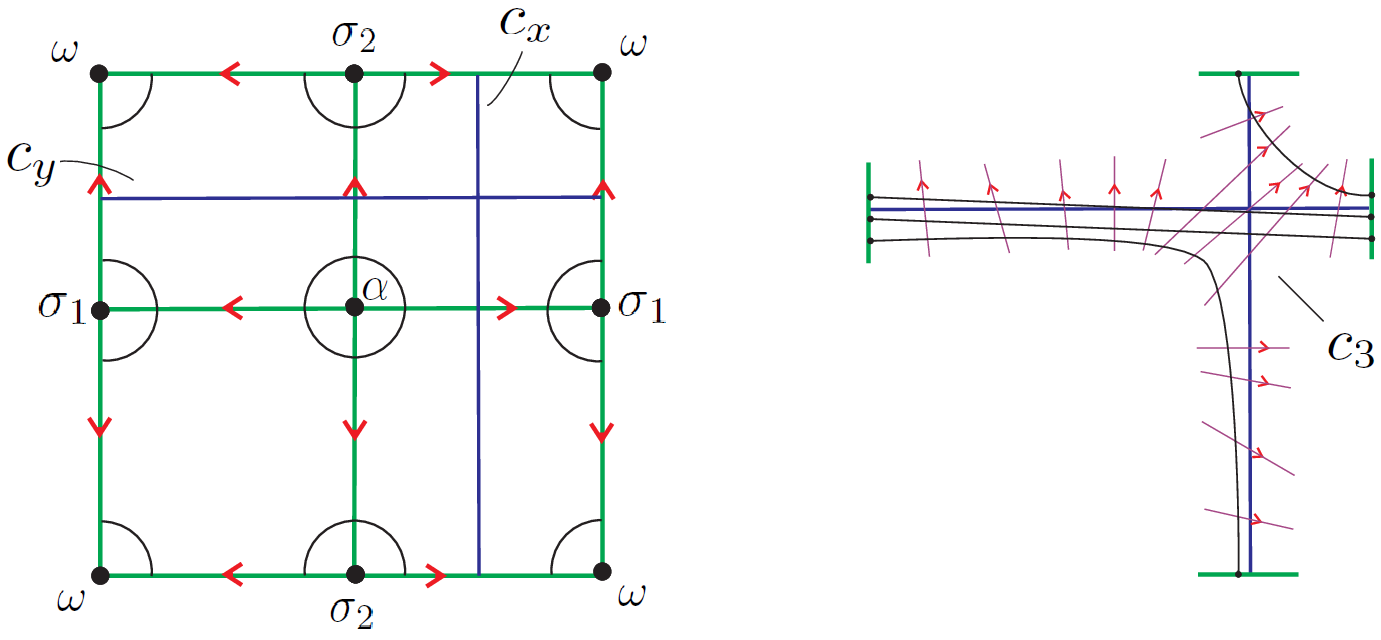}
  \caption{\label{f.transverse-curves}The curves $c_x$, $c_y$ and $c_{3}$}
  \end{center}
\end{figure}

\begin{prop}
\label{p.infinitely-many-tori}
There exists infinitely many pairwise non-isotopic tori embedded in $M$ which are transverse to the vector field $Z$.
\end{prop}

\begin{proof}
Let  $c_x$ and $c_y$ be the closed curves on $\TT^2$ defined respectively by the equations $x=\frac{1}{4}$ and $y=\frac{1}{4}$. We endow $c_x$ and $c_y$ with the orientations defined by the vector fields $\frac{\partial}{\partial y}$ and $\frac{\partial}{\partial x}$ respectively. One can easily check that the vector field $X_0$ is transverse to $c_x$ and $c_y$.

Let $q\in\NN-\{0\}$. Using classical and elementary desingularization process, one can easily find an oriented simple closed curve $c_{q}$ on $\TT^2$, freely homotopic to $c_x+q.c_y$, transverse to the vector field $X_0$, and disjoint from the discs $D_\alpha,D_{\sigma_1},D_{\sigma_2},D_{\omega}$ (see figure~\ref{f.transverse-curves}). The torus $T_{q}:=c_{q}\times\SS^1$ is embedded in $U=(\TT^2\setminus (D_\alpha\cup D_\beta))\times\SS^1$ and transverse to the vector field $X$ (because $c_q$ is transverse to $X_0$ and since $X(x,y,t)=X_0(x,y)$ for $(x,y)\in \TT^2\setminus (D_{\sigma_1},D_{\sigma_2})$). Now recall that $(M,Z)$ has been obtained by gluing together two copies of the plug $(U,X)$. Therefore the torus~$T_q$ can be seen as a torus embedded in $M$ transverse to the vector field $Z$.


Now consider two different positive integers $q,q'$. Fix any $\theta_0\in\SS^1$ and consider the simple closed curve $\hat c_q: = c_q \times \{\theta_0\} \subset T_q$. Obviously the algebraic intersection number of the curve $\hat c_q$ and the torus $T_{q'}$ is equal to $\pm |q-q'|$, which is nonzero. An easy cohomological argument shows that $T_q$ is not isotopic to $T_{q'}$ in $M$.

So we have found infinitely many pairwise non-isotopic tori embedded in $M$ and transverse to $Z$. The proof of Proposition~\ref{p.infinitely-many-tori} and Theorem~\ref{t.infinitely-many-tori} is complete.
\end{proof}

\begin{rema}
The construction above does not seem to be optimal. Indeed, it should be possible to find a gluing map $\theta:\partial^{out} U\to\partial^{in} U$ such that the vector field $Z_\theta$ induced by $X$ on the closed manifold $U/\theta$ is Anosov. Nevertheless, the existence of such a gluing map does not follow from Theorem~\ref{t.transitive}, since the entrance/exit laminations of the plug $(U,X)$ are not filling MS-laminations.
\end{rema}

\begin{rema}
The manifold $M$ constructed above is a graph manifold (it was obtained by gluing together two copies of $S\times \SS^1$ where $S$ is the torus minus two discs). Nevertheless,
the construction can easily be modified in order to get a manifold $M$ which has hyperbolic pieces in its JSJ decomposition.
\end{rema}

\vskip 1cm
\noindent Fran\c cois B\'eguin

\noindent {\small LAGA\\UMR 7539 du CNRS}

\noindent{\small Universit\'e Paris 13, 93430 Villetaneuse, FRANCE}

\noindent{\footnotesize{E-mail: beguin@math.univ-paris13.fr}}
\vskip 2mm

\noindent Christian Bonatti,

\noindent {\small Institut de Math\'ematiques de Bourgogne\\
UMR 5584 du CNRS}

\noindent {\small Universit\'e de Bourgogne, 21004 Dijon, FRANCE}

\noindent {\footnotesize{E-mail : bonatti@u-bourgogne.fr}}

\vskip 2mm

\noindent Bin Yu

\noindent {\small Department of Mathematics}

\noindent{\small Tongji University, Shanghai 2000
92, CHINA}

\noindent{\footnotesize{E-mail: binyu1980@gmail.com }}

\end{document}